\documentclass[preprint,12pt]{elsarticle}

\usepackage{amssymb,amsmath,color,comment,tikz,tikz-3dplot,pgfplots,mathtools}
\usetikzlibrary{3d,shadings}
\usepgfplotslibrary{fillbetween}

\newtheorem{ntheorem}{\bf Theorem}[section]
\newtheorem{theorem}[ntheorem]{\bf Theorem}
\newtheorem{lemma}[ntheorem]{\bf Lemma}

\newtheorem{corollary}[ntheorem]{\bf Corollary}

\newtheorem{proposition}[ntheorem]{\bf Proposition}
\newtheorem{assumption}[ntheorem]{\bf Assumption}
\newtheorem{example}[ntheorem]{\bf Example}

\newcommand{\eoe}
           {\hspace*{\fill}{$\vcenter{\hrule height1pt 
                     \hbox{\vrule width1pt height3pt 
            \kern3pt \vrule width1pt} \hrule height1pt}$} }

\newcommand{\eop}
           {\hspace*{\fill}{$\vcenter{\hrule height1pt 
                     \hbox{\vrule width1pt height5pt 
            \kern5pt \vrule width1pt} \hrule height1pt}$} }

\DeclareFontFamily{U}{mathx}{\hyphenchar\font45}
\DeclareFontShape{U}{mathx}{m}{n}{
  <5> <6> <7> <8> <9> <10>
  <10.95> <12> <14.4> <17.28> <20.74> <24.88>
  mathx10
}{}
\DeclareSymbolFont{mathx}{U}{mathx}{m}{n}
\DeclareMathAccent{\widecheck}{0}{mathx}{"71}

\renewcommand{\Re}{\mathbb{R}}
\newcommand{\ZZ}{\mathbb{Z}}
\newcommand{\dom}{\mbox{\rm dom }}
\newcommand{\Ball}{\mathbb{B}}

\newcommand{\R}{\mathcal{R}}

\usepackage{hyperref} 
\definecolor{mygreen}{RGB}{0,128,0}

\newcommand{\ooomega}{\mbox{\Large{$\boldsymbol\omega$}}}

\renewcommand{\L}{\mathcal{L}}
\renewcommand{\S}{\mathcal{S}}

\newcommand{\nats}{{\mathbb N}}
\newcommand{\reals}{{\mathbb R}}
\newcommand{\realsplus}{\reals_{\geq 0}}
\newcommand{\gph}{\mathop{\rm gph}\nolimits}

\newcommand{\sol}{\phi}

\newcommand{\length}{\mathop{\rm length}}
\newcommand{\wisol}{\widetilde{\sol}}

\newcommand{\sola}{\psi}
\newcommand{\Sol}{{\mathcal S}^{\mathcal H}}
\newcommand{\Sola}{\Psi}

\newcommand{\soldot}{\dot{\sol}}
\newcommand{\Suzie}{\Theta}
\newcommand{\SuzieS}{\left. \Suzie \right|_{S}}

\newcommand{\pcTsi}{\sola_{\chi,T}^{(s,i)}}
\newcommand{\pcTsidot}{\dot{\sola}_{\chi,T}^{(s,i)}}

\newcommand{\Omjk}{\Omega_{\overline{\jmath}_k\leq j}}
\newcommand{\Omkj}{\Omega_{\overline{k}_j\leq k}}
\newcommand{\Omjkinf}{\Omega_{\overline{\jmath}_k<\infty}}
\newcommand{\Omjke}{\Omega_{\overline{\jmath}_k = j}}
\newcommand{\Omjkei}{\Omega_{\overline{\jmath}_k = i}}
\newcommand{\Omkje}{\Omega_{\overline{k}_j = k}}
\newcommand{\Omjkminusone}{\Omega_{\overline{\jmath}_k \leq j-1}}
\newcommand{\Omkjei}{\Omega_{\overline{k}_j = i}}
\newcommand{\Omkjminusone}{\Omega_{\overline{k}_j \leq k-1}}

\renewcommand{\H}{\mathcal{H}}

\newcommand{\natz}{\mathbb{N}_0}
\newcommand{\ball}{{\mathbb B}}
\newcommand{\B}{{\mathcal B}}

\newcommand{\ve}{\varepsilon}

\newcommand{\rge}{\mathop{\rm rge}}

\renewcommand{\eop}
           {\hspace*{\fill}{$\vcenter{\hrule height1pt 
                     \hbox{\vrule width1pt height5pt 
            \kern5pt \vrule width1pt} \hrule height1pt}$} }

\newenvironment{proof}
{\par\noindent\textbf{Proof.}}{\eop\smallskip\vskip 3 pt}

\journal{arXiv}

\begin{document}

\begin{frontmatter}

\title{
Chain transitivity 
in generalized hybrid dynamics
with application to simulation
and stochastic approximation of hybrid systems}

\author[label1]{Rafal K. Goebel}
\affiliation[label1]{organization={Loyola University Chicago},
            city={Chicago},
            postcode={60660},
            state={IL},
            country={USA}}

\author[label2]{Andrew R. Teel}
\affiliation[label2]{organization={University of California},
            city={Santa Barbara},
            postcode={93106},
            state={CA},
            country={USA}}

\begin{abstract}
Asymptotic properties of discrete, stochastic approximations to
hybrid systems, modeled as hybrid inclusions, are studied. 
First, the internal chain transitivity of omega-limits of solutions
is concluded, along with other properties related to chain recurrence
and transitivity. A concept of an asymptotic solution is proposed 
to describe any mapping that, asymptotically, resembles a solution, and
for which the chain transitivity properties also turn out to hold. 
The mentioned developments are carried out in an abstract setting of
a generalized hybrid system defined by a set of hybrid curves, each
defined on a hybrid time domain, and possibly consisting of all solutions
to a given hybrid inclusion. 
Then, more specific kinds of perturbed solutions to a hybrid inclusion
are proposed and shown to include the solutions of a discretization and
of a stochastic approximation to the hybrid inclusion. 
Consequently, appropriate discretizations and stochastic approximations
of a hybrid inclusion produce mappings whose omega limits are 
internally chain transitive for the underlying hybrid inclusion. 
\end{abstract}


\end{frontmatter}

\section{Introduction}

This paper aims
to characterize the omega-limit set of each bounded, complete sample path of a 
stochastic, discrete signal
that is an asymptotic simulation of a hybrid system.
The findings are similar to
those in the literature on stochastic approximation
of differential equations and inclusions; see,
for example, \cite{Benaimetal2005} and \cite{NguyenYin23}, and the related
\cite{BorkarShah23}.
In particular, conditions are given under which
the omega-limit set of each sample path of the stochastic
signal
is internally
chain transitive for the underlying hybrid system.
The conditions involve the step sizes of the simulation
and conditional means and variances of the stochastic
signal.
Compared to the results on stochastic approximation
of hybrid systems in \cite{TeelSanfeliceGoebel25ARC},
no explicit simulator model is prescribed,
randomness is allowed in both the ``flows'' and the jumps, no global
attractor is required, and the conditions on the step sizes and the conditional means and variances
are relaxed.

The notions of chain recurrence and chain transitivity date back to 
\cite{Conley78}, where they were presented in the setting of flows. 
Chain recurrence lays the foundation for Conley's decomposition, which has been 
called ``the fundamental theorem of dynamical systems'' \cite{Norton95CMUC}.
The decomposition, among other things, identifies the subset of an invariant set 
--- its chain recurrent part --- to which all trajectories converge. 
The notion of chain recurrence and the decomposition have been extended 
to multivalued dynamics in continuous time \cite{BronsteinKopanskii88NA} 
and discrete time \cite{Akin93}, 
to an abstract setting of multivalued semiflows, with applications to 
infinite-dimensional dynamics, \cite{KapustyanKasyanovValero20CPAA} 
to hybrid dynamics 
\cite{KvalheimGustafsonKoditschek2021SIAD}, \cite{Goebel23SCL}
and to other settings as well.
Internal chain transitivity, considered early by \cite{BenaimHirsch95DCDS}, 
\cite{HirschSmithZhao01JDDE}, further narrows down the sets 
to which a trajectory of a dynamical system can converge. 
The key to this  narrowing is that the omega-limit of a trajectory
of regular-enough dynamical system is internally chain transitive. 
Converse results to this exist too; see \cite{MeddaughRaines12FM} 
and some references therein. An early and related result on internal 
chain recurrence is \cite{Robinson77RMJM}.
The concepts of internal chain recurrence and transitivity have not 
been studied in the case of general hybrid dynamics, and this inspires 
Section \ref{section: recurrence and transitivity in generalized hybrid dynamics} 
below. There, the concepts are defined and studied in the setting of 
an abstract hybrid system,  defined as a set of ``hybrid curves'' and subject 
to some structural properties.  
For an overview of abstract approaches to continuous-time dynamics, multivalued
or without uniqueness, see \cite{CaraballoMarin-RubioRobinson03SVA}, 
\cite[Chapter I]{ZgurovskyKasyanovKapustyanValeroZadoianchukIII}, 
or \cite{Suda24JDE}, and the references therein.

Chain transitivity properties are known to hold not just for trajectories 
of a regular-enough dynamical system, but may hold for functions 
or more general mappings that asymptotically resemble the said trajectories. 
This applies, for example, to ``asymptotic pseudotrajectories'' 
in \cite{BenaimHirsch96JDDE}, \cite{FaureRoth12SD} 
or solutions to asymptotically autonomous dynamics as in 
\cite{MischaikowSmithThieme95TAMS}. 
The properties of asymptotic pseudotrajectories were used
in \cite{Benaimetal2005}
to establish
asymptotic properties of non-stochastic
and stochastic simulations of a differential
inclusion, which is similar to the motivation
for the results presented here for hybrid 
systems.
The generalizations
in \cite{BenaimHirsch96JDDE}, \cite{FaureRoth12SD} 
and
\cite{MischaikowSmithThieme95TAMS}
inspire Section \ref{section: asymptotic curves and their asymptotic properties}.
In that section, a broad concept of an asymptotic hybrid curve 
is proposed and studied. 
Basic tools of set-valued analysis facilitate making precise 
the request that an asymptotic hybrid curve behave, asymptotically, 
as a hybrid curve from a desired set of hybrid curves, the latter perhaps
being the set of all solutions to a hybrid inclusion. 

Hybrid systems modeled by hybrid inclusions, as in \cite{Goebel12a}, fit 
in the setting of sections 
\ref{section: recurrence and transitivity in generalized hybrid dynamics} 
and 
\ref{section: asymptotic curves and their asymptotic properties}. 
Section \ref{section: back to earth} recalls the basics of hybrid inclusions,
translates some of the results in Sections
\ref{section: recurrence and transitivity in generalized hybrid dynamics} 
and 
\ref{section: asymptotic curves and their asymptotic properties} 
to this setting --- for example, one
result says that the omega-limit of an asymptotic solution to hybrid inclusion
is internally chain transitive for that hybrid inclusion ---
and introduces new concepts of perturbed solutions to a hybrid inclusion. 

The main application of the results in
Sections \ref{section: recurrence and transitivity in generalized hybrid dynamics} and \ref{section: asymptotic curves and their asymptotic properties} is to
help with the characterization of
the asymptotic behavior of what
are called ``asymptotic simulations'' of the hybrid systems
discussed in Section \ref{section:hybrid-systems};
asymptotic simulations are introduced in
Section \ref{section:asymptotic simulations of a hybrid system}
where their asymptotic properties are also established. 
That section also includes a connection to a recent result on asymptotic simulations
of hybrid systems given in \cite{GoebelTeel25SCLsubmission}.
Section \ref{section:stochastic-asymptotic-simulations} considers
the situation where the asymptotic simulations
are stochastic. It provides assumptions
on conditional means and variances of the stochastic hybrid
signal in order to guarantee that the stochastic signal
is an asymptotic simulation of a given hybrid system. This application provides new
results on stochastic approximations
that generalize results
for differential inclusions in
\cite{Benaimetal2005} and \cite{NguyenYin23}, and for
hybrid systems
in \cite{TeelSanfeliceGoebel25ARC}.
Section \ref{section:an example} provides an
illustrative
example that is related to simulated annealing, as studied
in \cite{Kushner87} for example.

\section{Recurrence and transitivity in generalized hybrid dynamics}
\label{section: recurrence and transitivity in generalized hybrid dynamics}

This section formulates an abstract, data-free, model of a hybrid system, by
considering a set of hybrid curves that is subject to some structural properties,
and studies the internal chain recurrence and transitivity properties of hybrid systems
in this abstract setting. Subsection \ref{section: sets of hybrid curves} introduces
hybrid curves and the properties that sets of them need to possess to be susceptible
to the analysis that follows. Subsection 
\ref{section: chain and internal chain transitivity and recurrence}
introduces chain and internal chain recurrence and transitivity. 
Subsection \ref{section: concatenation} defines a set-valued way to concatenate
hybrid curves if one ends not at the same point that the next one starts. 
This concept and several results from the mentioned three subsections are
preliminary work for the main developments in 
Section \ref{section: asymptotic curves and their asymptotic properties}. 
This section concludes with a related result on the internal chain recurrence
of a chain recurrent part of an invariant set, in Subsection
\ref{section: chain recurrent sets}

\subsection{Sets of hybrid curves}
\label{section: sets of hybrid curves} 

Hybrid time domains and arcs constitute basic building blocks in the definition of
a solution to a hybrid inclusion \cite{Goebel12a}. 
These concepts are recalled here. 
A definition of a hybrid curve is also proposed.

A set $E\subset\reals^2$ is a {\em compact hybrid time domain} 
\index{hybrid time domains}
if 
\begin{equation}
\label{chtd}
E=\left([t_0,t_1],0\right)\, \cup\, \left([t_1,t_2],1\right)\, \cup \, \dots \, \cup
\left([t_{J-1},t_J],J-1\right)
\end{equation}
where $J\in\nats$ and $0=t_0\leq t_1\leq t_2 \leq \dots \leq t_J$ form a finite sequence of real numbers. 
A set $E\subset\reals^2$ is a {\em hybrid time domain} if it is the union of a nondecreasing sequence
$E_1\subset E_2\subset E_3\subset\dots$ of compact hybrid time domains. 
The {\em length} of a hybrid time domain $E$ is $\length E:=\sup\{t+j\, |\, (t,j)\in E\}$.

A function $\sol:\dom\sol\to\reals^d$ is a {\em hybrid curve} if $\dom\sol$ is a hybrid time domain and, when $(t,j),(t',j) \in \dom \sol$ with $t<t'$,
$t \mapsto \sol(t,j)$ is continuous on $[t,t']$, and it is a
{\em hybrid arc} if, additionally, $t \mapsto \sol(t,j)$ is absolutely
continuous on $[t,t']$. 
A hybrid curve/arc $\sol$ is {\em trivial} if $\dom\sol=\{(0,0)\}$
and {\em nontrivial} otherwise.
A hybrid curve/arc $\sol$ is
{\em complete} if $\dom\sol$ is unbounded.
Throughout the paper, a hybrid curve or arc $\sol$ is occasionally identified with
a set-valued mapping from $\reals^2$ to $\reals^d$ that has empty values outside of $\dom\sol$. 

Let $\L$ be a set of hybrid curves. This notation applies throughout the paper.
For example, $\L$ can be the set of some, or all, solutions to 
a hybrid inclusion, as discussed later in Section \ref{section:hybrid-systems}.
Any $\sol\in\L$ may be called an {\em $\L$-curve}.

The set $\L$ is {\em locally bounded} if for every $r,T>0$ there exists $R>0$ such that, for every $\sol\in\L$ 
with $\sol(0,0)\in r\ball$ and every $(t,j)\in\dom\sol$ with $t+j\leq T$, 
one has $\sol(t,j)\in R\ball$. 
Similarly, a sequence $\sol_i\in\L$ is {\em locally eventually bounded} if for every $T>0$ 
there exist $i_0,R>0$ such that, for every $i\geq i_0$ and every $(t,j)\in\dom\sol$ with $t+j\leq T$, one has $\sol_i(t,j)\in R\ball$. 
The set $\L$ is {\em nominally well-posed} if 
\begin{itemize} 
\item[(i)]
every graphically convergent and locally eventually bounded sequence of hybrid curves 
$\left\{ \sol_n \right\}_{n=1}^{\infty}$ with
$\sol_n\in\L$ for all $n \in \nats$ has its graphical limit in $\L$, and
\item[(ii)]
if there exists no hybrid curve $\sol\in\L$ with $\sol(0,0)=x$ that {\em blows up in finite time},
in the sense that $\dom\sol$ is bounded and $\sol$ is not bounded, 
then every sequence of hybrid curves
$\left\{ \sol_n \right\}_{n=1}^{\infty}$ with
$\sol_n\in\L$ for all $n \in \nats$
and with $\lim_{n \rightarrow \infty} \sol_n(0,0) =x$ is locally eventually
bounded. 
\end{itemize} 
Note that local boundedness of $\L$ ensures that there is no blow up in 
finite time, as described in (ii) above, and thus that every sequence of $\L$-curves
with convergent initial conditions is locally eventually bounded.

Graphical convergence of $\sol_i$ to $\sol$ means that the sequence of graphs
$\gph\sol_i$ converges to the graph $\gph\sol$, in terms of set convergence. 
For a hybrid curve $\sol$, its {\em graph} is the subset of $\reals^{d+2}$ defined by
$$\gph\sol:=\left\{(t,j,x)\in\reals^{d+2}\, :\, (t,j)\in\dom\sol,\ x=\sol(t,j)\right\}$$
and a similar definition applies to a set-valued mapping $\sola:\dom\sola\rightrightarrows\reals^d$ with $\dom\sola\subset[0,\infty)^2$,
with $x\in\sola(t,j)$ replacing $x=\sol(t,j)$. 
For details, see \cite[Chapter 5, Section E]{RockafellarWets98}, 
or \cite[Section 5.3]{Goebel12a} in the context of hybrid arcs.
If $\L$ is the set of all solutions to a hybrid inclusion, the concept of nominal
well-posedness essentially agrees with \cite[Definition 6.2]{Goebel12a}. 
If $\L$ is locally bounded then every sequence of hybrid curves
$\left\{ \sol_n \right\}_{n=1}^{\infty}$ with
$\sol_n\in\L$ for all $n \in \nats$ and with
convergent initial conditions is locally eventually bounded, and nominal well-posedness
reduces to condition (i) above.

\begin{lemma}
\label{lemma: L-osc}
Let $\L$ be a locally bounded and nominally well-posed set of hybrid curves. 
For every compact $K\subset\reals^d$ and every $T,\ve>0$ there exists 
$\delta>0$ such that, for every $\sol\in\L$ with $\sol(0,0)\in K+\delta\ball$, 
there exists $\sola\in\L$ with $\sola(0,0)\in K\cap(\sol(0,0)+\ve\ball)$
such that
\begin{equation} 
\label{tau epsilon close}
\gph_{t+j\leq T}\phi\subset\gph\psi+\ve\ball \quad \mbox{and} \quad 
\gph_{t+j\leq T}\psi\subset\gph\phi+\ve\ball,
\end{equation} 
and for every such $\sol$ that, additionally, has range in $K+\delta\ball$,
one can take $\sola$ with range in $K$. 
\end{lemma} 

Above, $\gph_{t+j\leq T}\sol$ is the set
$$\left\{(t,j,x)\in\reals^{d+2}\, :\, (t,j)\in\dom\sol,\, t+j\leq T,\, x=\sol(t,j)\right\}.$$
By \cite[Theorem 4.10]{RockafellarWets98}, a bounded sequence of set-valued
mappings $\sol_i:\reals^2\rightrightarrows\reals^d$ with $\dom\sol_i\subset[0,\infty)^2$
converges graphically to $\sol:\reals^2\rightrightarrows\reals^d$ with $\dom\sol\subset[0,\infty)^2$ if and only if, for every $T,\ve>0$, for all large
enough $i$ \eqref{tau epsilon close} holds with $\sol_i$ in place of $\sola$.
In \cite{Goebel12a}, the notion of $(T,\ve)$-closeness was used to express 
properties like \eqref{tau epsilon close}. If $\sol$ and $\sola$ satisfy 
\eqref{tau epsilon close} with $T>0$ and $\ve\in(0,1)$ then they are 
$(T,\ve)$-close in the sense of \cite[Definition 5.23]{Goebel12a}, while if
$\sol$ and $\sola$ are $(T,\ve)$-close with $T,\ve>0$ then \eqref{tau epsilon close}
holds with $\ve$ replaced by $\sqrt{2}\ve$. 
Thus, the lemma statement and its proof are similar to those for
\cite[Theorem 5.25]{Goebel12a} but in the context of the set $\L$.
The final observation of the lemma, regarding $\sola$ with range in $K$
for $\sol$ with range in $K+\delta\ball$, is new, even if trivial.
 
\

\begin{proof}
Suppose not: there is a compact $K\subset\reals$ and $T,\ve>0$ such that, by taking
$\delta_n=1/n$, there exists $\sol_n\in\L$ with $\sol_n(0,0)\in K+n^{-1}\ball$ such that,
for every $\sola\in\L$ with $\sola(0,0)\in K\cap(\sol(0,0)+\ve\ball)$, \eqref{tau epsilon close} fails.
Use \cite[Theorem 5.36]{RockafellarWets98} to pass to a subsequence, without relabeling, such
that the sequence $\left\{ \sol_n \right\}_{n=1}^{\infty}$ converges graphically, to some set-valued mapping $\sola:\reals^2\rightrightarrows\reals^d$ and 
the sequence $\left\{\sol_n(0,0)\right\}_{n=1}^{\infty}$ converges to some $x\in K$
with $\sola(0,0)=x$.
Because $\L$ is locally bounded, the subsequence is locally eventually bounded. 
Then, because $\L$ is nominally well-posed, $\sola\in\L$. Because $\L$ is locally bounded,
again, \eqref{tau epsilon close} is, essentially, the characterization of the convergence of $\gph\sol_n$ to $\gph\sola$. 
Thus \eqref{tau epsilon close} cannot fail for every $\sola\in\L$ with $\sola(0,0)\in\sol(0,0)+\ve\ball$. This is a contradiction. 
The same reasoning, carried out for a sequence of $\sol_n$ with ranges in $K+n^{-1}\ball$,
and thus leading to $\sola$ with range in $K$, confirms the final observation. 
\end{proof}

Given a hybrid curve $\sol$ and any $(s,i)\in\dom\sol$, define the ``tail'' $\sol^{(s,i)}$ of $\sol$ after $(s,i)$ by
\begin{equation}
\label{tail}
\begin{aligned}
\dom\sol^{(s,i)} &:= \left\{(t,j)\in[0,\infty)^2\, :\, (s+t,i+j)\in\dom\sol\right\}, \\
\sol^{(s,i)}(t,j) &:= \sol(s+t,i+j) \qquad \forall (t,j)\in\dom\sol^{(s,i)}.
\end{aligned}
\end{equation}
A set $\L$ of hybrid curves is {\em tail invariant} if, for every $\sol\in\L$ and every 
$(s,i)\in\dom\sol$, one has $\sol^{(s,i)}\in\L$. 

\begin{lemma}
\label{lemma: tail invariance to locally uniformly continuous}
Let $\L$ be nominally well posed, locally bounded, and tail invariant. 
Then $\L$ is {\em locally uniformly continuous}, in the sense that for every
$R>0$, $\ve>0$ there exists $\delta>0$ such that, for every $\sol\in\L$,
every $(t,j),(t',j')\in\dom\sol$ with $|t+j-(t'+j')|<\delta$ and 
$\sol(t,j),\sol(t',j')\in R\ball$, one has 
$\sol(t,j)\subset\sol(t',j')+\ve\ball$. 
\end{lemma}

\begin{proof}
It the opposite case, and using the tail invariance property of $\L$,
there exist $R,\ve>0$, $\sol_n\in\L$ with $\sol_n(0,0)\in R\ball$,
and $(t_n,j_n)\in\dom\sol_n$ with $t_n+j_n<1/n$ and $\sol_n(t_n,j_n)\in R\ball$
such that $\|\sol_n(t_n,j_n)-\sol_n(0,0)\|>\ve$. Passing to a graphically
convergent subsequence, and relying on compactness of $R\ball$, suggest
that the graphical limit of $\sol_n$ is multivalued at $(0,0)$. 
However, local boundedness and nominal well posedness ensure that this graphical 
limit is an $\L$-curve, and cannot be multivalued.
\end{proof}

\subsection{Chain and internal chain transitivity and recurrence}
\label{section: chain and internal chain transitivity and recurrence}

Given two points $x,y\in\reals^d$ and $\tau,\ve>0$, a
{\em $(\tau,\varepsilon)$-$\L$-chain from $x$ to $y$} consists
of finite sequences of points  $x = x_0, x_1, \ldots , x_{k^*} = y$ in $\reals^d$, 
and curves $\sol_0,\sol_1,...,\sol_{k^*-1}$ in $\L$ such that
$\sol_k(0,0) = x_k$,
$x_{k+1} \in \sol_k(t_k,j_k) + \varepsilon \Ball$ for some $(t_k,j_k)\in\dom\sol_k$ 
with $t_k +j_k\geq\tau$. 
A {\em generalized $(\tau,\varepsilon)$-$\L$-chain from $x$ to $y$} 
{\em relaxes $x_0=x$ to $x_0\in x+\varepsilon\ball$.} 
When a nonempty and closed set $K\subset\reals^d$ is under consideration,
given two points $x,y \in K$ and $\tau,\ve>0$,
an {\em internal $(\tau,\varepsilon)$-$\L$-chain from $x$ to $y$} 
is a $(\tau,\varepsilon)$-$\L$-chain from $x$ to $y$
such that $\sol_k(t,j)\in K$ for all $(t,j)\in\dom\sol_i$ with $0\leq t+j\leq t_k+j_k$,
for $0\leq k<k^*$. A similar restriction defines a
{\em generalized internal $(\tau,\varepsilon)$-$\L$-chain from $x$ to $y$}.

The set $K\subset\reals^d$ is {\em (internally) $\L$-chain transitive} if, for every $x,y\in K$ 
and every $\tau,\ve>0$, there exists an (internal) $(\tau,\varepsilon)$-$\L$-chain from $x$ to $y$, 
and {\em generalized (internally) $\L$-chain transitive} if generalized (internal) $(\tau,\varepsilon)$-$\L$-chains from $x$ to $y$ exist. 
Further, $K$ is {\em (internally) $\L$-chain recurrent} if, for every $x\in K$ and every $\tau,\ve>0$, there exists an (internal) $(\tau,\varepsilon)$-$\L$-chain from $x$ to $x$. 
Clearly, if $K$ is (internally) $\L$-chain transitive then
it is (internally) $\L$-chain recurrent, by taking $y=x$.

Establishing the generalized version of (internal) chain recurrence or transitivity is often easier than establishing the not generalized version.
However, the two variations are often equivalent. 
For hybrid inclusions, the equivalence of generalized chain recurrence to
chain recurrence is shown in \cite[Lemma 5]{Goebel23SCL}. A similar
argument verifies the next result. 

\begin{lemma}
\label{lemma:from-generalized-2-not}
Let $\L$ be locally bounded and nominally well posed.
A compact set $K\subset\reals^d$ is internally $\L$-chain transitive 
if and only if it is generalized internally $\L$-chain transitive. 
\end{lemma}

\begin{proof}
One implication is obvious. For the other, suppose that $K$ is 
generalized internally $\L$-chain transitive. 
Fix $x,y\in K$, $\tau>0$, and $\varepsilon\in(0,1)$. 
Let $\L'$ be the set of all $\sol\in\L$ for which $\sol(t,j)\in K$ for all $(t,j)\in\dom\sol$
with $t+j< 2\tau+1+\ve$. Note that $\L'$ is locally bounded and nominally well posed.
Apply Lemma \ref{lemma: L-osc} to $\L'$ and the compact set $\{x\}$ to obtain $\delta\in(0,\varepsilon/2)$ so that for every
$\phi\in\L'$ with $\phi(0,0)\in x+\delta\ball$
there exists $\psi\in\L'$ with $\sola(0,0)=x$ such that
\eqref{tau epsilon close} holds with $T=\tau+\ve+1$.
(The rest of the proof is, essentially, as the proof of \cite[Lemma 5]{Goebel23SCL}.)
By assumption, there exists a generalized $(2\tau+1+\varepsilon,\delta)$-$\L$-chain 
from $x$ to $y$, with range in $K$, determined by $\phi_0,\phi_1,\dots,\phi_{k^*-1}$. 
Let $(t',j')\in\mathop{\rm dom}\nolimits\phi_0$ be such that $\tau+\varepsilon\leq t'+j'<\tau+\varepsilon+1$. 
By the choice of $\delta$, there exists $\psi\in\L'$ such that $\psi(0,0)=x$ and
$\gph_{t+j\leq t'+j'}\phi_0\subset\gph\psi+\ve\ball$.
Let $(t'',j'')\in\mathop{\rm dom}\nolimits\psi$ be such that $|t''-t'|\leq\varepsilon$, $j''=j'$, 
and $\phi'_0(t',j')\in\psi(t'',j'')+\varepsilon\mathbb{B}$. 
Then $\psi(0,0)=x$, 
$\phi_0^{(t',j')}(0,0)=\phi_0(t',j')\in\psi(t'',j'')+\varepsilon\mathbb{B}$ and $t''+j''\geq\tau$, and $\phi_0^{(t',j')}(t_0-t',j_0-j')=\phi_0(t_0,j_0)$ with $(t_0-t')+(j_0-j')\geq\tau$. Consequently, $\psi,\phi_0^{(t',j')},\phi_1,\dots,\phi_{k^*-1}$ determine an internal $(\tau,\varepsilon)$-$\L$-chain from $x$ to $y$. 
\end{proof}

``Internal'' above underlines that the relevant ranges of $\sol_i$ are in $K$,
but does not require that the whole ranges of $\sol_i$ be in $K$. 
In what follows, the {\em range} of a $(\tau,\ve)$-$\L$-chain is understood 
to be the union of the relevant ranges of $\sol_k$, i.e., the union of
$\sol_k(t,j)$ over all $(t,j)\in\dom\sol_k$ with $0\leq t+j\leq t_k+j_k$,
and $0\leq k < k^*$.

Chain recurrence implies some invariance properties. The set $K\subset\reals^n$ 
is {\em weakly forward $\L$-invariant} if, for every $x\in K$, there exists a complete $\sol\in\L$ with $\sol(0,0)=x$ and range contained in $K$,
and {\em weakly backward $\L$-invariant} if, for every $x\in K$ and $T>0$,
there exists $\sol\in\L$ and $(t,j)\in\dom\sol$ with $t+j\geq T$ such that $\sol(t,j)=x$
and $\sol(s,i)\in K$ for all $(s,i)\in\dom\sol$ with $s+i\leq t+j$.

The following result generalizes \cite[Lemma 3.5]{Benaimetal2005}
to the hybrid setting and $\L$-chains and $\L$-invariance, also with some more general assumptions.

\begin{proposition}
\label{proposition:brilliant-Benaim}
Let $\L$ be nominally well-posed and $K\subset\reals^d$ be a closed set that is internally 
$\L$-chain recurrent. 
\begin{itemize} 
\item[(a)]
If $\L$ is locally bounded or $K$ is compact then $K$ is weakly forward $\L$-invariant.
\item[(b)] 
If $\L$ is tail-invariant and $K$ is compact then $K$ is weakly backward $\L$-invariant.
\end{itemize} 
\end{proposition}

\begin{proof}
Pick $x\in K$. For an arbitrary $\ve>0$ and $\tau=n$, $n\in\nats$, take an internal
$(\tau,\ve)$-$\L$-chain from $x$ to $x$, and let $\sol^n:=\sol_0^n$ be the first ``link'' in it. 
Use \cite[Theorem 5.36]{RockafellarWets98} to pass to a subsequence $\sol^n$, without relabeling,
so that $\sol^n$ converges graphically. Because $\sol^n(0,0)=x$ for all $n\in\nats$ and 
$\L$ is locally bounded or $K$ is compact, the sequence $\sol^n$ is locally eventually bounded.
Because $\L$ is well posed, the graphical limit of the sequence, say $\sol$, is in $\L$.
Because $K$ is closed and, for each $n\in\nats$, 
$\sol^n(t,j)\in K$ for all $(t,j)\in\dom\sol^n$ with $t+j\leq n$, 
the range of $\sol$ is in $K$. 
Finally, \cite[Example 5.19]{Goebel12a} implies that $\dom\sol$ is unbounded and 
so $\sol$ is complete, because $\dom\sol^n$ are ``increasing in length,'' 
Obviously, $\sol(0,0)=x$. The weak invariance of $K$ was shown. 

For backward invariance, pick $x\in K$ and $T>0$. 
For $\ve=1/n$ and $\tau=n$, $n\in\nats$, take an internal $(\tau,\ve)$-$\L$-chain from $x$ to $x$, 
and let $\sola_{k^*_n}^n$ be the last ``link'' in it. 
For every $n$, there exists $(t_n,j_n)\in\dom\sola_{k^*_n}^n$ with $\sola_{k^*_n}^n(t_n,j_n)\in x+n^{-1}\ball$,
and for every large enough $n$, there exists $(t'_n,j'_n)\in\dom\sola_{k^*_n}^n$ with 
$t_n+j_n-T-1\leq t'_n+j'_n\leq t_n+j_n-T$. For every large enough $n$, let $\sola_n$ be the
tail $\left(\sola_{k^*_n}^n\right)^{(t'_n,j'_n)}$ of $\sola_{k_n}^n$. By the tail invariance
assumption, each such $\sola^n$ is in $\L$. 
Compactness of $K$ ensures that the sequence $\sola^n$ is locally eventually bounded. 
Then \cite[Theorem 5.36]{RockafellarWets98}, again, lets one pass to subsequence, 
without relabeling, so that $\sola^n$ converge graphically to some $\sola\in\L$ which is complete and has range in $K$, for reasons similar to those in the previous paragraph, while $(t_n-t'_n,j_n-j'_n)$ converge to some $(t,j)\in\dom\sola$ with $t+j\geq T$ and $\sola(t,j)=x$. 
This confirms weak backward invariance of $K$. 
\end{proof}

\subsection{Generalized concatenation}
\label{section: concatenation}

A {\em concatenation} of two $\L$-curves $\sol_1$ and $\sol_2$ at
$(s,i)\in\dom\sol_1$ such that $\sol_1(s,i)=\sol_2(0,0)$ is the mapping $\sola$
defined by
$$\sola(t,j):=
\left\{
\begin{array}{ll}
\sol_1(t,j) & \mbox{if}\ (t,j)\in\dom\sol_1,\ t+j\leq s+i, \\[3pt]
\sol_2(t-s,j-i) & \mbox{if}\ (t-s,j-i)\in\dom\sol_2.
\end{array}
\right. 
$$
This definition naturally extends to more than two curves. 
The set $\L$ is {\em closed under concatenation}, if for every $\sol_1,\sol_2\in\L$ and every $(s,i)\in\dom\sol_1$ such that $\sol_1(s,i)=\sol_2(0,0)$, the concatenation
$\sol$ is in $\L$. 
Note that if the set $\L$ is closed under concatenation, and if $K$ is 
weakly forward $\L$-invariant then for every internal 
$(\tau,\ve)$-$\L$-chain from $x\in K$ to $y\in K$ there exists one for 
which the curves $\sol_0,\sol_1,...,\sol_{k^*-1}$ are complete
and have ranges in $K$.

A generalized concept of concatenation is now proposed.
Consider a finite sequence of $\L$-curves $\sola_0,\sola_1,\dots,\sola_{k^*}$ 
or an infinite one $\sola_0,\sola_1,\dots$ --- in which case $k^*=\infty$ ---
and an associated sequence $(t_k,j_k)\in\dom\sola_k$ for $0\leq k<k^*$.
The {\em generalized concatenation} of the sequence is the (likely) 
set-valued mapping $\Sola:\reals^2\rightrightarrows\reals^d$ where
$$\Sola(t,j):=
\left\{
\begin{array}{lll}
\sola_0(t,j) & \mbox{for} & (t,j)\in\dom\sola_0,\ t+j<t_0+j_0 \\[3pt]
\{\sola_0(t_0,j_0),\sola_1(0,0)\} & \mbox{for} & (t,j)=(t_0,j_0) \\[3pt] 
\sola_1(t-t_0,j-j_0) & \mbox{for} & 
\begin{array}{l}
(t,j)\in\dom\sola_1+(t_0,j_0),\\ \ \ \ \ \ \ t+j < t_0+t_1+j_0+j_1
\end{array}\\[3pt]
\{\sola_1(t_1,j_1),\sola_2(0,0)\} & \mbox{for} & (t,j)=(t_0+t_1,j_0+j_1) \\[3pt]
\vdots 
\end{array}
\right.
$$
and $\Sola(t,j):=\emptyset$ otherwise. This way, $\dom\Sola$ is a hybrid time
domain that is the expected concatenation of the hybrid time domains 
$\dom\sola_k|_{t+j\leq t_k+j_k}$, and --- in the finite case ---
the last point in $\dom\Sola$ 
is $$(t_0+t_1+\dots+t_{k^*},j_0+j_1+\dots+j_{k^*}).$$
Given a generalized concatenation $\Sola$, its {\em concatenation times} are
the points in $\dom\Sola$ given by $(t_0,j_0)$, $(t_0+t_1,j_0+j_1)$, etc.
and the {\em length} of the $k$-th segment in the generalized 
concatenation is $t_{k-1}+j_{k-1}$.
The notion of $\L$ being closed under generalized concatenation can be formulated but,
because $\L$-curves are single-valued on their domains, a generalized concatenation
of $\L$-curves that belongs to $\L$ is also a usual concatenation. 

The construction of a generalized concatenation can be applied to 
$\L$-curves that form an $(\tau,\ve)$-$\L$-chain. 
The main result of this section, Lemma \ref{lemma: limit of concatenations is a solution},
shows that when a graphically convergent sequence of generalized concatenations comes from 
a sequence of  $(\tau,\ve)$-$\L$-chains with $\ve$ decreasing to $0$, the limit
is an $\L$-curve. The technicalities of the proof are easier handled when the
curves in each concatenation have roughly the same length. 
The next results help with this.

\begin{lemma}
\label{lemma: long chain}
Let $\L$ be tail-invariant and $\tau>0$.  
Then, every $\sol\in\L$ with $\length\dom\sol>\tau$ is a concatenation of finitely many
$\sol_1,\sol_2,\dots,\sol_K\in\L$ or infinitely many $\sol_1,\sol_2,\dots\in\L$ such
that each concatenation segment has length in $[\tau,2\tau+1)$. 
\end{lemma}

\begin{proof}
Suppose that $\length\dom\sol>\tau$ is finite. 
If $\length\dom\sol<2\tau+1$, there is nothing to do. If $\length\dom\sol\geq 2\tau+1$,
pick $(s_1,i_1)\in\dom\sol$ so that
\begin{equation}
\label{Paco}
\tau\leq s_1+i_1<\tau+1.
\end{equation}
Let $\sol_1:=\sol$ and $\sol_2:=\sol^{(s_1,i_1)}$. Then $\sol$ is a concatenation
of $\sol_1$ and $\sol_2$ at $(s_1,i_1)$, \eqref{Paco} holds, and
$\length\dom\sol_2>\tau$. If $\length\dom\sol_2<2\tau+1$, the process stops. 
Otherwise, it is iterated finitely many times. If $\length\dom\sol=\infty$,
the process is the same but is iterated ad nauseam. 
\end{proof}

Consequently, if $\Sola$ is a generalized concatenation of 
a $(\tau,\ve)$-$\L$-chain, then $\Sola$ is a generalized concatenation of a
$(\tau,\ve)$-$\L$-chain for which each link has length in $[\tau,2\tau+1)$.

\begin{lemma}
\label{lemma: limit of concatenations is a solution}
Let $\L$ be nominally well-posed, locally bounded, tail-invariant, 
and closed under concatenation. 
Given $\tau>0$ and a sequence $\ve_n>0$, $n\in\nats$, consider a sequence of 
(finite or infinite) $(\tau,\ve_n)$-$\L$-chains and a sequence of their 
generalized concatenations $\Sola_n$ such that
\begin{itemize} 
\item[(i)] $\ve_n\to 0$,
\item[(ii)] the sequence $\Sola_n$ converges graphically,
\item[(iii)] $\Sola_n(0,0)$ converge.
\end{itemize} 
Then the graphical limit of the sequence $\Sola_n$ is an $\L$-curve
or a truncation of an $\L$-curve. 
\end{lemma}

Above and in what follows, a {\em truncation} of an $\L$-curve $\sol$
is the restriction $\sol|_{t+j\leq T}$ for some $T\geq 0$. 

\begin{proof}
Without loss of generality, by Lemma \ref{lemma: long chain}, 
for each $n\in\nats$ let the $(\tau,\ve_n)$-$\L$-chain defining $\Sola_n$, 
be given by $\sola_0^n,\sola_1^n,\dots,\sola_{K_n}^n$ for some $K_n\in\natz$ 
or by $\sola_0^n,\sola_1^n,\dots$ (in which case $K_n:=\infty$)
and hybrid times $(t_k^n,j_k^n)\in\dom\sola_k^n$ with 
$t_l^n+j_k^n\in[\tau,2\tau+1)$ for $0 \leq k <K_n+1$, 
such that 
\begin{equation}
\label{dog park}
\sola_{k+1}^n(0,0)\in\sola_k^n(t_k^n,j_k^n)+\ve_n\ball \qquad 0\leq k<K_n.
\end{equation}  

Local boundedness of $\L$ and item (iii) ensure that the sequence 
is locally eventually bounded. 
Denote the graphical limit of the sequence $\Sola_n$ by $\Sola$. 
Each $\dom\Sola_n$ is a hybrid time domain.
Eventual local boundedness of the sequence ensures that $\dom\Sola$ 
is the limit of $\dom\Sola_n$,
and the limit of hybrid time domains is a hybrid time domain. (For details, see
\cite[Example 5.3, Example 5.19]{Goebel12a}.)

If $\dom\Sola$ is bounded, then, subject to passing to a subsequence of $\Sola_n$,
there exists $K\in\natz$ such that $K_n=K$ for all $n\in\nats$, and
\begin{itemize} 
\item[(a)]
$(t_k^n,j_k^n)$ converge, as $n\to\infty$, to some $(t_k,j_k)$,
for all $0\leq k < K+1$;
\item[(b)]
$\sola_k^n(t_k^n,j_k^n)$ converge, as $n\to\infty$, to some $x_k\in\reals^d$, 
for all $0\leq k < K+1$;
\item[(c)] 
by assumption (i) and \eqref{dog park}, 
$\sola_k^n(0,0)$ converge, as $n\to\infty$, to the same $x_k$, 
for all $0\leq k < K$, 
\item[(d)] 
by \cite[Theorem 5.36]{RockafellarWets98}, $\sola_k^n$ converge graphically, 
as $n\to\infty$, and because $\L$ is nominally well-posed, the graphical limit,
denoted $\sola_k$, is in $\L$, for all $0\leq k < K+1$. 
\end{itemize} 
If $\dom\Sola$ is unbounded, set $K:=\infty$ and observe that assumption (ii)
implies that for every $k\in\nats$, one has that $k_l\geq k$ for all large 
enough $l$. This observation ensures that (a), (b), (c), and (d) above hold,
although with the understanding that for any particular $k$ and an $n$ that is
not large enough, $\sola_k^n$ and thus $(t_k^n,j_k^n)\in\dom\sola_k^n$ 
need not exist. 

In either case, with $K\in\natz$ or $K=\infty$, 
one can show that, for each $0\leq k<K$, 
the graphical limit of $\sola_k^l$ evaluated at $(t_k,j_k)$ 
and the graphical limit of $\sola_{k+1}^n$ evaluated at $(0,0)$ 
are both equal to $x_k$. This relies on $\L$ being locally uniformly 
continuous, as in Lemma \ref{lemma: tail invariance to locally uniformly continuous}. 
Thus, if $K=\infty$, $\Sola$ is the concatenation of 
$\sola_k$, $0\leq k$, at points $(t_k,j_k)$, $0\leq k$. 
If $K\in\natz$, $\Sola$ is a truncation, to 
$0\leq t+j\leq \sum_{k=0}^K t_k+j_k$, of the concatenation of 
$\sola_k$, $0\leq k\leq K$, at points $(t_k,j_k)$, $0\leq k<K$. 
The mentioned concatenations are in $\L$, because $\L$ is closed
under concatenation, and the proof is done. 
\end{proof}

Equipped with this result, one can generalize Lemma \ref{lemma: L-osc},
which is about $\L$-curves, to $(\tau,\ve)$-$\L$-chains. 
This generalization helps with the main result of the next subsection. 

\begin{lemma}
\label{lemma: closeness of solutions to chains}
Let $\L$ be nominally well-posed and closed under concatenation and
let $K\subset\reals^d$ be a compact set. For every $T,\tau,\rho>0$ there exist
$\ve,\delta>0$ such that, for every $(\tau,\ve)$-$\L$-chain $\sola_0,\sola_1,\dots,\sola_k$
with range in $K+\delta\ball$ and every $(s,i)\in\dom\Sola$, where $\Sola$ is the generalized
concatenation of $\sola_0,\sola_1,\dots,\sola_k$, there exists $\sol\in\L$ 
with range in $K$ and such that 
\begin{equation}
\label{Klose}
\gph\Sola^{(s,i)}|_{t+j\leq T}\subset\gph\sol+\rho\ball, \quad
\gph\sol|_{t+j\leq T}\subset\gph\Sola^{(s,i)}+\rho\ball.
\end{equation}
\end{lemma}

\begin{proof}
Suppose not: there exist $T,\tau,\rho>0$ such that, with $\delta_n=\ve_n=1/n$, one
has $(\tau,\ve_n)$-chains $\sola^n_0,\sola^n_1,\dots,\sola^n_{k_n}$ with range in
$K+\delta_n\ball$ and $(s_n,i_n)\in\dom\Sola_n$ such that, for every $\sol\in\L$
with range in $K$, \eqref{Klose} fails with $\Sola_n$ in place of $\Sola$ and
$(s_n,i_n)$ in place of $(s,i)$. Without relabeling, pass to a subsequence so that
$$\Sola'_n:=\Sola_n^{(s_n,i_n)}$$
converge graphically. 
By Lemma \ref{lemma: limit of concatenations is a solution},
the graphical limit of $\Sola'_n$, denoted $\sol$, is in $\L$. 
Then
$$\gph\Sola'_n|_{t+j\leq T+\rho}\subset\gph\sol+\rho\ball, \quad
\gph\sol|_{t+j\leq T+\rho}\subset\gph\Sola^{(s,i)}+\rho\ball$$
for all large enough $n$. This is a contradiction. 
\end{proof}

\subsection{Chain recurrent sets}
\label{section: chain recurrent sets}

Throughout this section, $X\subset\reals^d$ is a nonempty and  compact set.

Let 
$$\L|_X:=\left\{\sol\in\L\, |\, \rge\sol\subset X\right\}.$$ 
Obviously, $X$ is {\em strongly forward $\L|_X$-invariant}, 
in the sense that every $\sol\in\L|_X$ with $\sol(0,0)\in X$ has 
$\rge\sol\subset X$. 
The {\em chain-recurrent set} of $X$, denoted $\R(X)$, is the set of all 
chain recurrent points $x\in X$, i.e., points such that, for every $\tau,\ve>0$
there exists a $(\tau,\ve)$-$\L|_X$-chain from $x$ to $x$. Equivalently, one
can request that there exist an internal, with respect to $X$, $(\tau,\ve)$-$\L$-chain.

The proof of the next result relies on the preliminary 
Lemma \ref{lemma: closeness of solutions to chains} and is similar to what was done in \cite{Robinson77RMJM} in a simpler setting. 

\begin{theorem}
\label{theorem: almost time to watch EPL}
Let $\L$ be nominally well-posed, tail invariant, and closed under concatenation.
Then the chain recurrent set $\R(X)$ of $X$ is internally chain recurrent. 
\end{theorem}

\begin{proof}
Recall that $\R(X)$ is defined using $\L|_X$, and --- without loss of 
generality --- suppose that the range of $\L$ is a subset of $X$. 
Fix $\tau>0$. Pick $x\in\R(X)$.

For $n\in\nats$, let $\sola^n_0,\sola^n_1,\dots,\sola^n_{k_n}$ be a $(\tau,1/n)$-chain
from $x$ to $x$ and let $K_n\subset X$ be its range, which is compact. 
Pass to a subsequence so that $K_n$ converge to some $K\subset X$. The following
will be shown:

\vskip.1cm

\noindent 
{\it Claim.} $K$ is internally chain recurrent. 

\vskip.1cm 

\noindent 
From the claim, it follows that $K\subset\R(X)$. Furthermore, because $x\in K$ and
$x$ is an arbitrary point in $\R(X)$, the set $\R(X)$ is internally chain recurrent. 

To prove the claim, pick $T>0$, $\ve'\in(0,1)$ and apply 
Lemma \ref{lemma: closeness of solutions to chains}
with $2T+4\rho+1$, $\tau$, and 
$\rho:=.5\ve'$ to obtain $\delta,\ve$ with properties as described in the lemma.  
Pick $y\in K$. 
Pick $n$ large enough so that $1/n<\ve$, $K_n\subset K+\delta\ball$, and there exists
$z\in K_n$ with $z\in y+\rho\ball$. 
Extend, if needed, the $(\tau,1/n)$-chain $\sola^n_0,\sola^n_1,\dots,\sola^n_{k_n}$
periodically, so that there exist $(s,i),(s',i')\in\dom\Sola$ with
$s'+i'\geq s+i+T+2\rho$ so that $z\in\Sola(s,i)=\Sola(s',i')$,
where $\Sola$ is the generalized concatenation of the (extended, if needed)
chain $\sola^n_0,\sola^n_1,\dots,\sola^n_{k_n}$.

The idea now is to break up $\Sola$ into ``shorter segments'', approximate each
segment by an $\L$-curve with ranges in $K$ using 
Lemma \ref{lemma: closeness of solutions to chains}, 
and from these $\L$-curves produce a generalized $(T,\ve')$-chain from $y$ to $y$. 

By Lemma \ref{lemma: long chain}, there exist
$(s_0,i_0)=(s,i)$, $(s_1,i_1)$, ..., $(s_{k^*},i_{k^*})=(s',i')$ in $\dom\Sola$ 
such that
\begin{equation}
\label{Mike Oldfield}
T+2\rho\leq s_k+i_k-s_{k-1}-i_{k-1}<2T+4\rho+1.
\end{equation}
for $k=1,2,\dots,k^*$. 
By Lemma \ref{lemma: closeness of solutions to chains}, for $k=0,1,\dots,k^*-1$
there exist $\sol'_k\in\L$ with range in $K$ such that
\begin{equation}
\label{Kloser}
\gph\Sola^{(s_k,i_k)}|_{t+j\leq 2T+4\rho+1}\subset\gph\sol_k+\rho\ball. 
\end{equation}
Let $y_0=z=y_{k^*}$ and for $k=1,2,\dots,k^*-1$ pick any $y_k\in\Phi(s_k,i_k)$. 
Then, for $k=0,1,\dots,k^*-1$, there exist $(t_k,j_k),(t'_k,j'_k)\in\dom\sol_k$ 
with $t_k\in[0,\rho]$, $j_k=i_k$ and $t'_k\in[s_{k+1}-s_k-\rho,s_{k+1}-s_k]$, $j'_k=i_{k+1}$
such that
$$y_k\in\sol_k(t_k,j_k)+\rho\ball, \quad
y_{k+1}\subset\sol_k(t'_k,j'_k)+\rho\ball.$$
Replacing $k$ by $k+1$ in the first inclusion above yields, for $k=0,1,\dots,k^*-2$, 
$$\sol_{k+1}(t_{k+1},j_{k+1})\in\sol_k(t'_k,j'_k)+2\rho\ball.$$
Because $y\in z+\rho\ball$, 
$$y\in\sol_0(t_0,j_0)+\rho\ball+\rho\ball, \qquad 
y\in\sol_{k^*-1}(t'_{k^*-1},j'_{k^*-1})+\rho\ball+\rho\ball.$$
Now, for $k=0,1,\dots,k^*{-1}$, let 
$$\sola_k:=\sol_k^{(t_k,j_k)}, \qquad (t''_k,j''_k):=(t'_k,j'_k)-(t_k,j_k).$$
Then
$$y\in\sola_0(0,0)+\ve'\ball, \qquad y\in\sola_{k^*-1}(t''_{k^*-1},j''_{k^*-1})+\ve'\ball,$$
and, for $k=0,1,\dots,k^*-2$,
$$\sola_{k+1}(0,0)\in\sola_k(t''_k,j''_k)+\ve'\ball.$$
Finally, for $k=0,1,\dots,k^*{-1}$,
$$t''_k+j''_k=t'_k+j'_k-t_k-j_k\geq s_{k+1}-s_k-\rho+i_{k+1}-\rho-i_k\geq T,$$
where the last inequality comes from \eqref{Mike Oldfield}. Thus
$\sola_0,\sola_1,\dots,\sola_{k^*-1}$ is a generalized $(T,\ve')$-chain from $y$ to $y$,
with range in $K$. 
Lemma \ref{lemma:from-generalized-2-not} implies that $K$ is internally chain recurrent. 
\end{proof}

\section{Asymptotic curves and their asymptotic properties}
\label{section: asymptotic curves and their asymptotic properties}

An asymptotic $\L$-curve, defined formally in 
Subsection \ref{section: asymptotic L-curves},
is a potentially set-valued mapping for which, roughly, the later its tail 
starts the more it resembles an $\L$-curve.
Subsection \ref{section: omega-limit sets of hybrid curves} shows that the $\omega$-limit
of any asymptotic $\L$-curve is internally $\L$-chain transitive, so that every asymptotic
$\L$-curve converges to a set that can be identified through $\L$.
Subsection \ref{section: abstract material} concludes the study of the abstract
hybrid dynamics represented by $\L$, by characterizing the smallest family $\L$ of 
hybrid curves for which a given hybrid curve is an asymptotic $\L$-curve. 

\subsection{Asymptotic $\L$-curves}
\label{section: asymptotic L-curves}

A set-valued mapping $\sol:\reals^2\rightrightarrows\reals^d$ with
$\dom\sol\subset[0,\infty)^2$ is an {\em asymptotic $\L$-curve} if:
\begin{itemize} 
\item[(i)] 
$\sol$ is bounded and $\dom\sol$ is unbounded, and
\item[(ii)]
for every graphically convergent sequence of the valued-valued mappings 
$\sol^{(t_n,j_n)}:[0,\infty)^2\rightrightarrows\reals^d$, 
with $(t_n,j_n)\in\dom\sol$ and $t_n+j_n\to\infty$ as $n\to\infty$,
its graphical limit is a complete $\L$-curve. 
\end{itemize} 
Above, $\sol^{(t,j)}$ represents the tail of $\sol$, with the definition identical to \eqref{tail}. 

For a simple example, consider $\L:=\{\sola_1,\sola_2\}$ with
$\dom\sola_1=\realsplus\times\{0\}$ and $\sola_1(t,0)=1$, and
$\dom\sola_2=\{0\}\times\natz$ and $\sola_2(0,j)=2$. Then the set-valued
$\sol$ with 
$\dom\sol=\left(\realsplus\times\{0\}\right) \cup\left(\{0\}\times\natz\right)$ 
and $\sol(0,0)=\{1,2\}$, $\sol(t,0)=1$ for all $t>0$, 
and $\sol(0,j)=2$ for all $j\in\nats$, 
is an asymptotic $\L$-curve. Indeed, for any sequence $(t_i,j_i)\in\dom\sol$
with $t_i+j_i\to\infty$, $\sol_i:=\sol^{(t_i,j_i)}$ converge graphically 
only if $\sol_i=\sola_1$ for all large enough $i$, or if $\sol_i=\sola_2$
for all large enough $i$, and in such case, the graphical limit is either
$\sola_1$ or $\sola_2$, and thus in $\L$. 

\begin{figure}[h]

\centering
\begin{tikzpicture}[scale=1.1]

\draw[thick,->] (-1,0) -- (10,0) node[anchor=north east] {$t$};
\draw[thick,->] (-0,-1) -- (0,3) node[anchor=north west] {$\phi(t,0)$};

\foreach \x in {1,2,3,4,5}
\draw (1.5*\x cm, 2pt) -- (1.5*\x cm, -2pt) node[anchor=north] {$\x$};

\foreach \x in 
{1.0000, 1.5000, 1.8333, 2.0833, 2.2833, 2.4500, 2.5929, 2.7179, 2.8290, 2.9290, 3.0199, 3.1032, 3.1792, 3.2485, 3.3118, 3.3693, 3.4220, 3.4709, 3.5169, 3.5600, 3.6005, 3.6389, 3.6755, 3.7106, 3.7446, 3.7771, 3.8085, 3.8392, 3.8687, 3.8976, 3.9258, 3.9528, 3.9791, 4.0049, 4.0301, 4.0547, 4.0787, 4.1020, 4.1247, 4.1470, 4.1687, 4.1900, 4.2108, 4.2312, 4.2512, 4.2707, 4.2899, 4.3088, 4.3273, 4.3454, 4.3633, 4.3808, 4.3980, 4.4149, 4.4316, 4.4479, 4.4640, 4.4798, 4.4953, 4.5106, 4.5256, 4.5404, 4.5550, 4.5693, 4.5834, 4.5973, 4.6110, 4.6245, 4.6378, 4.6509, 4.6638, 4.6765, 4.6891, 4.7015, 4.7137, 4.7257, 4.7376, 4.7493, 4.7609, 4.7723, 4.7836, 4.7947, 4.8057, 4.8166, 4.8273, 4.8379, 4.8484, 4.8588, 4.8690, 4.8791, 4.8891, 4.8990, 4.9088, 4.9185, 4.9280, 4.9375, 4.9469, 4.9562, 4.9654, 4.9745, 4.9835, 4.9924, 5.0012, 5.0099, 5.0185, 5.0270, 5.0354, 5.0437, 5.0519, 5.0600, 5.0681, 5.0760, 5.0839, 5.0917, 5.0994, 5.1070, 5.1145, 5.1220, 5.1294, 5.1367, 5.1440, 5.1512, 5.1583, 5.1653, 5.1723, 5.1792, 5.1860, 5.1928, 5.1995, 5.2061, 5.2127, 5.2192, 5.2256, 5.2320, 5.2383, 5.2446, 5.2508, 5.2570, 5.2631, 5.2691, 5.2751, 5.2811, 5.2870, 5.2928, 5.2986, 5.3043, 5.3100, 5.3156, 5.3212, 5.3267, 5.3322, 5.3376, 5.3430, 5.3484, 5.3537, 5.3589, 5.3641, 5.3693, 5.3744, 5.3795, 5.3845, 5.3895, 5.3945, 5.3994, 5.4042, 5.4090, 5.4138, 5.4185, 5.4232, 5.4279, 5.4325, 5.4371, 5.4416, 5.4461, 5.4506, 5.4550, 5.4594, 5.4637, 5.4680, 5.4723, 5.4765, 5.4807, 5.4849, 5.4890, 5.4931, 5.4971, 5.5011, 5.5051, 5.5090, 5.5129, 5.5168, 5.5206, 5.5244, 5.5281, 5.5319, 5.5356, 5.5392, 5.5428, 5.5464, 5.5500, 5.5535, 5.5570, 5.5605, 5.5639, 5.5673, 5.5707, 5.5740, 5.5774, 5.5807, 5.5840, 5.5872, 5.5905, 5.5937, 5.5969, 5.6001, 5.6033, 5.6064, 5.6096, 5.6127, 5.6158, 5.6189, 5.6220, 5.6251, 5.6282, 5.6312, 5.6343, 5.6373, 5.6403, 5.6433, 5.6463, 5.6493, 5.6522, 5.6552, 5.6581, 5.6610, 5.6639, 5.6668, 5.6697, 5.6726, 5.6754, 5.6783, 5.6811, 5.6840, 5.6868, 5.6896, 5.6924, 5.6952, 5.6980, 5.7007, 5.7035, 5.7062, 5.7090, 5.7117, 5.7144, 5.7171, 5.7198, 5.7225, 5.7252, 5.7279, 5.7305, 5.7332, 5.7359, 5.7385, 5.7411, 5.7438, 5.7464, 5.7490, 5.7516, 5.7542, 5.7568, 5.7594, 5.7619, 5.7645, 5.7671, 5.7696, 5.7722, 5.7747, 5.7772, 5.7798, 5.7823, 5.7848, 5.7873, 5.7898, 5.7923, 5.7947, 5.7972, 5.7997, 5.8021, 5.8046, 5.8070, 5.8095, 5.8119, 5.8143, 5.8168, 5.8192, 5.8216, 5.8240, 5.8264, 5.8288, 5.8312, 5.8336, 5.8360, 5.8384, 5.8407, 5.8431, 5.8455, 5.8478, 5.8502, 5.8525, 5.8549, 5.8572, 5.8595, 5.8619, 5.8642, 5.8665, 5.8688, 5.8711, 5.8734, 5.8757, 5.8780, 5.8803, 5.8826, 5.8848, 5.8871, 5.8894, 5.8916, 5.8939, 5.8961, 5.8984, 5.9006, 5.9029, 5.9051, 5.9073, 5.9096, 5.9118, 5.9140, 5.9162, 5.9184, 5.9206, 5.9228, 5.9250, 5.9272, 5.9294, 5.9316, 5.9338, 5.9360, 5.9382, 5.9403, 5.9425, 5.9447, 5.9468, 5.9490, 5.9511, 5.9533, 5.9554, 5.9575, 5.9597, 5.9618, 5.9640, 5.9661, 5.9682, 5.9703, 5.9725, 5.9746, 5.9767, 5.9788, 5.9809, 5.9830, 5.9851, 5.9872, 5.9893, 5.9914, 5.9935, 5.9956, 5.9977}
\draw[blue] (1.5*\x , {sin(deg(\x))-5/exp(\x)} ) -- (1.5*\x , {sin(deg(\x))+5/exp(\x)} ) ;

\end{tikzpicture} 

\caption{An asymptotic $\L$-curve for $\L=\left\{(t,0)\mapsto\sin(t+r)\, |\, r\in[0,2\pi)\right\}$.}
\label{mountain rain}

\end{figure}

For another example, where all curves are functions of $t$ only but graphical
convergence manifests itself, consider $\L=\{\sola_r\, |\, r\in[0,2\pi)\}$
with $\dom\sola_r=\realsplus\times\{0\}$, $\sola_r(t,0)=\sin(t+r)$.  
Let $T=\{s_1,s_2,\dots\}\subset\realsplus$ be such that $s_n\nearrow\infty$
while $s_{n+1}-s_n\to 0$, and define a set-valued $\sol$ on
$\dom\sol=T\times\{0\}$ by
$$\sol(t,0):=\left[\sin(t)-e^{-t},\sin(t)+e^{-t}\right].$$
Then $\sol$ is an asymptotic $\L$-curve. Note that for any sequence
$(t_n,0)\in\dom\sol$ with $t_n\nearrow\infty$, the domains of $\sol_n:=\sol^{(t_n,0)}$
are different from one another, but subject to passing to a subsequence,
$\sol_n$ converge graphically to one of the $\L$-curves $\sola_r$. 
See Figure \ref{mountain rain}.

Some of the results below hold if some or all of item (i) above is not
assumed, but not assuming (i) leads to somewhat pathological situations.
For example, every set-valued mapping $\sol:\reals^2\rightrightarrows\reals^d$
with a bounded domain vacuously satisfies item (ii). 
Also, any set-valued mapping $\sol:\reals^2\rightrightarrows\reals^d$
that ``diverges to $\infty$'' has the graphical limit of every
$\sol^{(t_n,j_n)}$, with $(t_n,j_n)\in\dom\sol$ and $t_n+j_n\to\infty$,
empty-valued at every $x\in\reals^2$ and thus not an $\L$-curve, no
matter what $\L$ is. Thus, any divergent to $\infty$ $\L$-curve cannot
satisfy (ii). Such pathologies are excluded by (i), and (i) is in fact
present in applications in later sections. 

The first result below shows that asymptotic $\L$-curves cannot have
persistent large ``gaps'' in their domains. 

\begin{lemma}
\label{lemma: ursaCs}
Let $\sol$ be an asymptotic $\L$-curve. For every $\ve>0$ there exists
$\kappa>0$ such that, for every $T\geq\kappa$, there exists
$(t,j)\in\dom\sol$ with $T\leq t+j<T+1+\ve$.
\end{lemma}

\begin{proof}
Suppose, for the sake of contradiction, that there exists $\ve>0$ such that
that for every $n\in\nats$ there exists $T_n\geq n$
such that there exist no $(t,j)\in\dom\sol$ with $T_n\leq t+j<T_n+1+\ve$.
For each $n$, let 
$(s_n,i_n)\in\dom\sol$ be such that
$$
s_n+i_n \geq\sup\{t+j\, :\, (t,j)\in\dom\sol,\ t+j\leq T_n\}-1/n.
$$
Then, after passing to a convergent subsequence of $\sol^{(s_n,i_n)}$,
its graphical limit contains in its domain no $(t,j)$ with $0<t+j\leq 1+\ve$ 
but is also supposed to be a complete $\L$-curve. This is impossible.
\end{proof}

The next two lemmas provide characterizations of the graphical convergence
of sequences of tails of an asymptotic $\L$-curve to $\L$-curves, in terms of
some uniform bounds applicable to compact sections of the graphs of the tails
and the $\L$-curves, like \eqref{tau epsilon close}.

\begin{lemma}
\label{lemma: uraCs}
Let $\sol$ be an asymptotic $\L$-curve.
For every $\ve,T>0$ there exists $\kappa>0$ such that, for every $(s,i)\in\dom\sol$
with $s+i\geq\kappa$, there exists an $\L$-curve $\sola$ such that
\eqref{tau epsilon close} holds and $\sola(0,0)\in\sol(s,i)+\ve\ball$.
\end{lemma}

\begin{proof}
Suppose not, so that there exist $\ve,T>0$ and, by taking $\kappa=n\in\nats$, there
exist $(s_n,i_n)\in\dom\sol$ with $s_n+i_n\geq n$ such that, for every 
$\sola\in\L$ one has
$\gph_{0\leq t+j\leq T}\sol_n \not\subset \gph\sola+\ve\ball$,
or $\gph_{0\leq t+j\leq T}\sola \not\subset \gph\sol_n+\ve\ball$
or $\sola(0,0)\not\in\sol_n(0,0)+\ve\ball$,
where $\sol_n:=\sol^{(s_n,i_n)}$. 
By \cite[Theorem 5.36]{RockafellarWets98}, because $\sol$ and thus
the sequence $\sol_n$ are bounded, one can pass to a subsequence, 
without relabeling, so that $\sol_n$ converge graphically to a set-valued
mapping with a nonempty graph and, moreover, that $\sol_n(0,0)$ converge.
The definition of an asymptotic $\L$-curve implies that the graphical
limit of $\sol_n$ is some $\sola\in\L$. Then $\sol_n(0,0)\to\sola(0,0)$ and
for all large enough $n$, \eqref{tau epsilon close} holds, 
resulting in a contradiction.
\end{proof}

\begin{lemma} 
\label{lemma: unwind your mind}
Let $\L$ be a nonempty set of hybrid curves that is nominally well posed.
Suppose that a bounded and complete set-valued mapping $\sol:\reals^2\rightrightarrows\reals^d$ with
$\dom\sol\subset[0,\infty)^2$ is such that 
for every $\ve,T>0$ there exists $\kappa>0$ such that, for every $(s,i)\in\dom\sol$
with $s+i\geq\kappa$, there exists an $\L$-curve $\sola$ such that
\eqref{tau epsilon close} holds with 
$\sol^{(s,i)}$ in place of $\sol$. Then $\sol$ is an asymptotic $\L$-curve.
\end{lemma}

\begin{proof}
Let $\sol$ be as described and suppose it is not an asymptotic $\L$-curve:
there exist $(t_n,j_n)\in\dom\sol$ with $t_n+j_n\to\infty$
such that $\sol_n:=\sol^{(t_n,j_n)}$ converge graphically to a set-valued
$\sola:\reals^2\rightrightarrows\reals^d$ that is not an $\L$-curve. 
Pick $T>0$, $\ve\in(0,1)$.
By \cite[Theorem 4.10]{RockafellarWets98}, there exists
$\kappa$ such that, for every $i\geq\kappa$, 
$$\gph_{t+j\leq T}\sol_n\subset\gph\sola+\ve\ball \quad \mbox{and} \quad 
\gph_{t+j\leq T}\sola\subset\gph\sol_n+\ve\ball,$$
and because $\ve<1$, 
\begin{align*}
\gph_{t+j\leq T}\sol_n & \subset\gph_{t+j\leq T+1}\sola+\ve\ball, \\
\gph_{t+j\leq T}\sola & \subset\gph_{t+j\leq T+1}\sol_n+\ve\ball.
\end{align*}
By the assumption on $\sol$, one can increase $\kappa$ so that, for every 
$n\geq\kappa$, one also has
$$\gph_{t+j\leq T+1}\sol_n\subset\gph\sola_n+\ve\ball \quad \mbox{and} \quad 
\gph_{t+j\leq T+1}\sola_n\subset\gph\sol_n+\ve\ball$$
for some $\sola_n\in\L$, and then
\begin{align*}
\gph_{t+j\leq T+1}\sol_n\subset\gph_{t+j\leq T+2}\sola_n+\ve\ball, \\
\gph_{t+j\leq T+1}\sola_n\subset\gph_{t+j\leq T+2}\sol_n+\ve\ball.
\end{align*}
Subject to passing to a subsequence, and without relabeling,
$\sola_n$ converge graphically to $\chi\in\L$, thanks to $\L$ being nominally well
posed. Then, subject to increasing $\kappa$, one has, for all $n\geq\kappa$, that
$$\gph_{t+j\leq T+2}\sola_n\subset\gph\chi+\ve\ball \quad \mbox{and} \quad 
\gph_{t+j\leq T+2}\chi\subset\gph\sola_n+\ve\ball.$$ 
Consequently, for every $T,\ve>0$ there exists $\kappa$ so that, for every
$i\geq\kappa$, 
$$\gph_{t+j\leq T}\sola\subset\gph\chi+3\ve\ball \quad \mbox{and} \quad 
\gph_{t+j\leq T}\chi\subset\gph\sola+3\ve\ball.$$
Because both $\sola$ and $\chi$ have closed graphs, in turn because both mappings
are graphical limits, the inclusions above ensure that $\sola=\chi$. This is a
contradiction. 
\end{proof}

\subsection{$\omega$-limit sets of hybrid curves}
\label{section: omega-limit sets of hybrid curves}

The $\omega$-limit of a hybrid curve $\sol$ is the set $\omega(\sol)$ defined by
\begin{equation}
\label{omega limit of a hybrid curve}
\omega(\sol):=\left\{z\in\reals^d\, :\
\exists (t_n,j_n)\in\dom\sol\ \mbox{s.t.}\ t_n+j_n\to\infty,\ \sol(t_n,j_n)\to z\right\}.
\end{equation}
This set is closed. If $\sol$ is complete and bounded, this set is nonempty, bounded and thus
compact, and $\sol$ converges to $\omega(\sol)$. 
All these basic properties generalize.

The {$\omega$-limit} of an asymptotic $\L$-curve $\sol$, also denoted $\omega(\sol)$,
is the set
$$\left\{z\in\reals^d\, |\, \exists (t_n,j_n)\in\dom\sol,\ y_n\in\sol(t_n,j_n)\ \mbox{s.t.}\ t_n+j_n\to\infty, y_n\to z\right\}.$$
Equivalently, 
$$\omega(\phi)=\bigcap_{n\in\nats} S_n \quad \mbox{where} \quad 
S_n:=\overline{ \sol\left(\{(t,j)\in[0,\infty)^2\, |\, t+j\geq n\}\right) },$$
or $\omega(\sol)$ is the outer limit, or just the limit, of the non-increasing sequence
of sets $S_n$. As a direct consequence of these alternative representations,
$\omega(\sol)$ is closed, and because $\phi$ is bounded and complete, $\omega(\sol)$  
is nonempty, bounded and hence compact, and $\sol$ converges to $\omega(\sol)$
in the sense that for every $\ve>0$ there exists $\tau\geq 0$ such that
\begin{equation} 
\label{ride to result}
\sol(t,j)\subset\omega(\sol)+\ve\ball\ \quad \forall (t,j)\in\dom\sol,\ t+j\geq\tau.\
\end{equation}
For details, see \cite[Exercise 4.3, Proposition 4.4, Theorem 4.10]{RockafellarWets98}.
These properties rely only on item (i) of the definition of an asymptotic $\L$-curve: 
that $\sol$ is bounded and $\dom\sol$ is not. 
The convergence concluded in \eqref{ride to result} implies
that the graphical limit of tails of $\sol$, in item (ii) of the definition 
of an asymptotic $\L$-curve is not just an $\L$-curve but has range in $\omega(\sol)$.

It is standard that for regular-enough dynamics, including for hybrid inclusions
subject to some assumptions, that $\omega$-limits of solutions are invariant, in 
appropriate senses. 
For an asymptotic $\L$-curve, $\omega(\sol)$ is weakly invariant, as made precise
below.
The proof is essentially the same as that of \cite[Proposition 6.21]{Goebel12a} 
although it relies on the definition of an asymptotic $\L$-curve and not on the hybrid
system, which is the setting of \cite{Goebel12a}, being nominally well-posed. 

\begin{proposition}
\label{proposition: PSG or Real}
Let $\sol$ be an asymptotic $\L$-curve. Then
$K:=\omega(\sol)$ is nonempty, 
compact, weakly $\L$-invariant: for every $x\in K$ and $T>0$ there exists a complete $\L$-curve
$\sola$ with range in $K$ such that $\sola(t,j)=x$ for some $(t,j)\in\dom\sola$ with
$t+j\geq T$, and $\sol$ converges to $K$.
\end{proposition}

The main results of this section are Theorem \ref{theorem:L-chain-igt} 
and the ensuing corollary, which conclude internal chain transitivity of
the $\omega$-limit of every asymptotic $\L$-curve. A preliminary result is needed.

\begin{lemma} 
\label{lemma: Dartmouth v Brown soon}
Suppose that $\L$ is nominally well-posed, locally bounded, and tail invariant. 
For every compact set $K\subset\reals^d$, $\ve>0$, and $T>0$ there exists
$\delta>0$ such that,
for every $\L$-curve $\sol$ with $\rge\sol|_{t+j\leq T}\subset K+\delta\ball$,
there exists a $\L$-curve $\sola$ such that
\begin{equation}
\label{larb}
\gph\sola|_{t+j\leq T}\subset\gph\sol+\ve\ball 
\end{equation}
and either
\begin{equation}
\label{tom} 
\rge\sola|_{t+j\leq T}\subset K \quad \mbox{and} \quad 
\gph\sol|_{t+j\leq T}\subset\gph\sola|_{t+j\leq T}+\ve\ball, 
\end{equation}
or  
\begin{equation}
\label{kha} 
\rge\sola|_{t+j<T}\subset K \quad \mbox{and} \quad 
\gph\sol|_{t+j\leq T}\subset\gph\sola|_{t+j<T}+\ve\ball. 
\end{equation}
\end{lemma}

\begin{proof}
In the opposite case, for some $K\subset\reals^d$, $\ve>0$, and $T>0$ there exists
a graphically convergent sequence $\sol_n$ with 
\begin{equation}
\label{pineapple}
\rge\sol_n|_{t+j\leq T}\subset K+n^{-1}\ball
\end{equation}
such that, for every $n\in\nats$, either \eqref{larb} fails for every $\sola\in\L$
with $\sol$ replaced by $\sol_n$, 
or \eqref{tom} and \eqref{kha} fail for every $\sola\in\L$
with $\sol$ replaced by $\sol_n$. Let the graphical limit of $\sol_n$ be $\sola$.
The bound \eqref{pineapple} and local boundedness of $\L$ imply that the sequence
of sets $\gph\sol_n|_{t+j\leq T+1}$ is uniformly bounded, and this ensures that
$\gph\sola|_{t+j\leq T}$ is bounded. From \cite[Theorem 4.10]{RockafellarWets98}, or its
consequence in \cite[Exercise 5.34]{RockafellarWets98} one can deduce that for all large enough
$n$,
\begin{equation} 
\gph\sola|_{t+j\leq T} \subset \gph\sol_n+\ve\ball, \qquad
\gph\sol_n|_{t+j\leq T} \subset \gph\sola+\ve\ball. 
\end{equation}
In particular, \eqref{larb} holds for all large enough $n$, with $\sol$ replaced by $\sol_n$,
by \cite[Theorem 4.10]{RockafellarWets98}.

Suppose first that $\sol_n|_{t+j\leq T}$ converge graphically to $\sola|_{t+j\leq T}$. 
By \cite[Theorem 4.10]{RockafellarWets98} again, the second inclusion in \eqref{tom} holds 
for all large enough $n$ with $\sol$ replaced by $\sol_n$. The first inclusion
in \eqref{tom} also holds, by \eqref{pineapple}. Thus \eqref{tom} holds 
for all large enough $n$, which contradicts the choice of $\sol_n$. 

Now suppose that $\sol_n|_{t+j\leq T}$ do not converge graphically to $\sola|_{t+j\leq T}$.
Because $\sol_n$ converge graphically to $\sola$, this can happen only if
there exists $(t',j')\in\dom\sola$ such that $t'+j'=T$ and $(t',j'-1)\in\dom\sola$. 
(And then, the point $(t',j',\sola(t',j'))\in\gph\sola$ is a limit of some 
sequence $(t'_n,j',\sol_n(t_n,j'))$ with $t_n+j'>T$ and not a limit of any such
sequence with $t_n+j'\leq T$.) 
Then 
$$\gph\sola|_{t+j<T}=\gph\sola|_{j\leq j'-1}$$
and, for all large enough $n$, 
$$\gph\sol_n|_{t+j\leq T}=\gph\sol_n|_{t+j<T}=\gph\sola|_{j\leq j'-1}. $$
By the definition of set or graphical convergence and the structure of hybrid time
domains, it must be that 
$\gph\sol_n|_{j\leq j'-1}$ converge to $\gph\sola|_{j\leq j'-1}$.
Thus $\gph\sol_n|_{t+j\leq T}$ converge to $\gph\sola|_{t+j<T}$. 
Then \cite[Theorem 4.10]{RockafellarWets98}, again, the second inclusion in \eqref{kha} holds 
for all large enough $n$ with $\sol$ replaced by $\sol_n$.
By \eqref{pineapple}, the first inclusion in \eqref{kha} holds. 
Thus \eqref{kha} holds 
for all large enough $n$, which contradicts the choice of $\sol_n$. 
\end{proof}

\begin{theorem}
\label{theorem:L-chain-igt}
Let $\L$ be locally bounded, nominally well-posed, and tail-invariant. 
Then, the $\omega$-limit set of every asymptotic $\L$-curve 
is internally generalized $\L$-chain transitive.
\end{theorem}

\begin{proof} 
Let $\sol$ be an asymptotic $\L$-curve and let $K:=\omega(\sol)$.
Pick $\tau>1$, $\ve \in (0,.25)$, and $x_a,x_b\in K$. 
The proof constructs a generalized internal $(\tau,5\ve)$-chain from $x_a$ to $x_b$. 

First, some constants are picked based on $\ve$ and $\tau$: 
\begin{itemize}
\item[(i)]  
Use Lemma \ref{lemma: tail invariance to locally uniformly continuous} to pick
$\mu<\ve$ such that, for every $\sol\in\L$, every $(t,j),(t',j')\in\dom\sol$ 
with $|t+j-(t'+j')|<2\mu$ (which ensures that $j=j'$)
and $\sol(t,j),\sol(t',j')\in K$, one has 
$\sol(t,j)\subset\sol(t',j')+\ve\ball$. 
\item[(ii)] 
Use Lemma \ref{lemma: Dartmouth v Brown soon}, 
with $\ve=\mu$ and $T=2\tau+5$, to obtain $\delta<\mu$ such that, 
for every $\L$-curve $\sol$ with $\rge\sol|_{t+j\leq T}\subset K+2\delta\ball$,
the conclusions \eqref{larb}, \eqref{tom}, and \eqref{kha} of the lemma hold.  
\item[(iii)] 
Use Lemma \ref{lemma: uraCs}, 
with $\ve=\delta$ and $T=2\tau+4$, to obtain $\kappa>0$.
Increase $\kappa$ if necessary so that the claim of 
Lemma \ref{lemma: ursaCs}
holds, with $\varepsilon=0.5$.
\end{itemize}

By the definition of $\omega(\sol)$, and because $\sol$ converges to $K$, 
there exist
$(t_a,j_a),(t_b,j_b) \in \dom\sol$ and $y_a\in\sol(t_a,j_a)$, $y_b\in\sol(t_b,j_b)$
such that
\begin{subequations}
	\begin{align}
	& t_b+j_b-(\tau+1)  \geq t_a+j_a \geq \kappa, 
	\label{eq:separated-times} \\
	& |y_a -x_a|  \leq \ve, 
	\label{eq:close-to-start} \\
	& |y_b - x_{b}|  \leq \ve, 
	\label{eq:close-to-end} \\
    & \sol(t,j)\in K+\delta\ball \quad \forall (t,j)\in\dom\sol,\ t+j\geq t_a+j_a. 
    \label{caffeine}
	\end{align}
\end{subequations}

Now, with the help of Lemma \ref{lemma: ursaCs}, 
the asymptotic solution $\sol$ between $(t_a,j_a)$ and $(t_b,j_b)$ 
is broken into ``shorter'' pieces,
similarly to what was done to an $\L$-curve in 
the proof of Theorem \ref{theorem: almost time to watch EPL}
using Lemma \ref{lemma: long chain}. Then, each shorter piece
is approximated by an $\L$-curve in $K$, leading to the needed internal 
$(\tau,5\ve)$-$\L$-chain.

There exist $(t^0,j^0),(t^1,j^1),\dots,(t^M,j^M)\in\dom\sol$ such that
$(t^0,j^0)=(t_a,j_a)$, $(t^M,j^M)=(t_b,j_b)$, and 
\begin{equation} 
\label{time gap}
t^{m+1}+j^{m+1}-t^m-j^m\in\left[\tau+1,2\tau+4\right] \quad \mbox{for}\ m=0,1,\dots,M-1.
\end{equation}
Indeed, if $t_b+j_b-t_a-j_a\leq 2\tau+4$, there is nothing to do. In the opposite case, 
use Lemma \ref{lemma: ursaCs} to obtain $(t^1,j^1)\in\dom\sol$
so that 
$$t_a+j_a+\tau+1\leq t^1+j^1 < t_a+j_a+\tau+2.5.$$
Then $t^1+j^1-t^0-j^0\geq \tau+1$ and 
$$t_b+j_b-t^1-j^1>t_b+j_b-t_a-j_a-\tau-2.5>2\tau+4-\tau-1.5>\tau+1,$$
and the process can be iterated.

By the choice of $\kappa$ via Lemma \ref{lemma: uraCs}, 
for $m=0,1,2,\dots,M-1$ there exist curves $\sola^m\in\L$ such that,
among other things, 
\begin{subequations}
	\begin{align}
    & \gph_{t+j\leq 2\tau+4}\sola^m \subset \gph\sol^{(t^m,j^m)}+\delta\ball,
	\label{cc-aka-CC}
	\\
	& \gph_{t+j\leq 2\tau+4}\sol^{(t^m,j^m)} \subset \gph\sola^m+\delta\ball. 
	\label{dd-aka-DD}
	\end{align}
\end{subequations}
Then, by \eqref{caffeine}, \eqref{cc-aka-CC},
and the choice of $\delta$ via Lemma \ref{lemma: Dartmouth v Brown soon}, 
for $m=0,1,\dots,M-1$ there exist curves 
$\sol_m\in\L$ so that a not needed here version of 
\eqref{larb}
holds and either
\begin{equation}
\label{tomi} 
\rge\sol_m|_{t+j\leq 2\tau+5}\subset K \quad \mbox{and} \quad 
\gph\sola^m|_{t+j\leq 2\tau+5}\subset\gph\sol_m|_{t+j\leq 2\tau+5}+\mu\ball, 
\end{equation}
or  
\begin{equation}
\label{khai} 
\rge\sol_m|_{t+j<2\tau+5}\subset K \quad \mbox{and} \quad 
\gph\sola^m|_{t+j\leq 2\tau+5}\subset\gph\sol_m|_{t+j<2\tau+5}+\mu\ball. 
\end{equation}
Either way, for $m=0,1,2,\dots,M-1$, $\rge\sol_m|_{t+j \leq 2\tau+4.5}\subset K$ and
combining \eqref{dd-aka-DD} with \eqref{khai} yields
\begin{equation} 
\label{dcd}
\gph_{t+j\leq 2\tau+4} \sol^{(t_m,j_m)} \subset 
\gph\sol_m|_{t+j\leq 2\tau+4.5}+2\mu\ball
\end{equation}

It remains to show that $\sol_m$, with appropriately picked hybrid times in
$\dom\sol_m$, form a generalized $(\tau,5\ve)$-chain from $x_a$ to $x_b$,
with range in $K$.
Let $y^0:=y_a$, $y^M:=y_b$, and for $m=1,2,\dots,M-1$, pick any $y^m\in\sol(t^m,j^m)$. 

Because $(y^0,0,0)\in\gph\sol^{(t^0,j^0)}$, \eqref{dcd} implies that
$|y^0-\sol_0(t,j)|\leq 2\mu$ for some $(t,j)\in\dom\sol_0$ with $t+j\leq 2\mu$,
and then, by the choice of $\mu$, one has $|y^0-\sol_0(0,0)|\leq 2\mu+\ve$.
This, and \eqref{eq:close-to-start}, imply 
\begin{equation}
\label{obi one}
|x_a-\sol_0(0,0)|=|x^0-\sol_0(0,0)|\leq 2\mu+2\ve\leq 4\ve.
\end{equation} 
 
Because, for $m=1,2,\dots,M$,  
$$(y^m,t^m-t^{m-1},j^m-j^{m-1})\in\gph\sol^{(t^{m-1},j^{m-1})}$$
and because of the bound \eqref{time gap}, the inclusion \eqref{dcd} yields 
$(t_{m-1},j_{m-1})\in\dom\sol_{m-1}$ such that 
\begin{equation} 
\label{drone me}
|y^m-\sol_{m-1}(t_{m-1},j_{m-1})|\leq 2\mu \qquad \mbox{for}\ m=1,2,\dots,M
\end{equation} 
and 
$$|t_{m-1}+j_{m-1}-(t^m-t^{m-1}+j^m-j^{m-1})|\leq 2\mu.$$
The latter inequality and \eqref{time gap} imply that 
\begin{equation}
\label{obi two}
t_{m-1}+j_{m-1}\geq\tau+.5-2\mu\geq\tau 
\qquad \mbox{for}\ m=1,2,\dots,M.
\end{equation}
Because, for $m=1,2,\dots,M-1$,
$$(y^m,0,0)\in\gph\sol^{(t^m,j^m)},$$
the inclusion \eqref{dcd} yields $(t',j')\in\dom\sol_m$ such that
$$|y^m-\sol_m(t',j')|\leq 2\mu, \quad t'+j'\leq 2\mu.$$
The choice of $\mu$ ensures that $|y^m-\sol_m(0,0)|\leq 2\mu+\ve$. Combined
with \eqref{drone me}, this yields
\begin{equation}
\label{obi three}
|\sol_m(0,0)-\sol_{m-1}(t_{m-1},j_{m-1})|\leq 4\mu+\ve\leq 5\ve 
\quad \mbox{for}\ m=1,2,\dots,M-1.
\end{equation}

Finally, 
\eqref{drone me} combined with \eqref{eq:close-to-end} yields 
\begin{equation}
\label{obi four}
|x^b-\sol_{M-1}(t_{M-1},j_{M-1})|\leq 2\mu+\ve\leq 3\ve.
\end{equation}

Based on \eqref{obi one}, \eqref{obi two}, \eqref{obi three}, and \eqref{obi four},
$\sol_m$ and $(t_m,j_m)$, $m=0,1,\dots,M-1$ form a generalized internal 
$(\tau,5\ve)$-chain from $x_a$ to $x_b$. 
Arbitrariness of $\ve$ terminates the proof. 
\end{proof} 

Combining Theorem
\ref{theorem:L-chain-igt} with Lemma \ref{lemma:from-generalized-2-not} yields the following conclusion:

\begin{corollary} 
\label{corollary:a confirmation provided}
Let $\L$ be a nonempty set of hybrid curves that is locally bounded, nominally well posed, and tail invariant.
The $\omega$-limit set of every asymptotic $\L$-curve is internally 
$\L$-chain transitive. 
\end{corollary}

A kind of a converse result to Corollary 
\ref{corollary:a confirmation provided}
is given next. 
The proof is inspired by the idea of \cite[Lemma 3.1]{HirschSmithZhao01JDDE},
which was done for discrete-time dynamics,
and uses the generalized concatenation concept, to help with the more technical
setting of this paper.

\begin{proposition}
Let $\L$ be nominally well posed, locally bounded, tail invariant, and closed under concatenation.
Let $K\subset\reals^d$ be a compact set that is internally 
$\L$-chain transitive. Then $K=\omega(\Sola)$ for some asymptotic $\L$-curve $\Sola$.
\end{proposition}

\begin{proof}
Let $x_0,x_1,x_2,\dots$ be a dense sequence in $K$. Fix $\tau>0$, and
for each $n\in\nats$, set $\ve_n=1/n$, and find a sequence of $(\tau,\ve_n)$-$\L$-chains
$\sol_0^n,\sol_1^n,\dots,\sol_{k_n}^i$ between consecutive pairs of points in 
this sequence: 
$$x_0,x_1,x_0,x_1,x_2,x_0,x_1,x_2,x_3,x_0,\dots$$ 
By Lemma \ref{lemma: long chain}, without loss of generality, one
can assume that the length of each link in each chain is in $[\tau,2\tau+1)$. 
For each $i\in\nats$, let $\Sol_i$ be the generalized concatenation of the 
$(\tau,\ve_i)$-$\L$-chain $\sol_0^i,\sol_1^i,\dots,\sol_{k_i}^i$,
and then let $\Sola$ be the generalized concatenation of the infinite sequence
$\Sola_1,\Sola_2,\dots$
In other words, $\Sola$ is the generalized concatenation of
\begin{equation}
\label{Arsenal won}
\sola_0^1,\sola_1^1,\dots,\sola_{k_1}^1,\sola_0^2,\sola_1^2,\dots,\sola_{k_2}^2,
\sola_0^3,\sola_1^3,\dots,\sola_{k_3}^3,\dots
\end{equation}
The range of $\Sola$ is contained in $K$ and, by construction, it is dense in $K$.
In fact, $\omega(\Sola)=K$. 
It remains to argue that $\Sola$ is an asymptotic $\L$-curve. 
Clearly, $\Sola$ is bounded and $\dom\Sola$ is not. 
Take any sequence $(t_i,j_i)\in\dom\Sola$ such that $t_i+j_i\to\infty$ 
and $\Sola_i:=\Sola^{(t_i,j_i)}$ converge graphically. For each $i$, let
$(s_i,k_i)\in\dom\Sola_i$ minimize $s+k$ over all $(s,k)\in\dom\Sola_i$ such that
$(t_i+s,j_i+k)$ is a concatenation time of $\Sola$. 
In the case that $s_i+k_i$ do not converge to $0$, subject to passing to a subsequence,
there exists $\tau'\in(0,\tau]$ such that $s_i+k_i\geq\tau'$ for all $i\in\nats$. 
Then $\Sola_i|_{t+j\leq\tau'}$ is a truncation of some $\L$-curve and, in particular,
$\Sola_i(0,0)$ is a singleton. Subject to passing to a further subsequence, one
can assume that $\Sola_i(0,0)$ converge.
Furthermore, by Lemma \ref{lemma: long chain} again, 
every $\Sola_i$ is a generalized concatenation of an infinite $(\tau',1/i)$-$\L$-chain.
Lemma \ref{lemma: limit of concatenations is a solution} implies that
the graphical limit of $\Sola_i$ is an $\L$-curve. 
Now, suppose that $s_i+k_i\to 0$, so that, subject to passing to a subsequence,
$k_i=0$ and $s_i\to 0$. 
For each $i\in\nats$, let $\Sola'_i$ be the tail $\Sola^{(t_i+s_i,j_i)}$
but with one of potentially two values of $\Sola^{(t_i+s_i,j_i)}(0,0)$ removed,
so that $\Sola'_i|_{t+j\leq\tau/2}$ is an $\L$-curve. 
(In other words, if $(t_i+s_i,j_i)$ is a concatenation time where two consecutive
entries $\sola$ and $\sola'$ from \eqref{Arsenal won} are concatenated, 
$\Sola'_i(0,0)=\sola'(0,0)$.) 
Because $\ve_i\searrow 0$ and $\L$ is locally uniformly continuous,
the graphical limit of $\Sola'_i$ is the same as the graphical limit of $\Sola_i$. 
Another application of Lemma \ref{lemma: limit of concatenations is a solution}
concludes the proof. 
\end{proof}

\subsection{A generalized $\omega$-limit in $\L$}
\label{section: abstract material}

Let $\L$ be nominally well-posed.
Given a bounded and complete $\sol\in\L$, one can construct the smallest 
family of hybrid curves $\L'\subset\L$ for which $\sol$ is an asymptotic $\L'$-curve.
Let $X$ be the set of all set-valued mappings $\sola:\reals^2\rightrightarrows\reals^d$
with $\dom\sola\subset[0,\infty)^2$ and nonempty and closed graphs.
Define the {\em $\ooomega$-limit} of $\sol$ as the set $\ooomega(\sol)$ given by
$$\ooomega(\sol):=\left\{\sola\in X\, :\
\exists (t_n,j_n)\in\dom\sol\ \mbox{s.t.}\ t_n+j_n\to\infty,\ \sol^{(t_n,j_n)}\xrightarrow[\mbox{\rm\scriptsize graphically}]{}\sola\right\}.$$
Then $\ooomega(\sol)$ is nonempty, by \cite[Theorem 5.36]{RockafellarWets98}, 
and because $\L$ is nominally well-posed, $\ooomega(\sol)\subset\L$.
Also, $\ooomega(\sol)$ is tail invariant. Indeed,
take $\sola\in\ooomega(\sol)$ and $(t,j)\in\dom\sola$, and let $\sola$ be the graphical
limit of $\sol_n:=\sol^{(s_n,i_n)}$ for some sequence $(s_n,i_n)\in\dom\sol$ with 
$s_n+i_n\to\infty$. Then, there exist $(t_n,j_n)\in\dom\sol_n$ with $(t_n,j_n)\to(t,j)$
and $\sol_n(t_n,j_n)\to\sola(t,j)$. By the definition of a tail, and because $s_n+i_n\to\infty$,
$(s_n+t_n,i_n+t_n)\in\dom\sol$ and $s_n+t_n+i_n+t_n\to\infty$. It is an exercise to verify
that the sequence $\sol^{(s_n+t_n,i_n+t_n)}$ converges graphically and its graphical 
limit equals $\sola^{(t,j)}$. 
Thus, $\sola^{(t,j)}\in\ooomega(\sol)$, by the definition of this set.

For example, if $\sol(t,0):=\sin(t)+e^{-t}$ on $\dom\sol=T\times\{0\}$, 
then $\ooomega(\sol)=\{\sola_r\, |\, r\in[0,2\pi)\}$ with 
$\dom\sola_r=\realsplus\times\{0\}$, $\sola_r(t,0)=\sin(t+r)$.  
Clearly, for every $\sola_r$ and $s\geq 0$, $(t,0)\mapsto\sola_r(s+t,0)$
is in $\ooomega(\sol)$.

The set $X$ can be equipped with a metric that characterizes graphical convergence (see \cite[Theorem 5.50]{RockafellarWets98}) and some properties of 
$\ooomega(\sol)$ can be studied in a metric space setting. For example, 
$\ooomega(\sol)$ is closed, because it is an intersection of closed sets:
$$\ooomega(\sol)=\bigcap_{n\in\nats} N_n\quad \mbox{where}\quad 
N_n=\overline{\left\{\sol^{(s,i)}\, |\, s+i\geq n\right\} }.$$
In earlier terminology, this says that $\ooomega(\sol)$ is nominally well-posed. In summary of
the developments above:

\begin{proposition}
Let $\L$ be nominally well-posed and $\sol\in\L$ be bounded and complete.
The set $\ooomega(\sol)$ is the smallest set of $\L$-curves such that $\sol$ 
is an asymptotic $\ooomega(\sol)$-curve. 
Moreover, it is nonempty, tail invariant, and nominally well-posed. 
\end{proposition}

In the setting of differential equations or inclusions, leading to flows in
a metric space parameterized by $t\in[0,\infty)$, existing tools and results about 
invariance and chain transitivity of $\omega$-limits for flows on a metric
space can be used to study objects corresponding to $\ooomega(\sol)$ above. 
This was done, for example, in \cite{Benaimetal2005}. For hybrid dynamics,
and even more so for generalized hybrid dynamics given by sets $\L$,
such tools are not yet available.

\section{Hybrid inclusions: solutions and asymptotic solutions}
\label{section:hybrid-systems}

This section focuses the discussion on arcs generated by 
hybrid systems modeled by hybrid inclusions. 
Subsection \ref{section: back to earth} gives the basic definitions and
assumptions.
In Subsection \ref{section: ex-plaining} some of the more important
results from earlier two sections are restated in the current, somewhat
simpler, setting. 
Subsection \ref{section: enough with simplicity} introduces new concepts
of perturbed solutions to a hybrid inclusion, which are then shown to be
asymptotic solutions in the sense the previous section.

\subsection{Basics}
\label{section: back to earth}

A symbolic representation of a hybrid inclusion that models a hybrid system is
\begin{equation}
\label{eq:hybrid}
\begin{array}{lcl}
x \in C & & \dot{x} \in F(x)  \\[3pt]
x \in D & & x^{+} \in G(x). 
\end{array}
\end{equation}
The set $C \subset \Re^{n}$ is called the flow set, 
the set-valued mapping $F:\Re^{n} \rightrightarrows \Re^{n}$ is called
the flow map, 
the set $D \subset \Re^{n}$ is called the jump
set, 
and the set-valued mapping $G:\Re^{n} \rightrightarrows \Re^{n}$
is called the jump map.

The data of \eqref{eq:hybrid} satisfies the {\em Hybrid Basic Conditions}
if
\begin{enumerate}
    \item the sets $C$ and $D$ are closed;
    \item the set-valued mappings $F$ and $G$ are outer semicontinuous and locally bounded;
    \item $F(x)$ is nonempty and convex for all $x \in C$;
    \item $G(x)$ is nonempty for all $x \in D$.
\end{enumerate}
\

Under the Hybrid Basic Conditions, the definition of a solution to 
\eqref{eq:hybrid} is as follows.
A hybrid arc $\sol:\dom\sol\to\reals^d$ is a {\em solution to \eqref{eq:hybrid}} if
$\sol(0,0)\in C\cup D$, and
\begin{itemize} 
\item[$\bullet$] 
for every $([t,t'],j)\subset\dom\sol$ with $t<t'$, $\sol(t,j)\in C$  and 
$\soldot(t,j)\in F(\sol(t,j))$ for almost all $t\in[t,t']$; 
\item[$\bullet$]
for every $(t,j),(t,j+1)\in\dom\sol$, 
$\sol(t,j)\in D$ and $\sol(t,j+1)\in G(\sol(t,j))$.
\end{itemize}
The set of all solutions 
of \eqref{eq:hybrid} is denoted $\mathcal{S}_{\mathcal{H}}$.
To reflect what was done earlier with $\L$-curves, each solution
could be called a $\S_\H$-curve, in fact a $\S_\H$-arc, but this terminology
is not used in what follows. 
Similarly, an asymptotic solution refers below to what could be called
an asymptotic $\S_\H$-arc. All properties, for example invariance,
are understood as being with respect to $\S_\H$, unless explicitly stated
otherwise. 

The following result collects the now standard facts about $\S_\H$ which confirm
that the results developed in earlier sections apply to solutions
to \eqref{eq:hybrid}. 
Nominal well-posedness is concluded in \cite[Theorem 6.8]{Goebel12a},
and the subsequent two properties are obvious. 
The local boundedness conclusion follows from \cite[Proposition 6.13]{Goebel12a}.

\begin{proposition}
\label{proposition: CFDG is like L}
Under the Hybrid Basic Conditions, the set $\S_\H$ is
nominally well posed, tail invariant, and closed under concatenation.
If, additionally, every solution $\sol\in\S_\H$ is bounded or complete, 
then $\S_\H$ is locally bounded. 
\end{proposition}

It is often useful to consider ``inflations'' of the data
of the hybrid inclusion \eqref{eq:hybrid} through a small parameter
$\varepsilon>0$. 
The $\varepsilon$-inflation of \eqref{eq:hybrid} is the system represented symbolically
as
\begin{equation}
\label{eq:inflated-system}
\begin{array}{lcl}
x \in C_{\varepsilon} & & \dot{x} \in F_{\varepsilon}(x) \\[3pt]
x \in D_{\varepsilon} & & x^{+} \in G_{\varepsilon}(x) 
\end{array}
\end{equation}
where
\begin{equation}
\label{eq:inflated-data}
\begin{array}{rcl}
    C_{\varepsilon} & := & C+ \varepsilon \Ball \\[2pt]
    F_{\varepsilon}(x) & := & 
    \overline{\mbox{\rm co}} F((x+ \varepsilon \Ball) \cap C) + \varepsilon \Ball \qquad \forall x \in C_{\varepsilon} \\[2pt]
    D_{\varepsilon} & := & D + \varepsilon \Ball \\[2pt]
    G_{\varepsilon}(x) & := & G((x+ \varepsilon \Ball) \cap D) + \varepsilon \Ball  \qquad \forall x \in D_{\varepsilon} .
\end{array}
\end{equation}
In these definitions, $\varepsilon \Ball$ denotes the
closed ball of radius $\varepsilon$.
Among other uses, such inflations are featured in the definition of 
semi-global practical robustness of asymptotic stability, while similar inflations
where $\varepsilon$ is allowed to depend on the state $x$ appear in the
concept of robustness of asymptotic stability for \eqref{eq:hybrid}; 
see \cite[Chapter 7]{Goebel12a}. 
The key to the mentioned robustness is that, roughly, solutions to
\eqref{eq:inflated-data} with small enough $\varepsilon$ are close to solutions
to \eqref{eq:hybrid}. This general principle is used in the results
of Subsection \ref{section: enough with simplicity}.

\subsection{Earlier results applied to hybrid inclusions}
\label{section: ex-plaining}

Some of the earlier results, obtained in the setting of $\L$-curves, 
are now restated and simplified in the setting of hybrid inclusions.
The four corollaries below all apply to the hybrid inclusion \eqref{eq:hybrid} and,
throughout this subsection, \eqref{eq:hybrid} is assumed to satisfy 
the Hybrid Basic Conditions.
The corollaries are facilitated by Proposition \ref{proposition: CFDG is like L}.
Internal chain transitivity has not been studied in this setting before.

Proposition \ref{proposition:brilliant-Benaim} generates this corollary:

\begin{corollary}
\label{corollary: obvious 1}
Let $K\subset\reals^n$ be compact and internally chain recurrent.
Then $K$ is weakly forward and weakly backward invariant. 
\end{corollary}

Theorem \ref{theorem: almost time to watch EPL} implies this corollary:

\begin{corollary}
\label{corollary: obvious 2}
Let $X\subset\reals^n$ be nonempty and compact. 
Then, the chain recurrent set $\R(X)$ of $X$ is internally chain recurrent. 
\end{corollary}

Recall the definition of an asymptotic $\mathcal{L}$-curve in Section
\ref{section: asymptotic L-curves}.
A hybrid curve is said to be an {\em asymptotic solution of \eqref{eq:hybrid}}
if it is an asymptotic $\mathcal{L}$-curve with $\mathcal{L}=\mathcal{S}_{\mathcal{H}}$.
Under the Hybrid Basic Conditions, every 
bounded and complete
solution to \eqref{eq:hybrid} is,
of course, an asymptotic solution to \eqref{eq:hybrid} as well. 
Corollary \ref{corollary:a confirmation provided}
leads to this corollary: 

\begin{corollary}
\label{corollary: obvious 3}
For every asymptotic solution to \eqref{eq:hybrid}, in particular for
every bounded and complete solution $\sol\in\S_\H$, 
the set $\omega(\sol)$ is internally chain transitive. 
\end{corollary}

Note that Corollary \ref{corollary: obvious 3}, when combined with
Corollary \ref{corollary: obvious 1}, implies the known invariance properties
of the $\omega$-limit of a solution to \eqref{eq:hybrid}.

Because the set $\S_\H$ is closed under concatenation and $K:=\omega(\sol)$ is
weakly invariant, the internal chain transitivity of $K$ can be 
verified by chains where the solutions $\sol_0,\sol_1,\dots,\sol_{k-1}$ forming them
are complete and have ranges in $K$. Let $\S^c_\H$ be the set of all complete
$\sol\in\S_\H$. Let $\H|_K$ be the hybrid inclusion \eqref{eq:hybrid} restricted to $K$,
i.e., be the hybrid inclusion with data $(C\cap K,F,D \cap K,G^K)$, where
$G^K(x):=G(x)\cap K$ for every $x\in\reals^d$. This data satisfies the Hybrid
Basic Conditions if $(C,F,D,G)$ does.
The data ensures that the range of every $\sol\in\S_{\H|_K}$ is contained in $K$. 
In turn, the conclusion of Corollary \ref{corollary: obvious 3} can be strengthened:

\begin{corollary}
\label{corollary: obvious 4}
For every asymptotic solution to \eqref{eq:hybrid}, in particular for
every bounded and complete solution $\sol\in\S_\H$, 
the set $\omega(\sol)$ is internally $\S^c_{\H|_K}$-chain transitive. 
\end{corollary}

Given a set $S\subset\reals^d$, let $\S_\H(S)$ be the set 
of all $\sol\in\S_\H$ with $\sol(0,0)\in S$.
For \eqref{eq:hybrid}, 
a set $A\subset\reals^d$ is i) {\em Lyapunov stable} if for every $\ve>0$ there exists 
$\delta>0$ such that every $\sol\in\S_\H(A+\delta\ball)$ satisfies
$\sol(t,j)\in A+\ve\ball$ for all $(t,j)\in\dom\sol$; ii) 
{\em attractive} if there exists $\delta>0$ such that every
$\sol\in\S_\H(A+\delta\ball)$ is bounded and, if it is complete, then it converges to $A$, 
and iii) {\em asymptotically stable} if it is both Lyapunov stable and attractive.
The {\em basin of attraction} of an asymptotically stable $A$, denoted
$\B(A)$, is the set of all $x\in\reals^d$ for which every $\sol\in\Sol(x)$
converges to $A$. If $\B(A)=\reals^d$, the set $A$ is {\em globally asymptotically
stable}. A particular kind of an asymptotically stable set is provided by a (compact)
attractor for \eqref{eq:hybrid}. See \cite{Goebel23SCL}, \cite{GoebelTeel25SCLsubmission}
for definitions and some details, in the context of \eqref{eq:hybrid}. 

From Corollary \eqref{corollary: obvious 3} one can deduce: 
\begin{corollary}
\label{corollary: obvious 5}
Let $A\subset\reals^n$ be a nonempty and compact asymptotically stable 
for \eqref{eq:hybrid} set. 
Let $\sol$ be an asymptotic solution to \eqref{eq:hybrid}; in particular,
$\sol$ may be a bounded and complete solution $\sol\in\S_\H$. If
\begin{equation}
\omega(\sol)\cap\B(A)\neq \emptyset, 
\end{equation} 
which holds in particular when there exists a compact set $K\subset\B(A)$ 
such that $\rge\sol\subset K$,
then the set $\omega(\sol)$ is an internally chain transitive subset of $\R(A)$. 
\end{corollary}

Above, $\R(A)$ is the chain recurrent set of $A$, defined in
Section \ref{section: chain recurrent sets}.
To see that the result holds, take any $x\in\omega(\sol)$.
The uniform attractivity of $A$ from compact subsets of $\B(A)$
(see \cite[Lemma 7.8]{Goebel12a} for this property or \cite[Lemma 10]{Goebel23SCL} 
for a similar use of it) implies the existence of $\tau,\ve,\delta>0$
such that, for the compact set $U:=A+\delta\ball+\ve\ball\subset\B(A)$, one has
$x\not\in U$, $\sola(t,j)\in A+\delta\ball$ for 
every solution $\sola$ with $\sola(0,0)\in\{x\}\cup U$ 
and $(t,j)\in\dom\sola$ such that
$t+j\geq\tau$. Then there is no $(\tau,\ve)$-chain from $x$ to $x$, or from
$A\subset U$ to any point outside of $U$. 
This ensures that any internally chain transitive set is either a subset of $A$
or disjoint from $\B(A)$.

\subsection{Families of asymptotic solutions}
\label{section: enough with simplicity}

Conditions for a hybrid arc to be
an asymptotic solution of \eqref{eq:hybrid}
can be given in terms of the inflated
system \eqref{eq:inflated-system}-\eqref{eq:inflated-data}.
These conditions involve asking that
the arbitrarily late tails of the hybrid
arc, perhaps plus appropriately
vanishing perturbations,
are solutions of \eqref{eq:inflated-system}-\eqref{eq:inflated-data}
for arbitrarily small $\varepsilon >0$.
These conditions are made precise next.

A complete, bounded hybrid arc $\sola$ is said to be a {\em simple
vanishing inflation solution} to \eqref{eq:hybrid} if,
for each $\varepsilon>0$, there exists $(s,i) \in \dom \sola$
such that the tail $\sola^{(s,i)}$ defined in
\eqref{tail} is a solution of \eqref{eq:inflated-system}-\eqref{eq:inflated-data}.
A more general concept is given as follows.
Given a complete, bounded hybrid arc $\sola$,
let $\mathcal{P}(\sola)$ 
denote the class of functions
$\chi:\dom \chi \rightarrow \Re$ satisfying
\begin{itemize}
\item[i.) ]
$\dom \chi  = \left\{ (s,i,t,j)\in\reals^4\, :\, 
(s,i),(t,j) \in \dom \sola, \, s+i\leq t+j
\right\}$,
\item[ii.)]
If $(s,i,t_1,j),(s,i,t_2,j)\in\dom\chi$
with $t_{2} > t_{1}$ then 
$t\mapsto\chi(s,i,t,j)$
is absolutely continuous on $[t_{1},t_{2}]$,
\item[iii.)]
for each $\tau>0$,
\begin{equation}
\label{eq:chi-prop}
\lim_{s+i \rightarrow \infty}
\sup 
\left\{ |\chi(s,i,t,j)|\, :\,
\begin{array}{l} (s,i,t,j) \in \dom \chi \\[4pt]
0\leq t+j-(s+i)\leq\tau
\end{array}
\right\} = 0 .
\end{equation}
\end{itemize}
A complete, bounded hybrid arc $\sola$ is said to be
a {\em vanishing inflation solution} to \eqref{eq:hybrid}
if there exists $\chi \in \mathcal{P}(\sola)$ and, for each $\varepsilon>0$ and $T>0$ 
there exists
$\kappa>0$ such that, for each
$(s,i) \in \dom \sola$ with $s+i \geq \kappa$,
$\pcTsi:\dom\pcTsi\to\reals^d$, given by
\begin{subequations}
\label{eq:sol-chi-def}
\begin{align}
&  \dom\pcTsi : =
\dom \sola^{(s,i)} \cap \left\{(t,j)
\in \Re_{\geq 0} \times \ZZ_{\geq 0} :
t+j \leq T \right\}  \\
&     \pcTsi(t,j):=\sola^{(s,i)}(t,j)+\chi(s,i,s+t,i+j)
\qquad \forall (t,j) \in \dom \pcTsi,
\end{align}
\end{subequations}
is a solution of 
(\ref{eq:inflated-system})-\eqref{eq:inflated-data}.
A simple vanishing inflation solution corresponds to a vanishing inflation
solution with $\chi \equiv 0$.

\begin{proposition}
\label{proposition:PAS-2-pseudo}
Under the Hybrid Basic Conditions, each (bounded, complete)
vanishing inflation solution to (\ref{eq:hybrid}) is an asymptotic solution to (\ref{eq:hybrid}).
\end{proposition}

\begin{proof}
Let $\sola$ be a vanishing inflation solution, with an associated $\chi$.
Pick $T>0$ and $\ve\in(0,1)$. 
Let $K\subset\reals^d$ be a compact set containing the range of $\sola$. 
By \cite[Proposition 6.34]{Goebel12a}, and because $\ve<1$, 
there exists $\rho\in(0,\ve)$ such that, for every solution $\sola'$ to $(C_{\rho},F_{\rho},D_{\rho},G_{\rho})$
with $\sola\in K+\rho\ball$ there exists a solution $\sol'$ to $(C,F,D,G)$ such that 
\begin{align*}
\gph_{t+j\leq T}\sola'\subset\gph_{t+j\leq T+1}\sol'+.5\ve\ball, \\
\gph_{t+j\leq T}\sol'\subset\gph_{t+j\leq T+1}\sola'+.5\ve\ball.
\end{align*}
Let $\kappa>0$ be such that, for every
$(s,i) \in \dom \sola$ with $s+i \geq \kappa$,
\begin{itemize} 
\item[(i)]
$\pcTsi$ is a solution to $(C_\rho,F_\rho,D_\rho,G_\rho)$ 
(which is done using the definition of a vanishing inflation solution), and
\item[(ii)] 
$|\chi(s,i,t,j)|\leq 0.5\ve$ for all $(s,i,t,j)\in\dom\chi$ with $0\leq t+j-(s+i)\leq T+1$  (this is done using $\tau:=T+1$ in \eqref{eq:chi-prop}). 
\end{itemize} 
Then, for every $(s,i)\in\dom\sola$ with $s+i\geq\kappa$
there exists a solution $\sol$ to $(C,F,D,G)$ such that 
\begin{align*}
\gph_{t+j\leq T}\pcTsi & \subset\gph_{t+j\leq T+1}\sol+.5\ve\ball, \\
\gph_{t+j\leq T}\sol & \subset\gph_{t+j\leq T+1}\pcTsi+.5\ve\ball, 
\end{align*}
and, consequently,
\begin{align*}
\gph_{t+j\leq T}\sola^{(s,i)} & 
\subset\gph_{t+j\leq T+1}\sol+\ve\ball\subset\gph\sol+\ve\ball, \\
\gph_{t+j\leq T}\sol & 
\subset\gph_{t+j\leq T+1}\sola^{(s,i)}+\ve\ball\subset\gph\sola^{(s,i)}+\ve\ball. 
\end{align*}
Lemma \ref{lemma: unwind your mind} implies that $\sol$ is an
asymptotic solution to (\ref{eq:hybrid}).
\end{proof}

The next corollary combines
the result of Proposition \ref{proposition:PAS-2-pseudo} 
with Corollaries \ref{corollary: obvious 1}
and \ref{corollary: obvious 4}.

\begin{corollary}
\label{corollary:itc-for-perturbed-asymptotic-solutions}
Under the Hybrid Basic Conditions, the $\omega$-limit set of every
(bounded, complete) vanishing inflation solution
of (\ref{eq:hybrid}) is nonempty, compact, weakly invariant, and internally chain transitive, and thus
internally $\S^c_{\H|_K}$-chain transitive.
\end{corollary}

\section{Asymptotic simulations of a hybrid system}
\label{section:asymptotic simulations of a hybrid system}

\subsection{Hybrid sequences}

A set $E \subset \Re^2$ is a {\em compact, hybrid sequence domain} 
if
$$
E=\left(\{k_0,\dots,k_1\},0\vphantom{\frac{}{}}\right)
\cup\left(\{k_1,\dots,k_2\},1\vphantom{\frac{}{}}\right)
\cup\dots\cup\left(\{k_{J-1},\dots,k_J\},J-1\vphantom{\frac{}{}}\right)
$$
where $J \in \nats$ and $0 =k_{0} \leq k_{1} \leq \cdots \leq k_J$ form
a finite sequence of integers. 
It is a {\em hybrid sequence domain} if it is the union
of a nondecreasing sequence of compact hybrid sequence domains.
A mapping $\sol:\dom \sol \rightarrow \Re^{n}$ is a
{\em hybrid sequence} if its domain is a
hybrid sequence domain. It is {\em complete} if its
domain is unbounded. 
It is {\em complete in the $k$
direction} if the set of all $k$'s defining
its domain is unbounded.
It is {\em complete in the $j$
direction} if the set of all $j$'s defining
its domain is unbounded.

Given a hybrid sequence $\sol$, define
\begin{subequations}
\label{eq:the-overline-functions}
\begin{align}
\overline{\jmath}_{k}  & :=
 \inf \left\{ j \in \ZZ_{\geq 0}: (k+1,j) \in \dom \sol  \right\} \\
\overline{k}_{j} & : =  \inf \left\{ k \in \ZZ_{\geq 0}: (k,j+1) \in \dom \sol  \right\}
\end{align}
\end{subequations}
with the convention that $\inf(\emptyset) = \infty$. 
For each $k \in \ZZ_{\geq 0}$,
$(k,j),(k+1,j) \in \dom \sol$ if and only if
$j=\overline{\jmath}_{k}$; also $(k,\overline{\jmath}_{k}) \in \dom \sol$
if and only if $(k+1,\overline{\jmath}_{k}) \in \dom \sol$.
If $\phi$ is complete in the $k$ direction then,
for all $k \in \ZZ_{\geq 0}$,
$\overline{\jmath}_{k} < \infty$.
Similarly, 
for each $j \in \ZZ_{\geq 0}$,
$(k,j),(k,j+1) \in \dom \sol$ if and only if
$k=\overline{k}_{j}$; also $(\overline{k}_{j},j) \in \dom \sol$
if and only if $(\overline{k}_{j},j+1) \in \dom \sol$.
If $\phi$ is complete in the $j$ direction then, for all $j \in \ZZ_{\geq 0}$,
$\overline{k}_{j} < \infty$. 


\subsection{Asymptotic simulations}
\label{subsection: asymptotic simulations}

In this section, the concept of an
asymptotic simulation of a hybrid inclusion
is defined.
An asymptotic simulation includes a
hybrid sequence and
a sequence of positive step sizes
that play a role in approximating the flows
of the hybrid inclusion. A sequence
of step sizes $\left\{ h_{k} \right\}_{k=1}^{\infty}$ is said to
be {\em admissible} if $h_{k}>0$ for each $k \in \ZZ_{\geq 1}$ and the sequence converges to
zero but is not summable.
Given an admissible sequence of step sizes $\left\{ h_{k} \right\}_{k=1}^{\infty}$, 
define
\begin{subequations}
\begin{align}
    \tau_{k} & : = \sum_{i=0}^{k-1} h_{i+1} \qquad \forall k \in \ZZ_{\geq 0} \label{eq:tau-k} \\
m(t) &  : = \max \left\{k \in \ZZ_{\geq 0}: \tau_{k} \leq t \right\} 
\qquad \forall t \in \Re_{\geq 0} .
\label{BarcaChelsea}
\end{align}
\end{subequations}

A hybrid sequence $\sol$ and a sequence
of admissible step sizes $\left\{ h_{k} \right\}_{k=1}^{\infty}$ form
a {\em asymptotic simulation of \eqref{eq:hybrid}} 
if $\phi$ is bounded and complete and the following properties hold:
\begin{enumerate}
    \item If $\sol$ is complete in $k$ direction
    then there exists a bounded sequence $\left\{ f_{k} \right\}_{k=0}^{\infty}$ of vectors in $\Re^{n}$ 
    such that
    \begin{align}
    \label{eq:CT-nonstandard-dataless}
    \limsup_{k \rightarrow \infty} \left(\sol(k,\overline{\jmath}_{k}),f_{k} \right) \subset
    \mbox{\rm graph}(F) \cap (C \times \Re^{n})
    \end{align}
    and,
    with the definition
    \begin{align}
  \label{eq:widehat-f-def}
\widehat{f}_{k+1}:=\frac{\sol(k+1,\overline{\jmath}_{k})-\sol(k,\overline{\jmath}_{k})}{h_{k+1}} \qquad \forall k \in \ZZ_{\geq 0} ,
    \end{align}
    the following limit holds for some
    $T>0$:
    \begin{align}
\label{eq:Benaim-bound-from-variance-dataless}  
    \lim_{k \rightarrow \infty}
    \sup_{k+1 \leq n \leq m(\tau_{k}+T)}
    \left| \sum_{i=k}^{n-1} h_{i+1} (\widehat{f}_{i+1}-f_{i} ) \right| = 0.
    \end{align}
    \item If $\sol$ is complete in the $j$ direction then
\begin{align}
\label{eq:DT-nonstandard-dataless}
    \limsup_{j \rightarrow \infty} \left( \sol(\overline{k}_{j},j),\sol(\overline{k}_{j},j+1) \right) \subset
    \mbox{\rm graph}(G) \cap (D \times \Re^{n}) .
\end{align}
\end{enumerate}

The $\limsup$ in the conditions above indicates the outer limit of the sequences, that is, the set of all
accumulation points of the sequences. 
See \cite{RockafellarWets98}
or the introductory \cite{Goebel24book}.

The next subsection, Section \ref{section:asymptotic-simulators}, contains a discussion of simulator
models that produce asymptotic simulations. In the meantime,
the main result regarding asymptotic simulations of \eqref{eq:hybrid} 
is stated below. 
Its proof is facilitated by first compressing the domain
of the hybrid sequence $\phi$
of an asymptotic simulation to reflect the step sizes
$\left\{ h_{k} \right\}_{i=1}^{\infty}$,
and then by interpolating the resulting domain
to obtain a mapping defined on a hybrid time domain. This is done
in Subsection \ref{subsection: compressing}, and then the proof
of the theorem is executed in Subsection 
\ref{subsection:omega-limits-for-asymptotic-simulations-dataless}. 

\begin{theorem}
\label{theorem: for-stochastic-simulators}
Pose the Hybrid Basic Conditions.
Let $(\sol,\left\{ h_{k} \right\}_{k=1}^{\infty})$ be an asymptotic 
simulation of \eqref{eq:hybrid}.
Then the $\omega$-limit set of $\sol$ is nonempty, compact, and weakly invariant and internally chain transitive for \eqref{eq:hybrid}.
\end{theorem}

Theorem \ref{theorem: for-stochastic-simulators} for the special
case where it is possible to take $f_{k} = \widehat{f}_{k+1}$ for all $k \in \ZZ_{\geq 0}$ in the characterization of an asymptotic simulation is
explained and illustrated via examples in \cite{GoebelTeel2026CDC}.
The full generality of Theorem \ref{theorem: for-stochastic-simulators} is
used to develop results for asymptotic simulations of two time-scale
hybrid systems in \cite{Crisafulli2026}.

To conclude this subsection, it is observed that if 
\eqref{eq:Benaim-bound-from-variance-dataless}
holds for some $T>0$, then it holds for all $T>0$. 
Indeed, and more generally, 
let $T>0$ and a sequence $\{a_i\}_{i=0}^{\infty}$ of vectors in $\reals^n$ be such that 
$$\lim_{k\to\infty} \sup_{k+1\leq n\leq m(\tau_k+T)} \left|\sum_{i=k}^{n-1} a_i\right|=0.$$
Pick $\ve>0$. There exists $k^* \in \ZZ_{\geq 0}$ such that, for every integer $k\geq k^*$, 
$$\sup_{k+1\leq n\leq m(\tau_k+T)} \left|\sum_{i=k}^{n-1} a_i\right|\leq\ve/2.$$
In particular, this holds for every $k\geq k^*$ when $k$ above is replaced by $m(\tau_k+T)$,
because
$m(\tau_k+T)\geq k$ for every $k$. 
Now, notice that, because $h_k\to 0$,
$$m(\tau_k+1.5T)\leq m\left(\tau_{m(\tau_k+T)}+T\right)$$
for all large enough $k$, in particular for all $k\geq k^*$ subject to increasing $k^*$. 
Indeed, because the sequence $\{h_k\}$ converges to $0$ but is not summable,
$\tau_{m(\tau_k+T)}$ gets arbitrarily close to $\tau_k+T$ for large $k$, 
and then $\tau_k+1.5T<\tau_{m(\tau_k+T)}+T$. 
Then, for all $k\geq k^*$,
\begin{align*}
\sup_{k+1\leq n\leq m(\tau_k+1.5T)} \left|\sum_{i=k}^{n-1} a_i\right|
\leq
\sup_{k+1\leq n\leq m\left(\tau_{m(\tau_k+T)}+T\right)} \left|\sum_{i=k}^{n-1} a_i\right|.
\end{align*}
The quantity on the right above equals
\begin{align*}
\max \left\{
\sup_{k+1\leq n\leq m(\tau_k+T)} \left|\sum_{i=k}^{n-1} a_i\right|,
\sup_{m(\tau_k+1)+1\leq n\leq m\left(\tau_{m(\tau_k+T)}+T\right)} \left|\sum_{i=k}^{n-1} a_i\right|
\right\}
\end{align*}
The first of the two suprema above is $\leq\ve/2$. 
In the second supremum, 
\begin{align*}
\left|\sum_{i=k}^{n-1} a_i\right| 
= 
\left|\sum_{i=k}^{m(\tau_k+1)-1} a_i+\sum_{i=m(\tau_k+1)}^{n-1} a_i\right| 
\leq 
\left|\sum_{i=k}^{m(\tau_k+1)-1} a_i\right| + \left|\sum_{i=m(\tau_k+1)}^{n-1} a_i\right| 
\end{align*}
and each of the two summands in the last expression is $\leq\ve/2$. Thus, the second
supremum is $\leq\ve$. Consequently, for all $k\geq k^*$,
\begin{align*}
\sup_{k+1\leq n\leq m(\tau_k+1.5T)} \left|\sum_{i=k}^{n-1} a_i\right|
\leq\ve.
\end{align*}
Thus, if \eqref{eq:Benaim-bound-from-variance-dataless}
holds for some $T>0$ then it holds if $T$ is replaced by $1.5T$. 
Clearly, if \eqref{eq:Benaim-bound-from-variance-dataless}
holds for some $T>0$ then it holds if $T$ is replaced by anything smaller. 
This, and iterating the ``$1.5$ argument'' shows that 
if \eqref{eq:Benaim-bound-from-variance-dataless}
holds for some $T>0$ then it holds for all $T>0$.


\subsection{Asymptotic simulators}
\label{section:asymptotic-simulators}

This section briefly discusses when non-stochastic and then stochastic simulators
of a hybrid inclusion, studied before in
\cite{GoebelTeel25SCLsubmission} and \cite{TeelSanfeliceGoebel25ARC}, 
produce asymptotic simulations. 

\subsubsection{Non-stochastic simulators}

In \cite[eq. (11)]{GoebelTeel25SCLsubmission}, simulators for \eqref{eq:hybrid} of the form
\begin{subequations}
    \begin{align}
     & x \in C  \quad \quad \ x^{+} -x   \in  h^{+} \widehat{F}(x,h^{+}) \\
     & x \in D  \qquad \qquad    x^{+}   \in  G(x) 
    \end{align}
\end{subequations}
are considered. A 
simulation is a hybrid sequence $\phi: \dom \phi \rightarrow \Re^{n}$ such that
\begin{enumerate}
    \item
$(k,j),(k+1,j) \in \dom \phi$ implies $\phi(k,j) \in C$ and $\phi(k+1,j)-\phi(k,j) \in h_{k+1} \widehat{F}(\phi(k,j),h_{k+1})$, and
\item
$(k,j),(k,j+1) \in \dom \phi$ implies $\phi(k,j) \in D$ and
$\phi(k,j+1) \in G(\phi(k,j))$.
\end{enumerate}
In \cite[eq. (12)]{GoebelTeel25SCLsubmission}, it is assumed that there exists a continuous, non-decreasing function $\gamma:\Re_{\geq 0} \rightarrow
\Re_{\geq 0}$ such that
\begin{align}
\label{eq:from-SCL}
    \widehat{F}(x,h_{k+1}) \subset F(x) + h_{k+1} \gamma(|x|) \Ball
    \qquad \forall (x,k) \in C \times \ZZ_{\geq 0} .
\end{align}
From this condition, it can be verified that, when $\left\{ h_{k} \right\}_{k=1}^{\infty}$ is admissible and $\phi$ is bounded and complete, the
pair $(\phi,\left\{ h_{k} \right\}_{k=1}^{\infty})$ is an asymptotic simulation in
the sense of the previous section. In particular,
the sequence $\left\{ f_{k} \right\}_{k=0}^{\infty}$ with
$f_{k}:= \widehat{f}_{k+1}$ for all $k \in \ZZ_{\geq 0}$, where $\widehat{f}_{k+1}$ is defined in (\ref{eq:widehat-f-def}), satisfies
(\ref{eq:CT-nonstandard-dataless}) and (\ref{eq:Benaim-bound-from-variance-dataless}).

Work on model-based, non-stochastic, non-asymptotic
simulators for hybrid systems includes
\cite{SanfeliceTeel10} and the references therein.

\subsubsection{Stochastic simulators}

In \cite[eq. (2)]{TeelSanfeliceGoebel25ARC}, stochastic simulators for
\eqref{eq:hybrid} of the form
\begin{subequations}
\label{eq:TSG25ARC-simulators}
    \begin{align}
     & x \in C  \quad \quad \ x^{+} -x   \in  h^{+} \widehat{F}(x,y^{+},h^{+}) \\
     & x \in D  \qquad \qquad    x^{+}   \in  G(x) 
    \end{align}
\end{subequations}
are considered,
where $y^{+}$ is a placeholder for a sequence of random variables
$\left\{ \mathbf{y}_{k} \right\}_{k=1}^{\infty}$ defined on a probability space
$(\Omega,\mathcal{F},\mathbb{P})$. As indicated in
\cite[Section 4]{TeelSanfeliceGoebel25ARC}, a
solution $\mathbf{x}$, defined on $\Omega$,
has certain measurability properties adapted to the minimal filtration of
$\left\{ \mathbf{y}_{k} \right\}_{k=1}^{\infty}$
and
is such that, for almost every $\omega \in \Omega$, the sample path
$\phi:= \mathbf{x}(\omega)$ is a hybrid sequence such that
\begin{enumerate}
    \item
$(k,j),(k+1,j) \in \dom \phi$ implies $\phi(k,j) \in C$ and $\phi(k+1,j)-\phi(k,j) \in h_{k+1} \widehat{F}(\phi(k,j),\mathbf{y}_{k+1}(\omega),h_{k+1})$, and
\item
$(k,j),(k,j+1) \in \dom \phi$ implies $\phi(k,j) \in D$ and
$\phi(k,j+1) \in G(\phi(k,j))$.
\end{enumerate}
When discussing the asymptotic behavior of the sample paths of $\mathbf{x}$
in \cite[Section 8]{TeelSanfeliceGoebel25ARC}, a containment like
(\ref{eq:from-SCL}) above is imposed but it is expressed in terms of a
conditional expectation. Further, the step sizes are assumed to be square summable
and a bound on a conditional variance is imposed
to guarantee the ability to construct a sequence $\left\{ f_{k} \right\}_{k=0}^{\infty}$ so that (\ref{eq:CT-nonstandard-dataless}) and (\ref{eq:Benaim-bound-from-variance-dataless}) hold for almost every sample path, i.e., so that, together with the step sizes,
almost every bounded, complete sample path is an asymptotic simulation in the sense
of the previous section.
See \cite[Lemma 8.3]{TeelSanfeliceGoebel25ARC}. Stochastic simulations
are also discussed in Section \ref{section:stochastic-asymptotic-simulations} below, where
Proposition \ref{prop:appendix-flows} resembles and extends \cite[Lemma 8.3]{TeelSanfeliceGoebel25ARC}.


\subsection{Compressing and interpolating
an asymptotic simulation's time domain}
\label{subsection: compressing}

To make connection to earlier results in the paper,
the time domain of a simulation component $\sol$ is compressed
and interpolated based on the associated step sizes $\left\{ h_{k} \right\}_{k=1}^{\infty}$
and their running sums
$\left\{ \tau_{k} \right\}_{k=1}^{\infty}$
defined in \eqref{eq:tau-k}.
Given a simulation pair $\left(\phi,\left\{ h_{k} \right\}_{k=1}^{\infty}\right)$
define the domain-compressed version $\widetilde{\phi}$ by
\begin{subequations}
\label{eq:set-up-simulator-solution-to-be-pseudo}
    \begin{align}
        \dom \widetilde{\phi} & : = \bigcup_{(k,j) \in \footnotesize \dom \sol} (\tau_{k},j) \\
        \widetilde{\phi}(\tau_{k},j) & := \sol(k,j) \qquad \forall (k,j) \in \dom \sol
    \end{align}
\end{subequations}
and the
interpolated hybrid arc
$\psi$ as follows:
\begin{itemize}
\item its domain is 
\begin{align}
\label{eq:psi-dom-def}
\dom\sola:=\bigcup \left\{[t,t']\times\{j\}\, :\, t,t',j\ \mbox{such that}\ (t,j), (t',j)\in\dom \widetilde{\sol} \right\} ;
\end{align}
    \item if $(k,j),(k+1,j) \in \dom \sol$ then
    \begin{align}
        \label{eq:psi-flows}
    \psi(t,j) := \sol(k,j) +
    \frac{t-\tau_{k}}{h_{k+1}} \left( \phi(k+1,j)- \phi(k,j) \right) \qquad \forall t \in [\tau_{k},\tau_{k+1}] ;
    \end{align}
    \item if $(k,j),(k,j+1) \in \dom \phi$ then
    \begin{align}
        \label{eq:psi-jumps}
    \psi(\tau_{k},j) := \phi(k,j) \qquad \psi(\tau_{k},j+1) := \phi(k,j+1) . 
    \end{align}
\end{itemize}

\begin{lemma}
\label{lemma:about-simulations}
Suppose that $\left(\phi,\left\{h_{k} \right\}_{k=1}^{\infty} \right)$ is
an asymptotic simulation of \eqref{eq:hybrid} and $\phi$ is complete
in the $k$ direction.
Then
    $\lim_{k \rightarrow \infty} | \phi(k+1,\overline{\jmath}_{k})
    -\phi(k,\overline{\jmath}_{k})| = 0$.
\end{lemma}
\begin{proof}
Since $\left(\phi,\left\{h_{k} \right\}_{k=1}^{\infty} \right)$ is
an asymptotic simulation of \eqref{eq:hybrid} and $\sol$ is complete in the $k$ direction, 
the bounded sequence $\left\{ f_{k} \right\}_{k=0}^{\infty}$ in \eqref{eq:CT-nonstandard-dataless} 
in the definition of an asymptotic simulation is well defined.
Because $\left\{h_{k} \right\}_{k=1}^{\infty}$ is admissible, 
$h_{k}=0$ converge to $0$ and then $\lim_{k \rightarrow \infty} |h_{k+1} f_{k}|=0$.
This, the definition of $\widehat{f}_{k+1}$ in \eqref{eq:widehat-f-def}, 
and \eqref{eq:Benaim-bound-from-variance-dataless} with $n=k+1$ ensure that
\begin{align*}
        \limsup_{k \rightarrow \infty} |\phi(k+1,\overline{\jmath}_{k})
    -\phi(k,\overline{\jmath}_{k})|
& = 
\limsup_{k \rightarrow \infty} |h_{k+1} \widehat{f}_{k+1}| \\
& \leq  
\limsup_{k \rightarrow \infty} | h_{k+1} (\widehat{f}_{k+1} - f_{k})| + \limsup_{k \rightarrow \infty} | h_{k+1} f_{k} | \\ 
& =  0.
\end{align*}
This establishes the result.
\end{proof}

\begin{lemma}
\label{lemma:same-tails}
Suppose $\left(\phi,\left\{h_{k} \right\}_{k=1}^{\infty} \right)$ is
an asymptotic simulation of \eqref{eq:hybrid}. 
Let $\widetilde{\sol}$ be the compressed-domain version 
\eqref{eq:set-up-simulator-solution-to-be-pseudo} of $\phi$ 
and  $\sola$ be the interpolated mapping 
\eqref{eq:psi-dom-def}-\eqref{eq:psi-jumps}.
Then, the graphical
limits of the tails of $\sola$ agree with the graphical limits 
of the tails of $\widetilde{\sol}$.
\end{lemma}

\begin{proof}
If $\phi$ is not complete in the $k$ direction, then $\sola$ and $\widetilde{\sol}$ 
are equal to one another for all large enough $t+j$, and the conclusion of the lemma
is clear. Suppose that $\phi$ is complete in the $k$ direction, so that 
$\dom\widetilde{\sol}$ and $\dom\sola$ are unbounded in the $k$ and $t$ directions, respectively.
For $(t_n.j_n)\in\dom\widetilde{\sol}$, let $(t'_n.j_n)$ 
be the next point in $\dom\widetilde{\sol}$, if it exists. More precisely, if
$(t_n,j_n)$ is such that $t_n=\tau_k$ and $(t_n,j_n+1)\not\in\dom\widetilde{\sol}$,
then let $t'_n=\tau_{k+1}$. 
Lemma \ref{lemma:about-simulations}, the decrease of step sizes to $0$, and the nature 
of the interpolation imply that, for every $t_n,s_n,j_n$ such that
$(t_n.j_n)\in\dom\widetilde{\sol}$, $(s_n,j_n)\in\dom\sola$, 
and $t_n\leq s_n<t'_n$,
\begin{equation}
\label{Barca} 
\gph\widetilde{\sol}^{(t'_n,j_n)} \subset 
\gph\widetilde{\sol}^{(t_n,j_n)} \subset 
\gph\sola^{(t_n,j_n)}
\end{equation}
and for every $\ve>0$ and all large enough $t_n+j_n$,
\begin{equation}
\label{Chelsa}
\gph\sola^{(s_n,j_n)} \subset \gph\wisol^{(t_n,j_n)}+\ve\ball 
\subset \gph\wisol^{(t'_n,j_n)}+2\ve\ball.
\end{equation}
(The inclusions involving $t'_n$ in \eqref{Barca} and \eqref{Chelsa} should be ignored
if $t'_n$ does not exist.)
Suppose that $t_n+j_n\to\infty$ and $\wisol^{(t_n,j_n)}$ converge graphically to
some mapping $\chi$. Then, given any $\ve>0$,
\begin{eqnarray*}
\gph\chi & = & 
\lim_{n\to\infty} \gph \wisol^{(t_n,j_n)} =
\liminf_{n\to\infty} \gph \wisol^{(t_n,j_n)} \subset 
\liminf_{n\to\infty} \gph \sola^{(t_n,j_n)} \\
& \subset & 
\limsup_{n\to\infty} \gph\sola^{(t_n,j_n)} 
\subset \limsup_{n\to\infty} \gph\wisol^{(t_n,j_n)} +\ve\ball 
=\gph\chi+\ve\ball.
\end{eqnarray*}
Because $\ve>0$ is arbitrary and all the limits above are closed sets,
$$\liminf_{n\to\infty} \gph \sola^{(t_n,j_n)}=
\limsup_{n\to\infty} \gph\sola^{(t_n,j_n)}=\gph\chi,$$
which implies that $\sola^{(t_n,j_n)}$ converge graphically to $\chi$.
On the other hand, suppose that $s_n+j_n\to\infty$, $\sola^{(s_n,j_n)}$ converge graphically
to some mapping $\eta$. Pick $(t_n,j_n)\in\dom\wisol$ so that
$t_n=s_n$, if possible, or $t_n<s_n<t'_n$. Then, 
subject to replacing the only mention of $t'_n$ below by $t_n$ in the case where
$t'_n$ does not exist and $t_n=s_n$, one has
\begin{eqnarray*}
\gph\eta
& = & 
\lim_{n\to\infty} \gph\sola^{(s_n,j_n)} =
\liminf_{n\to\infty} \gph\sola^{(s_n,j_n)} \subset
\liminf_{n\to\infty} \gph\wisol^{(t_n,j_n)} +\ve\ball  \\ 
& \subset & 
\limsup_{n\to\infty} \gph\wisol^{(t'_n,j_n)} + 2\ve\ball \subset 
\limsup_{n\to\infty} \gph\sola^{(s_n,j_n)} + 2\ve\ball \\ 
& = & \gph \eta+2\ve\ball
\end{eqnarray*}
for any $\ve>0$. Again, this implies that $\wisol^{(t_n,j_n)}$ converge graphically
to $\eta$.
\end{proof}

\subsection{Properties of the omega-limit set of
an asymptotic simulation}
\label{subsection:omega-limits-for-asymptotic-simulations-dataless}

The omega-limit set of an asymptotic simulation $\phi$ is investigated
by studying the omega-limit set of its interpolated mapping $\psi$.
The definition of the interpolated mapping's evolution during 
flows is given in (\ref{eq:psi-flows}) and its time derivative
in the interval $(\tau_{k},\tau_{k+1})$
is seen to be equal to the value $\widehat{f}_{k+1}$ 
defined in \eqref{eq:widehat-f-def}. On the other hand, 
when $\phi$ is complete in the $k$ direction, it is
the sequence $\left\{ \left( \phi(k,\overline{\jmath}_{k}),f_{k} \right) \right\}_{k=0}^{\infty}$, rather than the sequence $\left\{ \left( \phi(k,\overline{\jmath}_{k}),\widehat{f}_{k+1} \right) \right\}_{k=0}^{\infty}$, whose
outer limit is assumed to belong to $\mbox{\rm graph}(F) \cap (C \times \Re^{n})$,
as in (\ref{eq:CT-nonstandard-dataless}).
To account for the potential mismatch between the values of
$\widehat{f}_{k+1}$ and $f_{k}$, which is characterized by
\eqref{eq:Benaim-bound-from-variance-dataless}, an appropriate function 
$\chi$
is added to the tails of the interpolated mapping $\psi$. 
That function is
\begin{align}
\label{eq:chi-def}
    \chi(s,i,t,j): = \int_{s}^{t} U(r) dr  
    \quad \forall (s,i),(t,j) \in \dom \psi, \, s+i \leq t+j
\end{align}
where $U:\Re_{\geq 0} \rightarrow \Re^{n}$ is defined as
\begin{align}
\label{eq:U-def}
    U(t) : = f_{k} - \widehat{f}_{k+1}  \qquad \forall t \in [\tau_{k},\tau_{k+1})
    \qquad \forall k \in \ZZ_{\geq 0} .
\end{align}
The domain of $U$ is $\Re_{\geq 0}$ due to the definition of $\tau_k$ in
(\ref{eq:tau-k}) and the fact that the step sizes
are not summable. The following lemma asserts that
$\chi \in \mathcal{P}(\psi)$.

\begin{lemma}
\label{lemma:chi-is-in-calP}
Let $(\sol,\left\{ h_{k} \right\}_{k=1}^{\infty})$ be an asymptotic simulation of \eqref{eq:hybrid} and let $\psi$ be its interpolated mapping.
Then the function $\chi$ defined in (\ref{eq:chi-def})-(\ref{eq:U-def}) satisfies
    $\chi \in \mathcal{P}(\psi)$.
\end{lemma}
\begin{proof}
That the first item in the characterization of $\mathcal{P}(\psi)$ holds for $\chi$
is evident by the definition of $\chi$. That the second item holds follows from $U(\cdot)$ being piecewise constant and thus locally integrable.
That the condition
\eqref{eq:Benaim-bound-from-variance-dataless} implies that the third item in the definition of 
$\mathcal{P}(\psi)$ holds, i.e., for every $\tau>0$, \eqref{eq:chi-prop} holds,
follows as asserted in the proof of \cite[Proposition 1.3]{Benaimetal2005}. 

In particular, let $T>0$
and let $s^{*} \in \Re_{\geq 0}$ be sufficiently large such that
$m(\tau_{m(s)+1}+T)+1 \leq m(\tau_{m(s)}+2T)$ for all $s \geq s^{*}$.
Now consider $(s,i,t,j) \in \dom \chi$ with $s \geq s^{*}$, 
$s+i \leq t+j$, so that $s \leq t$, and $t \leq s+T$, and observe that
\begin{align*}
 & |\chi(s,i,t,j)| = \left|  \int_{s}^{t} U(r) dr \right|
  \leq h_{m(s)+1} | f_{m(s)} - \widehat{f}_{m(s)+1} | + \\
  & \qquad \qquad \qquad h_{m(t)+1} | f_{m(t)}- \widehat{f}_{m(t)+1} |
  + \left| \sum_{i=m(s)+1}^{m(t)-1} h_{i+1} (f_{i} - \widehat{f}_{i+1} ) \right| 
\\
  & \leq 3 \left(  \sup_{\begin{array}{c} m(s) \leq k \leq m(t), \\
  k+1 \leq n \leq m(t)+1
  \end{array}}
  \left| \sum_{i=k}^{n-1} h_{i+1} \left( f_{i} - \widehat{f}_{i+1} \right) \right| \right) 
\end{align*}
Because $m(t)\leq m(s+T) \leq m(\tau_{m(s)+1}+T)\leq m(\tau_{m(s)}+2T)-1$ and then
$m(t)+1\leq m(\tau_{m(s)}+2T)\leq m(\tau_{k}+2T)$ for $k\geq m(s)$,
\eqref{eq:Benaim-bound-from-variance-dataless} with $T$ replaced by $2T$
implies
\eqref{eq:chi-prop} since $m(s) \rightarrow \infty$ as $s \rightarrow \infty$.
\end{proof}

\begin{lemma}
\label{lemma: for-stochastic-simulators}
Pose the Hybrid Basic Conditions and let $(\sol,\left\{ h_{k} \right\}_{k=1}^{\infty})$ be an asymptotic simulation of \eqref{eq:hybrid}.
Then, $\sol$'s compressed-domain version $\widetilde{\sol}$ defined by
\eqref{eq:set-up-simulator-solution-to-be-pseudo}
is an asymptotic solution of \eqref{eq:hybrid}
and the interpolated mapping $\psi$ defined in
\eqref{eq:psi-dom-def}-\eqref{eq:psi-jumps} is a vanishing inflation 
solution and an asymptotic solution of \eqref{eq:hybrid}.
\end{lemma}

\begin{proof}
First consider
the case where $\sol$ is not complete in the $k$ direction. Let $\overline{k}:=\max \left\{ k \in \ZZ_{\geq 0} : \exists j \in \ZZ_{\geq 0} \ \mbox{\rm s.t.} \ (k,j) \in \dom \sol \right\}$.
Using that $\sol$ is complete, and thus
complete in the $j$ direction,
let $j_{1}^{*} \in \ZZ_{\geq 0}$ be sufficiently large
such that $(\overline{k},j) \in \dom \sol$
for all $j \in \ZZ_{\geq j_{1}^{*}}$. 
Using the second property characterizing
an asymptotic simulation of
(\ref{eq:hybrid}) together with \cite[Theorem 4.10(b)]{RockafellarWets98} and
the boundedness of $\phi$, for each $\varepsilon>0$,
there exists $j_{2}^{*} \in \ZZ_{\geq j_{1}^{*}}$
such that, for all $j \in \ZZ_{\geq j_{2}^{*}}$,
\begin{align}
\label{eq:52-graph-plus-eps-contain}
    \left( \phi(\overline{k},j),\phi(\overline{k},j+1) \right) \in \left( \mbox{\rm graph}(G) \cap
    (D \times \Re^{n}) \right) + \varepsilon \Ball .
\end{align}
It follows from (\ref{eq:52-graph-plus-eps-contain}) that,
for all $j \in \ZZ_{\geq j_{2}^{*}}$,
\begin{align}
    \sol(\overline{k},j) \in D + \varepsilon \Ball
    , \quad \sol(\overline{k},j+1)
    \in G\left( \left( \sol(\overline{k},j) + \varepsilon \Ball \right) \cap D  \right) 
    + \varepsilon \Ball .
\end{align}
Then, from the definition of compressed-domain version 
$\widetilde{\phi}$ of $\phi$ in (\ref{eq:set-up-simulator-solution-to-be-pseudo}), 
and the definition of the interpolated mapping $\psi$
in \eqref{eq:psi-dom-def}-\eqref{eq:psi-jumps}, both
$\widetilde{\phi}$ and $\psi$ are (simple) vanishing inflation
solutions to (\ref{eq:hybrid})
and, using Proposition \ref{proposition:PAS-2-pseudo},  
also are asymptotic solutions of (\ref{eq:hybrid}), i.e., the conclusions 
of the lemma hold.

Henceforth, suppose $\sol$ is complete in the $k$ direction.
According to (\ref{eq:psi-flows}) and (\ref{eq:widehat-f-def}), for each
$k \in \ZZ_{\geq 0}$,
\begin{align}
\label{eq:psi-dot1-stochastic}
    \dot{\sola}(t,j) = \frac{\sol(k+1,j) - \sol(k,j)}{h_{k+1}}
    =: \widehat{f}_{k+1}  \qquad \forall
    t \in (\tau_{k},\tau_{k+1}) .
\end{align}
Let the bounded sequence $\left\{ f_{k} \right\}_{k=0}^{\infty}$ come
from the first item that characterizes
an asymptotic simulation
of (\ref{eq:hybrid}), i.e., be such that \eqref{eq:CT-nonstandard-dataless} holds.
Define $\chi \in \mathcal{P}(\sola)$ as in
(\ref{eq:chi-def})-(\ref{eq:U-def}).
Let $T>0$ and $\varepsilon>0$ be given. 
Let $k^{*} \in \ZZ_{\geq 0}$
be sufficiently large such that,
for all $k \in \ZZ_{\geq k^{*}}$,
the following bounds hold:
\begin{subequations}
\label{eq:two-quarter-epsilon-bounds}
\begin{align}
\label{eq:55a}
    h_{k+1} |f_{k}| & \leq 0.25 \varepsilon  \\
        \label{eq:integral-bound}
    \sup_{s \geq \tau_{k^{*}}, v \in [0,T] }   \left| \int_{s}^{s+v} U(r) dr \right| 
    & \leq 0.25 \varepsilon 
\end{align}
\end{subequations}
where $U(\cdot)$ was defined in \eqref{eq:U-def},
and the following properties hold: 
\begin{enumerate}[a)]
    \item
    \label{item:first-of-two-items}
if $(k,j),(k+1,j) \in \dom \sol$ then
\begin{align}
\label{eq:CT-in-proof}
\sol(k,j) \in C + 0.25 \varepsilon \Ball , \quad
f_{k} \in \overline{\mbox{\rm co}} F\left( \left( \sol(k,j)+0.25 \varepsilon \Ball \right) \cap C \right) + 
\varepsilon \Ball
\end{align}
\item
   \label{item:second-of-two-items}
if $(k,j),(k,j+1) \in \dom \sol$ then
\begin{align}
\label{eq:DT-in-proof}
\sol(k,j) \in D +  0.75 \varepsilon \Ball , \quad
\sol(k,j+1) \in G\left( \left( \sol(k,j)+ 0.75 \varepsilon \Ball \right) \cap D \right) + 
\varepsilon \Ball .
\end{align}
\end{enumerate}
There exists $k^{*} \in \ZZ_{\geq 0}$
such that (\ref{eq:55a}) holds for all $k \geq k^{*}$
since the sequence
$\left\{ f_{k} \right\}_{k=0}^{\infty}$ is bounded while
the sequence $\left\{ h_{k} \right\}_{k=1}^{\infty}$ converges
to zero.
There exists $k^{*} \in \ZZ_{\geq 0}$
such that (\ref{eq:integral-bound}) holds due to
Lemma \ref{lemma:chi-is-in-calP}, which asserts that $\chi \in \mathcal{P}(\psi)$,
and the third item in the characterization of $\mathcal{P}(\psi)$.
There exists $k^{*} \in \ZZ_{\geq 0}$ such that item \ref{item:first-of-two-items}) above holds for all $k \geq k^{*}$
due to \eqref{eq:CT-nonstandard-dataless}, the boundedness of
the sequence $\left(\sol(k,\overline{\jmath}_{k}),f_{k} \right)$,
and \cite[Theorem 4.10(b)]{RockafellarWets98}. 
Similarly, there
exists $k^{*} \in \ZZ_{\geq 0}$ such that item
\ref{item:second-of-two-items}) above holds for
all $k \geq k^{*}$ due to \eqref{eq:DT-nonstandard-dataless}, boundedness of
the sequence $\left( \sol(\overline{k}_{j},j),\sol(\overline{k}_{j},j+1) \right)$,
and \cite[Theorem 4.10(b)]{RockafellarWets98} since
either 1) the domain of $\phi$ is bounded in the $j$-direction, in which
case, for $k$ large enough, there does not exist $j \in \ZZ_{\geq 0}$ such that
$(k,j),(k,j+1) \in \dom \phi$, or 2) the domain of $\phi$ is unbounded in
the $j$-direction, in which case, for each $j^{*} \in \ZZ_{\geq 0}$ there
exists $k^{*} \in \ZZ_{\geq 0}$ such that $k \geq k^{*}$ and $(k,j) \in \dom \phi$
implies $j \geq j^{*}$.

It follows from 
the definition $\sola$ in \eqref{eq:psi-dom-def}-\eqref{eq:psi-jumps}, the definition of $U(\cdot)$
in \eqref{eq:U-def}, the relation (\ref{eq:psi-dot1-stochastic}),
and the bounds \eqref{eq:two-quarter-epsilon-bounds}
that, for all $(k,j) \in \dom \sol$ such that
$k \in \ZZ_{\geq k^{*}}$, and all $t \in [\tau_{k},\tau_{k+1}]$,
\begin{align}
\label{eq:new-bound-added}
\left| \sola(t,j)
    - \sol(k,j) \right| \leq 
    h_{k+1} |f_{k}|
    +h_{k+1} |\widehat{f}_{k+1} - f_{k} |
    \leq 0.5 \varepsilon .
\end{align}
Define $\underline{\jmath}^{*}:=\min \left\{j : (\tau_{k^{*}},j) \in
\dom \psi \right\}$, which is well defined since
$\phi$ is complete in the $k$ direction.
Define $\kappa:=\tau_{k^{*}}+\underline{\jmath}^{*}$,
let $(s,i) \in \dom \psi$ satisfy $s+i \geq \kappa$,
and recall the definition of $\pcTsi$ established
in (\ref{eq:sol-chi-def}), \eqref{eq:chi-def}, and \eqref{eq:U-def}.
Given $(t,j) \in \dom \pcTsi$, let
$k \in \ZZ_{\geq k^{*}}$ be such that $s+t  \in 
\left[ \tau_{k}, \tau_{k+1} \right]$.
It follows using
(\ref{eq:integral-bound}), (\ref{eq:new-bound-added}),
and the definition
of $\pcTsi$ that
\begin{align}
  &  |\pcTsi(t,j) - \sol(k,i+j)|
    \nonumber \\
 & \leq  | \pcTsi(t,j) - \sola(s+t,i+j)
    + \sola(s+t,i+j)
    - \sol(k,i+j)| 
   \nonumber \\
& \leq  \sup_{s \geq \tau_{k^{*}}, 0 \leq v \leq T} 
    \left| \int_{s}^{s+v}
    U(r) dr \right| + 0.5 \varepsilon
     \leq 0.75 \varepsilon . 
    \label{eq:sola-compared-to-sol-stochastic}
\end{align}
Then, combining \eqref{eq:psi-dot1-stochastic}, item
\ref{item:first-of-two-items}) above,
the definition of $\pcTsi$,
    and (\ref{eq:sola-compared-to-sol-stochastic}), it follows that if $(t_{1},j),(t_{2},j) \in \dom \pcTsi$ with $t_{1} < t_{2}$ then, for almost all $t \in [t_{1},t_{2}]$,
    \begin{subequations}
        \begin{align}
            \pcTsi(t,j) & \in C + \varepsilon  \Ball  \\
            \label{eq:sola-flows-stochastic}
            \pcTsidot(t,j)
            & \in \overline{\mbox{\rm co}} F\left( \left( \pcTsi(t,j) +
            \varepsilon  \Ball \right) \cap C \right)
            + \varepsilon \Ball .
        \end{align}
    \end{subequations}
    Also, using the definitions of $\psi$ and
    $\pcTsi$, (\ref{eq:integral-bound}), and
    item \ref{item:second-of-two-items}) above,
    if $(t,j),(t,j+1) \in \dom \pcTsi$ then
    \begin{subequations}
        \begin{align}
            \pcTsi(t,j) & \in D + \varepsilon  \Ball  \\
            \pcTsi(t,j+1)
            & \in G\left( \left( \pcTsi(t,j) + \varepsilon \Ball \right) \cap D \right) + \varepsilon \Ball .
        \end{align}
    \end{subequations}
It follows that $\sola$ is a vanishing inflation solution
of (\ref{eq:hybrid}).
It follows from Proposition \ref{proposition:PAS-2-pseudo} that $\sola$ is a 
asymptotic solution of (\ref{eq:hybrid}).
Also, from \eqref{eq:new-bound-added} and the fact that $\varepsilon>0$ is arbitrary, 
it follows from Lemma \ref{lemma:same-tails} that
the graphical limits of the tails of $\sola$ are the same as the graphical limits
of the tails of $\widetilde{\sol}$.
Thus, $\widetilde{\sol}$ is an asymptotic solution of (\ref{eq:hybrid}).
\end{proof}

Lemma \ref{lemma: for-stochastic-simulators} above, combined with Corollary 
\ref{corollary:itc-for-perturbed-asymptotic-solutions},
proves Theorem \ref{theorem: for-stochastic-simulators}.

\section{Stochastic signals as
asymptotic simulations}
\label{section:stochastic-asymptotic-simulations}

In this section, we consider a collection of hybrid sequences
defined on a probability space and
give statistical conditions under
which almost every sample path is
an asymptotic simulation of the
hybrid inclusion \eqref{eq:hybrid}.

First, adaptedness of the collection
is formulated. Then, conditional mean
and variance-like conditions are
provided to ensure the desired properties.

\subsection{Adaptedness}

Let $(\Omega,\mathcal{F},\mathbb{P})$ be a probability space
and let
$\left\{ \mathcal{H}_{k,j} \right\}_{(k,j) \in \ZZ_{\geq 0} \times \ZZ_{\geq 0}}$ 
be a two-parameter filtration of the probability space; 
that is, for each $(k,j) \in \ZZ_{\geq 0} \times \ZZ_{\geq 0}$,
$\mathcal{H}_{k,j} \subset \mathcal{F}$ is a $\sigma$-field
\cite[Section 1.2, Def. 1]{FristedtGray97},
which contains the empty set, and
    \begin{align}
    \label{eq:filtration}  
       \mathcal{H}_{k,j} \subset \left( \mathcal{H}_{k+1,j} 
        \cap  \mathcal{H}_{k,j+1} \right) .
    \end{align}
    
Let $\mathcal{X}$ be the collection of set-valued mappings
$\phi:\Re^{2} \rightrightarrows \Re^{n}$ 
such that $\mbox{\rm graph}(\phi)$ is
non-empty and closed.
    Let $\mathbf{x}:\Omega \rightarrow \mathcal{X}$ be such that,
    for almost every $\omega \in \Omega$, $\mathbf{x}(\omega)$
    is a hybrid sequence, also called a {\em sample path}. 
To save on notation, we write $\mathbf{x}(k,j)$ to represent 
$\mathbf{x}(\omega)(k,j)$ when $(k,j) \in \dom \mathbf{x}(\omega)$.

Extending the definitions in (\ref{eq:the-overline-functions}) to the stochastic case,
for each $(k,j) \in \ZZ_{\geq 0} \times \ZZ_{\geq 0}$, define the functions
\begin{subequations}
\begin{align}
\overline{\jmath}_{k}(\omega)  & :=
 \inf \left\{ j \in \ZZ_{\geq 0}: (k+1,j) \in \dom \mathbf{x}(\omega)  \right\}   \qquad \forall \omega \in \Omega \\
\overline{k}_{j}(\omega) & : =  \inf \left\{ k \in \ZZ_{\geq 0}: (k,j+1) \in \dom \mathbf{x}(\omega)  \right\}   \qquad \forall \omega \in \Omega
\end{align}
\end{subequations}
again with the convention that $\inf( \emptyset)= \infty$.

The mapping $\mathbf{x}$ is said to be {\em adapted}
if, for each $(k,j) \in \ZZ_{\geq 0} \times \ZZ_{\geq 0}$, the following properties hold:
\begin{enumerate}
\item 
the sets
$$\Omjk:=\left\{ \omega \in \Omega : \overline{\jmath}_{k}(\omega) \leq j \right\}, \qquad
\Omkj:=
\left\{ \omega \in \Omega : \overline{k}_{j}(\omega) \leq k \right\}$$
are elements of $\mathcal{H}_{k,j}$;
\item 
the mapping
\begin{align}
\label{eq:old64-new63}
    \omega \mapsto \begin{cases} \mathbf{x}(k,j)  & \mbox{\rm if} \ (k,j) \in \dom \mathbf{x}(\omega) \\
    \emptyset & \mbox{\rm otherwise} \end{cases}
\end{align}
is $\mathcal{H}_{k,j}$-measurable.
\end{enumerate}
If $\mathbf{x}$ is adapted, it is called
a {\em stochastic asymptotic simulation candidate}.

The first condition above is equivalent to the sets
$$\Omjke:=\left\{ \omega \in \Omega : \overline{\jmath}_{k}(\omega) = j \right\}, \qquad
\Omkje:=
\left\{ \omega \in \Omega : \overline{k}_{j}(\omega) = k \right\}$$
belonging to $\mathcal{H}_{j,k}$ for each $(k,j)
\in \ZZ_{\geq 0} \times \ZZ_{\geq 0}$. This is
because
\begin{align*}
& \Omjk = \cup_{i \in \left\{0,\ldots,j \right\}}
\Omjkei \qquad \mbox{and} \qquad
\Omjke = \Omjk \cap (\Omjkminusone)^{c} \\
& \Omkj = \cup_{i \in \left\{0,\ldots,k \right\}}
\Omkjei \qquad \mbox{and} \qquad
\Omkje = \Omkj \cap (\Omkjminusone)^{c} ,
\end{align*}
(\ref{eq:filtration}) holds, and
$\mathcal{H}_{k,j}$ is closed under
complements, countable unions and countable intersections.

The second condition above is equivalent to
$\mathcal{H}_{k,j}$-measurability of
the mapping
\begin{align}
\label{eq:old63-new64}
\omega \mapsto \mbox{\rm graph}(\mathbf{x}(\omega)) \cap \left( \left\{(k,j)\right\} \times \Re^{n} \right) ,
\end{align}
i.e., the mapping
\begin{align}
\label{eq:new64-bis}
    \omega \mapsto \begin{cases} (k,j,\mathbf{x}(k,j))  & \mbox{\rm if} \ (k,j) \in \dom \mathbf{x}(\omega) \\
    \emptyset & \mbox{\rm otherwise.} 
    \end{cases}
\end{align}
That $\mathcal{H}_{k,j}$-measurability of
the mapping in (\ref{eq:new64-bis}) implies
$\mathcal{H}_{k,j}$-measurability of
the mapping in (\ref{eq:old64-new63}) follows from
\cite[Theorem 14.13(a)]{RockafellarWets98} and the fact that the projection
from $\Re^{2} \times \Re^{n}$ to $\Re^{n}$ is single-valued and continuous. The opposition direction
follows from \cite[Theorem 14.13(b)]{RockafellarWets98}
where, in that result,
$S$ is the mapping in
(\ref{eq:old64-new63}), $m=n+2$, and 
$M(\omega,\cdot)= (k,j,\cdot)$.

In words, a collection of hybrid
sequences is adapted if, at the time $(k,j)$, 1) when there is a decision to be made between incrementing $k$ or incrementing $j$ (or neither),
the decision must depend only on the information contained in $\mathcal{H}_{j,k}$, and 2) 
when there is a decision to be made about the value of the state after an increment of $k$ or $j$, if $k$ is incremented then the next value of the state is $\mathcal{H}_{k+1,j}$-measurable while if $j$ is incremented then the next
value of the state is $\mathcal{H}_{k,j+1}$-measurable.

In the context of the
stochastic simulators in (\ref{eq:TSG25ARC-simulators}), where 
the filtration $\mathcal{H}_{k,j}$ is independent of $j$ for each $k$,
conditions for the existence of adapted solutions with complete sample paths 
are provided in \cite[Proposition 4.1]{TeelSanfeliceGoebel25ARC}.
The following example illustrates adaptedness
in the general $2$-parameter filtration case.
The case where $\mathcal{H}_{k,j}$ is independent of $j$ for each $k$ is discussed after this example.

\begin{example}
\label{example:hsd1}
{\rm 
    Let $\left\{ \mathbf{y}_{k} \right\}_{k=1}^{\infty}$ and
$\left\{ \mathbf{z}_{j} \right\}_{j=1}^{\infty}$ be mutually independent sequences
of iid random variables defined on the probability
space $(\Omega,\mathcal{F},\mathbb{P})$ taking values in $\left\{0,1 \right\}$.
Let $\left\{ \mathcal{F}_{k} \right\}_{k=0}^{\infty}$ and $\left\{ \mathcal{G}_{j} \right\}_{j=0}^{\infty}$
be the minimal filtrations of $\left\{ \mathbf{y}_{k} \right\}_{k=1}^{\infty}$ and 
$\left\{ \mathbf{z}_{j} \right\}_{j=1}^{\infty}$, respectively, with
$\mathcal{F}_{0} = \mathcal{G}_{0} : = \left\{ \emptyset, \Omega \right\}$.
Define $\mathcal{H}_{k,j}:= \mathcal{F}_{k} \vee \mathcal{G}_{j}$, i.e., the $\sigma$-algebra
generated by $\left\{ \mathbf{y}_{1},\ldots,\mathbf{y}_{k},\mathbf{z}_{1},\ldots,\mathbf{z}_{j} \right\}$.
For each $\omega \in \Omega$, create
the hybrid sequence $\mathbf{x}(\omega)$
as follows:
\begin{enumerate}
\item
Let $\mathbf{x}(0,0)=1$.
\item
For $(k,j) \in \dom \mathbf{x}(\omega)$, 
\begin{enumerate}
\item if $\mathbf{x}(k,j)=1$
then $(k+1,j) \in \dom \mathbf{x}(\omega)$ and $\mathbf{x}(k+1,j)=\mathbf{y}_{k+1}$; else,
\item
if $\mathbf{x}(k,j)=0$ then $(k,j+1) \in \dom \mathbf{x}(\omega)$ and $\mathbf{x}(k,j+1)=\mathbf{z}_{j+1}$.
\end{enumerate}
\end{enumerate}
The first observation to make in verifying the adaptedness property
for $\mathbf{x}$
is that,
with the definition $\mathbf{y}_{0}(\omega)=1$
for all $\omega \in \Omega$,
if $(k,j) \in \dom \mathbf{x}(\omega)$
then 
\begin{subequations}
\label{eq:formulas}
\begin{align}
    k & = \sum_{i=0}^{k-1} \mathbf{y}_{i}(\omega) +
\sum_{i=1}^{j-1} \mathbf{z}_{i}(\omega)  \\
    j & =
     \sum_{i=0}^{k-1} (1-\mathbf{y}_{i}(\omega)) +
\sum_{i=1}^{j-1} (1-\mathbf{z}_{i}(\omega)) .
\end{align}
\end{subequations}
It is clear that
these equalities holds for $(k,j)=(0,0)$,
which belongs to
$\dom \mathbf{x}(\omega)$ for all $\omega \in \Omega$. 
Fix $\omega \in \Omega$ and, for the purposes
of induction, suppose the pair
$(k,j) \in \ZZ_{\geq 0} \times \ZZ_{\geq 0}$
is such that $(k,j) \in \mathbf{x}(\omega)$, $k+j \geq 1$, and the equalities
in (\ref{eq:formulas})
hold for $(k-1,j)$ if $(k-1,j) \in \dom \mathbf{x}(\omega)$ or for $(k,j-1)$ if $(k,j-1) \in \mathbf{x}(\omega)$; since $\dom \mathbf{x}(\omega)$ is a hybrid sequence domain, one
and only one of $(k,j-1)$ and $(k-1,j)$ belongs
to $\dom \mathbf{x}(\omega)$. 
The value $(k,j)$ belongs to $\mathbf{x}(\omega)$ 
either because
$(k-1,j) \in \dom \mathbf{x}(\omega)$ and
$\mathbf{y}_{k}(\omega)=1$ or $(k,j-1) \in \dom \mathbf{x}(\omega)$ and $\mathbf{z}_{j}(\omega)=0$.
It thus follows that the equalities hold for
$(k,j) \in \dom \mathbf{x}(\omega)$.

Next, it follows that
\begin{subequations}
\begin{align}
     \overline{\jmath}_{k}(\omega) & = 
     \inf \left\{ j \in \ZZ_{\geq 0} :
     \sum_{i=0}^{k} (1-\mathbf{y}_{i}(\omega))
     + \sum_{i=1}^{j-1} (1-\mathbf{z}_{i}(\omega)) = j \right\} \\
          \overline{k}_{j}(\omega) & = 
     \inf \left\{ k \in \ZZ_{\geq 0} :
     \sum_{i=0}^{k-1} \mathbf{y}_{i}(\omega)
     + \sum_{i=1}^{j} \mathbf{z}_{i}(\omega) = k \right\} .
\end{align}
\end{subequations}
Then
\begin{subequations}
\begin{align} 
    \Omjk & = 
     \left\{ \omega \in \Omega :
     \sum_{i=0}^{k} (1-\mathbf{y}_{i}(\omega))
     + \sum_{i=1}^{j-1} (1-\mathbf{z}_{i}(\omega)) \leq j \right\} \\
     \Omkj
     & = 
     \left\{ \omega \in \Omega :
     \sum_{i=0}^{k-1} \mathbf{y}_{i}(\omega)
     + \sum_{i=1}^{j} \mathbf{z}_{i}(\omega) \leq k \right\} .
\end{align}
\end{subequations}
It is now evident that $\Omega_{\overline{\jmath}_{k} \leq j}$
and $\Omega_{\overline{k}_{j} \leq k}$ belong
to $\mathcal{H}_{k,j}$.

Finally, note that the mapping in (\ref{eq:old64-new63}) is given by
\begin{align}
\begin{cases}
   1 &  \mbox{\rm if} \ (k,j) = (k,\overline{\jmath}_{k}) \\
  0 & \mbox{\rm if} \ (k,j) = (\overline{k}_{j},j) \\
  \emptyset & \mbox{\rm otherwise,} 
  \end{cases}
\end{align}
which is $\mathcal{H}_{k,j}$-measurable
since $\Omjke$ and $\Omkje$ belong to $\mathcal{H}_{k,j}$.
\null \hfill \null $\blacksquare$}
\end{example}

When
$\mathcal{H}_{k,j}$ is independent of $j$ for each $k$ (or independent of $k$ for each $j$), the 
first condition for adaptedness can be simplified.
Indeed, in the proposition below, the condition
$\Omjkinf  := \left\{\omega \in \Omega:  \overline{\jmath}_{k}(\omega) < \infty \right\} \in \mathcal{H}_{k}$
for each $k \in \ZZ_{\geq 0}$ is satisfied
when every sample path is complete in the $k$-direction, as in this case $\Omega_{\overline{\jmath}_{k} < \infty}=\Omega$ for each $k$.

\begin{proposition}
    Suppose, for each $(k,j) \in \ZZ_{\geq 0}
    \times \ZZ_{\geq 0}$,
    $\mathcal{H}_{k,j}=\mathcal{H}_{k,0}=:
    \mathcal{H}_{k}$. 
    The mapping $\mathbf{x}$ is
    adapted if and only if, for each $k \in \ZZ_{\geq 0}$,
    $\Omjkinf :=\left\{\omega \in \Omega:  \overline{\jmath}_{k}(\omega) < \infty \right\} \in \mathcal{H}_{k}$ and
    the mapping
    \begin{align}
    \label{eq:66}
        \omega \mapsto \mbox{\rm graph}(\mathbf{x}(\omega)) \cap \left( \left\{k \right\}
        \times \Re \times \Re^{n} \right)
    \end{align}
    is $\mathcal{H}_{k}$-measurable.
\end{proposition}
\begin{proof}
$\Longrightarrow$
By assumption,
$\Omjk
\in \mathcal{H}_{k,j}=\mathcal{H}_{k}$ for each
$(k,j) \in \ZZ_{\geq 0} \times \ZZ_{\geq 0}$.
Taking the countable union over $j \in \ZZ_{\geq 0}$ gives that $\Omjkinf \in \mathcal{H}_{k}$. Also, the mapping in (\ref{eq:66})
is the countable union over $j \in \ZZ_{\geq 0}$ of the mappings in (\ref{eq:old63-new64}), each of which
is $\mathcal{H}_{k}$-measurable. Thus, according
to \cite[Proposition 14.11(b)]{RockafellarWets98}, the mapping in
(\ref{eq:66}) is $\mathcal{H}_{k}$-measurable.

$\Longleftarrow$ The mapping
in (\ref{eq:old63-new64}) can be viewed as
the intersection of the mapping in
(\ref{eq:66}) with the constant, closed-valued mapping
$ \omega \mapsto (k,j) \times \Re^{n}$; thus, according
to \cite[Proposition 14.11(a)]{RockafellarWets98},
the mapping in (\ref{eq:old63-new64}) is
$\mathcal{H}_{k}$-measurable.

To show that $\Omjk
 \in \mathcal{H}_{k}$ for all $(k,j) \in \ZZ_{\geq 0} \times\ZZ_{\geq 0}$, note that
\cite[Proposition 14.13(a)]{RockafellarWets98}
gives that $\mathcal{H}_{k}$-measurability of
the mapping in (\ref{eq:66}) implies $\mathcal{H}_{k}$-measurability of the mapping
$\omega \mapsto \dom \mathbf{x}(\omega) \cap 
    \left( \left\{k \right\} \times \Re \right)$
    while 
    \begin{align}
     \left\{ \omega \in \Omega :
    \dom \mathbf{x}(\omega) \cap 
    \left( \left\{k+1 \right\} \times \Re \right)
    \neq \emptyset \right\} = \Omjkinf  \in \mathcal{H}_{k} .
    \label{eq:67}
    \end{align}
Then
\begin{align}
&\Omjk  = \left\{ 
\omega\in\Omega\, :\, \dom \mathbf{x}(\omega) \cap 
\left(\left\{k \right\} \times \reals\right) \subset 
\left(k,\{0,1,2,\dots,j\}\right)
\right\} \nonumber \\
&\qquad  \qquad \qquad \qquad \qquad \cap
\left\{ 
\omega\in\Omega\, :\, \dom\mathbf{x}(\omega) 
\cap \left( \left\{k+1 \right\} \times \reals\right)\not=\emptyset
\right\} . \nonumber
    \end{align}
    The intersection belongs to $\mathcal{H}_{k}$ since
    the second set in the intersection belongs
    to $\mathcal{H}_{k}$ due to (\ref{eq:67})
    while the first set in the intersection
    belongs to $\mathcal{H}_{k}$ due to the 
    $\mathcal{H}_{k}$-measurability of the mapping
    $\omega \mapsto \dom \mathbf{x}(\omega) \cap 
    \left( \left\{k \right\} \times \Re \right)$ and
    \cite[Theorem 14.3(i)]{RockafellarWets98}.

To show that $\Omkj
\in \mathcal{H}_{k}$ for all $(k,j) \in \ZZ_{\geq 0} 
\times \ZZ_{\geq 0}$, note that
the mapping
\begin{align}
  &  \omega \mapsto  \dom \mathbf{x}(\omega)
    \cap \left( \left\{0,\ldots,k \right\} \times
    \left\{ j+1 \right\} \right) \nonumber \\
 & \qquad \qquad = \bigcup_{i \in \left\{0,\ldots,k \right\}} 
 \dom \mathbf{x}(\omega)
    \cap \left( \left\{ i \right\} \times
    \left\{ j+1 \right\} \right)
    \label{eq:last-mapping}
\end{align}
is $\mathcal{H}_{k}$-measurable since it is
the finite union of $\mathcal{H}_{k}$-measurable
mappings \cite[Proposition 14.11(b)]{RockafellarWets98}.
Then note that
\begin{align*}
\Omkj
= \left\{ \omega \in \Omega :
\dom \mathbf{x}(\omega) \cap \left( \left\{0,\ldots,k \right\} \times \left\{j+1 \right\} \right) \cap \Re^{2} \neq \emptyset
\right\}
\end{align*}
the latter set belonging to $\mathcal{H}_{k}$ due
to the $\mathcal{H}_{k}$-measurability of the mapping
in (\ref{eq:last-mapping}).
\end{proof}

\vspace{.1in}

Example \ref{example:hsd1} can be modified
readily to illustrate
the case where $\mathcal{H}_{k,j}$ is
independent of $j$, as shown below.

\begin{example}
{\rm
Consider the collection of hybrid sequences
generated in Example
\ref{example:hsd1} but where
the random variables $\left\{ \mathbf{z}_{j} \right\}_{j=1}^{\infty}$ are replaced by
constant functions $\mathbf{z}_{j}(\omega)=1$
for all $j \in \ZZ_{\geq 1}$ and all $\omega \in \Omega$.
In this case, $\mathcal{H}_{k,j} = \mathcal{F}_{k}$
for all $(k,j) \in \ZZ_{\geq 0} \times \ZZ_{\geq 0}$,
where $\left\{ \mathcal{F}_{k} \right\}$ is the minimal
filtration of $\left\{ \mathbf{y}_{k} \right\}_{k=1}^{\infty}$ with $\mathcal{F}_{0}: = \left\{ \emptyset, \Omega \right\}$.
Also, for
each $\omega \in \Omega$,
 $\mathbf{x}(\omega)$ is generated
as follows:
\begin{enumerate}
\item
Let $\mathbf{x}(0,0)=1$.
\item
For $(k,j) \in \dom \mathbf{x}(\omega)$, 
\begin{enumerate}
\item if $\mathbf{x}(k,j)=1$
then $(k+1,j) \in \dom \mathbf{x}(\omega)$ and $\mathbf{x}(k+1,j)=\mathbf{y}_{k+1}$; else,
\item
if $\mathbf{x}(k,j)=0$ then $(k,j+1) \in \dom \mathbf{x}(\omega)$ and $\mathbf{x}(k,j+1)=1$.
\end{enumerate}
\end{enumerate}
In this case, the formulas in Example
\ref{example:hsd1} simplify to
\begin{subequations}
\label{eq:formulas-bis}
\begin{align}
    k & =\sum_{i=0}^{k-1} \mathbf{y}_{i}(\omega) +
    \max \left\{0,j-1 \right\} \\
    j & =
     \sum_{i=0}^{k-1} (1-\mathbf{y}_{i}(\omega)) 
\end{align}
\end{subequations}
for $(k,j) \in \dom \mathbf{x}(\omega)$,
\begin{subequations}
\begin{align}
    \overline{k}_{j}(\omega) & =
    \inf \left\{ k \in \ZZ_{\geq 0} : \sum_{i=0}^{k-1} \mathbf{y}_{i}(\omega) + j = k \right\} \\
    \overline{\jmath}_{k}(\omega) & =
     \sum_{i=0}^{k} (1-\mathbf{y}_{i}(\omega)) 
\end{align}
\end{subequations}
and
\begin{subequations}
    \begin{align}
    \Omkj
        & =
        \left\{ \omega \in \Omega : 
        \sum_{i=0}^{k-1} \mathbf{y}_{i}(\omega) + j \leq k \right\} \\
        \Omjk & =
        \left\{ \omega \in \Omega : 
        \sum_{i=0}^{k} (1-\mathbf{y}_{i}(\omega)) \leq j \right\} .
    \end{align}
\end{subequations}
All of the required measurability properties for adaptedness then hold
by inspection or by appealing to the
observations in Example \ref{example:hsd1}.
\null \hfill \null $\blacksquare$}
\end{example}

\subsection{Constructing two
$1$-parameter filtrations used
in conditional expectations}

The next lemma supposes the first property
in the definition of adaptedness and,
from that property, constructs two $1$-parameter
filtrations that are used subsequently
in conditional expectations. The filtrations
are similar to that constructed in
\cite[p. 191]{Doob94} via ``optional times''.

\begin{lemma}
Suppose that $\Omjk
\in \mathcal{H}_{k,j}$ (respectively,
$\Omkj
\in \mathcal{H}_{k,j}$)
for each $(k,j) \in \ZZ_{\geq 0} \times \ZZ_{\geq 0}$. Define $\mathcal{H}_{k,\overline{\jmath}_{k}}$
(respectively, $\mathcal{H}_{\overline{k}_{j},j}$) to be the collection of events
$H \in \mathcal{F}$ such that
$H  \cap \Omjk
\in \mathcal{H}_{k,j}$
for all 
$j \in \ZZ_{\geq 0}$
(respectively, 
$H  \cap 
\Omkj
\in \mathcal{H}_{k,j}$
for all $k \in \ZZ_{\geq 0}$).
Then the sequence
$\left\{ \mathcal{H}_{k,\overline{\jmath}_{k}} \right\}_{k=0}^{\infty}$ (respectively,
$\left\{ \mathcal{H}_{\overline{k}_{j},j} \right\}_{j=0}^{\infty}$)
is a filtration and $\overline{\jmath}_{k}$ is $\mathcal{H}_{k,\overline{j}_{k}}$-measurable for each $k \in \ZZ_{\geq 0}$
(respectively, $\overline{k}_{j}$ is $\mathcal{H}_{\overline{k}_{j},j}$-measurable for each $j \in \ZZ_{\geq 0}$.)
\end{lemma}

{\bf Proof.} By symmetry, it is sufficient to provide the proof for the 
mapping $\overline{\jmath}_{k}$ and the sequence $\left\{  \mathcal{H}_{k,\overline{\jmath}_{k}} \right\}_{k=0}^{\infty}$.
The ideas used here are standard; see, for example, \cite[p. 191-192]{Doob94}.
It is evident that, for each $k \in \ZZ_{\geq 0}$, $\mathcal{H}_{k,\overline{\jmath}_{k}}$ includes the countable union
of its members, and it is closed under complements since $H \in \mathcal{H}_{k,\overline{\jmath}_{k}}$ implies
that 
\begin{align}
H^{c} \cap 
\Omjk
= \Omjk 
\cap \left( \vphantom{\frac{}{}} H \cap 
\Omjk \right)^{c}
\in \mathcal{H}_{k,\overline{\jmath}_{k}} .
\end{align}
Since $\mathcal{H}_{k,j} \subset \mathcal{H}_{k+1,j}$ for all $(k,j) \in \ZZ_{\geq 0} \times \ZZ_{\geq 0}$
and and
$\overline{\jmath}_{k} \leq \overline{\jmath}_{k+1}$ for all
$k \in \ZZ_{\geq 0}$, it follows that
$\left\{  \mathcal{H}_{k,\overline{\jmath}_{k}} \right\}_{k=0}^{\infty}$ is a filtration. Finally,
$\overline{\jmath}_{k}$ is $\mathcal{H}_{k,\overline{\jmath}_{k}}$-measurable for each $k \in \ZZ_{\geq 0}$ since,
for each $(k,j,i) \in \ZZ_{\geq 0} \times \ZZ_{\geq 0} \times \ZZ_{\geq 0}$,
\begin{align}
\Omega_{\overline{\jmath}_{k} \leq i} \cap
\Omjk
= \Omega_{\overline{\jmath}_{k} \leq \min \left\{ i , j \right\}}
\in \mathcal{H}_{k,j}  
\end{align}
where the last relationship has used that if $\ell \leq j$ then $\mathcal{H}_{k,\ell} \subset \mathcal{H}_{k,j}$.
\null \hfill \null $\blacksquare$

\vspace{.1in}

The two $1$-parameter filtrations defined
in the previous lemma are used in
the next two propositions.

\begin{proposition}
If $\mathbf{x}$ is adapted then:
\begin{itemize}
\item
for each $k \in \ZZ_{\geq 0}$, the mapping 
\begin{align}
    M_{k}(\omega) := \begin{cases}
         \mathbf{x}(k,\overline{\jmath}_{k}) & \mbox{\rm if} \
         (k,\overline{\jmath}_{k}) \in \dom \mathbf{x}(\omega) \\
         \emptyset & \mbox{\rm otherwise}
    \end{cases}
\end{align}
is
$\mathcal{H}_{k,\overline{\jmath}_{k}}$-measurable;
\item
for each $j \in \ZZ_{\geq 0}$, the mapping 

\begin{align}
    N_{j}(\omega) := \begin{cases}
         \mathbf{x}(\overline{k}_{j},j) & \mbox{\rm if} \
         (\overline{k}_{j},j) \in \dom \mathbf{x}(\omega) \\
         \emptyset & \mbox{\rm otherwise}
    \end{cases}
\end{align}
is
$\mathcal{H}_{\overline{k}_{j},j}$-measurable.
\end{itemize}
\end{proposition}
\begin{proof}
The result follows like
in \cite[p.192]{Doob94}.
Indeed, for the first item, from the adaptedness property,
for each open set $\mathcal{O} \subset \Re^{n}$,
\begin{align*}
& \left\{ \omega \in \Omega : \mathbf{x}(k,\overline{\jmath}_{k}) \cap \mathcal{O} \neq \emptyset \right\}
\cap \Omjk  
\\
& \qquad \qquad  =  \bigcup_{i \in \left\{0,\ldots,j \right\}} \left( \vphantom{\frac{}{}} \left\{ \omega \in \Omega : \mathbf{x}(k,i) \cap \mathcal{O}  \neq \emptyset \right\}
 \cap 
 \Omjkei
 \right) \in \mathcal{H}_{k,j} 
 \end{align*}
 where $\mathbf{x}(k,i) \cap \mathcal{O} \neq \emptyset$
 if and only if $(k,i) \in \dom \mathbf{x}(\omega)$
 and $\mathbf{x}(k,i) \in \mathcal{O}$.
The second item of the proposition follows by symmetry.
\end{proof}

\begin{proposition}
If $\mathbf{x}$ is adapted then
 \begin{itemize}
     \item for each $k \in \ZZ_{\geq 0}$,
     \begin{align}
      \label{eq:Omega_k+}
     \widetilde{\Omega}_{k}:=\left\{ \omega \in \Omega: (k+1,\overline{\jmath}_{k})
     \in \dom \mathbf{x}(\omega) \right\} \in \mathcal{H}_{k,\overline{\jmath}_{k}};
     \end{align}
     \item for each $j \in \ZZ_{\geq 0}$,
     \begin{align}
     \label{eq:Omega_j+}
     \widehat{\Omega}_{j}:=\left\{ \omega \in \Omega: (\overline{k}_{j},j+1)
     \in \dom \mathbf{x}(\omega) \right\} \in \mathcal{H}_{\overline{k}_{j},j}.
     \end{align}
 \end{itemize}
 \end{proposition}
 \begin{proof}
 The claim that $\widetilde{\Omega}_{k} \in \mathcal{H}_{k,\overline{\jmath}_{k}}$ for 
 each $k \in \ZZ_{\geq 0}$
follows by noting from the definition of $\overline{\jmath}_{k}$ that
    $\widetilde{\Omega}_{k} = \Omjkinf$
and then
    $\widetilde{\Omega}_{k} \cap 
    \Omjk
     = \Omjk 
     \in \mathcal{H}_{k,j}$.
 The reasoning for the second item is identical.
\end{proof}

 \vspace{.1in}

 Also, for future use, define
 \begin{align}
    \Omega_{k-\mbox{\rm \footnotesize complete}}:= \bigcap_{k \in \ZZ_{\geq 0}} \widetilde{\Omega}_{k},
    \quad \Omega_{j-\mbox{\rm \footnotesize complete}} : = \bigcap_{j \in \ZZ_{\geq 0}} \widehat{\Omega}_{j} .
\end{align}

\subsection{Conditional mean and variance-like conditions}

Let the mapping $\mathbf{x}:\Omega \rightarrow \mathcal{X}$ be a stochastic asymptotic simulation
candidate and
let the sequence of step sizes $\left\{ h_{k} \right\}_{k=1}^{\infty}$ be admissible.
For a set $\widecheck{\Omega} \subset \Omega$,
define
\begin{align}
\label{eq:indicator-deffy}
    \mathbf{1}_{\widecheck{\Omega}}(\omega) := 
    \begin{cases} 1 & \omega \in \widecheck{\Omega} \\
    0 & \omega \notin \widecheck{\Omega} .
\end{cases}
\end{align}
Observe that, for each $k \in \ZZ_{\geq 0}$, $\mathbf{1}_{\widetilde{\Omega}_{k}}$ is
$\mathcal{H}_{k,\overline{j}_{k}}$ measurable and, for each
$j \in \ZZ_{\geq 0}$,
$\mathbf{1}_{\widehat{\Omega}_{j}}$ is $\mathcal{H}_{\overline{k}_{j},j}$-measurable.
Next, for each $(k,j) \in \ZZ_{\geq 0} \times \ZZ_{\geq 0}$,
define
\begin{subequations}
\begin{align}
\label{eq:define-widehat-mathbf-f}
    \widehat{\mathbf{f}}_{k+1} & := 
   \displaystyle 
   \mathbf{1}_{\widetilde{\Omega}_{k}} \cdot \left( \frac{\mathbf{x}(k+1,\overline{\jmath}_{k})-
    \mathbf{x}(k,\overline{\jmath}_{k})}{h_{k+1}} \right)
      \\
   \widehat{\mathbf{g}}_{j+1} & : = 
   \mathbf{1}_{\widehat{\Omega}_{j}} \cdot \mathbf{x}(\overline{k}_{j},j+1)  .
   \label{eq:bf-g-j+1-hat-def}
\end{align}
\end{subequations}
For each $k \in \ZZ_{\geq 0}$, 
$\widehat{\mathbf{f}}_{k+1}$ is $\mathcal{H}_{k+1,\overline{\jmath}_{k}}$-measurable,
and, for each $j \in \ZZ_{\geq 0}$, $\widehat{\mathbf{g}}_{j+1}$ is $\mathcal{H}_{\overline{k}_{j}j+1}$-measurable.

For almost every $\omega \in \Omega$,
define
\begin{subequations}
\begin{align}
    \mathbf{f}_{k}(\omega) & : = \mathbb{E} \left[ \,
    \widehat{\mathbf{f}}_{k+1} \, | \, \mathcal{H}_{k,\overline{\jmath}_{k}} \, \right](\omega) 
      \label{eq:bf-f-j-def} \\
        \mathbf{g}_{j}(\omega) & : = \mathbb{E} \left[ \,
    \widehat{\mathbf{g}}_{j+1} \, | \, \mathcal{H}_{\overline{k}_{j},j} \, \right](\omega) . 
    \label{eq:bf-g-j-def}
\end{align}
\end{subequations}
The mapping $\mathbf{x}$ is said to be an 
{\em almost sure asymptotic simulation of 
(\ref{eq:hybrid}) using conditional means} if,
for each $i \in \ZZ_{\geq 0}$, the following properties hold:
\begin{enumerate}
\item
for almost all $\omega \in \Omega_{k-\mbox{\rm \footnotesize complete}}$:
\begin{align}
\label{eq:as-simulation-uCMs-flows}
& \limsup_{k \rightarrow \infty}
\left( \mathbf{x}(k,\overline{\jmath}_{k}) \cap i \Ball , \mathbf{f}_{k} \right)
 \subset \mbox{\rm graph}(F) \cap (C \times \Re^{n});
 \end{align}
 \item
 for almost all $\omega \in \Omega_{j-\mbox{\rm \footnotesize complete}}$:
 \begin{align}
 \label{eq:as-simulation-uCMs-jumps}
\limsup_{j \rightarrow \infty}
\left( \mathbf{x}(\overline{k}_{j},j) \cap i \Ball ,
\mathbf{g}_{j} \right)
 \subset \mbox{\rm graph}(G) \cap (D \times \Re^{n}) .
\end{align}
\end{enumerate}

In the context of the definition
of an asymptotic
simulation of (\ref{eq:hybrid}),
for a fixed sample path $\mathbf{x}(\omega)$,
consider taking
\begin{subequations}
\begin{align}
   & f_{k} = \mathbf{f}_{k}(\omega), \quad \widehat{f}_{k+1} =
    \widehat{\mathbf{f}}_{k+1}(\omega)  \\
    &    g_{j} = \mathbf{g}_{j}(\omega), \quad \widehat{g}_{j+1} =
    \widehat{\mathbf{g}}_{j+1}(\omega) .
\end{align}
\end{subequations}
In this case, conditions on the two
quantities
\begin{subequations}
\begin{align}
 \mathbf{v}_{k+1} & : = \widehat{\mathbf{f}}_{k+1}
  - \mathbf{f}_{k} \\
\mathbf{w}_{j+1} & : = \widehat{\mathbf{g}}_{j+1}
  - \mathbf{g}_{j}
\end{align}
\end{subequations}
and the step sizes
can be given to guarantee that, almost 
every bounded, complete sample path is
an asymptotic simulation of (\ref{eq:hybrid}).
Such conditions are contained in the following
assumptions.

\begin{assumption}
\label{assume:relate-data-stochastic-appendix-flows}
At least one of the following conditions holds:
\begin{enumerate}
    \item There exists $p \in [1,\infty)$ such that
$\sum_{k=0}^{\infty} h_{k+1}^{1+p} < \infty$ and
there exists a continuous, nondecreasing function
$\gamma:\Re_{\geq 0} \rightarrow \Re_{\geq 0}$
such that, almost surely, for all $k \in \ZZ_{\geq 0}$,
\begin{align}
\label{eq:variance-appendix}
  \mathbb{E} \left[ \,   |\mathbf{v}_{k+1}|^{2p} \, | \, \mathcal{H}_{k,\overline{\jmath}_{k}} \,
  \right] \leq \gamma \left(
    |\mathbf{x}(k,\overline{\jmath}_{k})| \right) .
\end{align}
\item
\label{assume:relate-data-stochastic-appendix-2}
    For each $c>0$, $\sum_{k=0}^{\infty} \exp(-c/h_{k+1}) < \infty$
    and there exists a continuous, nondecreasing function
    $\gamma:\Re_{\geq 0} \rightarrow \Re_{\geq 0}$ such that,
    almost surely,
    for all $\theta \in \Re^{n}$ and all $k \in \ZZ_{\geq 0}$,
    \begin{align}
    \label{eq:subgaussian}
    \mathbb{E} \left[ \exp \left( \langle \theta, \mathbf{v}_{k+1} \rangle \right) \, | \, \mathcal{H}_{k,\overline{\jmath}_{k}} \, \right] \leq \exp \left( 
    \vphantom{\frac{}{}}\gamma(|\mathbf{x}(k,\overline{\jmath}_{k})|) |\theta|^{2} \right) .
\end{align}
\end{enumerate}
\end{assumption}

Both conditions in Assumption \ref{assume:relate-data-stochastic-appendix-flows}
are inspired by similar conditions
for differential inclusions
given in \cite[Proposition 1.4]{Benaimetal2005}.
A special case where the second
(sub-Gaussian)
condition holds, or the first condition holds for every $p \in [1,\infty)$, is when
there exists a continuous, nondecreasing function
$\gamma:\Re_{\geq 0} \rightarrow \Re_{\geq 0}$ such that, almost surely, for all $k \in \ZZ_{\geq 0}$, $|\mathbf{v}_{k+1}| \leq \gamma(|\mathbf{x}(k,\overline{\jmath}_{k})|)$. 

\begin{proposition}
\label{prop:appendix-flows}
If Assumption \ref{assume:relate-data-stochastic-appendix-flows}
holds
then,
for each $T>0$ and almost
all $\omega \in \Omega$ such that $\mathbf{x}(\omega)$ is bounded:
   \begin{align}
    \label{eq:Benaim-bound-appendix-flows}
                \lim_{k \rightarrow \infty}
        \sup_{k+1 \leq n \leq m(\tau_{k}+T)}
        \left| \sum_{i=k}^{n-1} h_{i+1} (\widehat{\mathbf{f}}_{i+1}(\omega)
        -\mathbf{f}_{i}(\omega )) \right| = 0 .
    \end{align}
\end{proposition}
\begin{proof}
The proof for the case where the first item
of Assumption \ref{assume:relate-data-stochastic-appendix-flows} holds
is identical to the proof 
of \cite[Lemma 8.3]{TeelSanfeliceGoebel25ARC} (see also
\cite[Lemma 4]{TeelNOLCOS25}), which
relies heavily on \cite[Proposition 1.4]{Benaimetal2005}),
and thus
is omitted.

The proof for the case where the second item
of Assumption \ref{assume:relate-data-stochastic-appendix-flows} holds also relies heavily on \cite[Proposition 1.4]{Benaimetal2005}) in a similar way. The details are as follows. For each $i \in \ZZ_{\geq 0}$,
let $\mathbf{s}_{i}:\Omega \rightarrow
\ZZ_{\geq 0}  \cup \left\{ \infty \right\}$ 
be defined as
\begin{align}
\label{eq:si-def}
\mathbf{s}_{i}: = \inf \left\{ k \in \ZZ_{\geq 0}
 : |\mathbf{x}(k,\overline{\jmath}_{k})| > i \right\}.
\end{align}
Then,
for each $k \in \ZZ_{\geq 0}$,
\begin{align}
    \left\{ \mathbf{s}_{i} \leq k
\right\}
    = \bigcup_{\ell=0}^{k} \left\{ 
    | \mathbf{x}(\ell,\overline{\jmath}_{\ell}) | > i \right\}
    \in \mathcal{H}_{k,\overline{\jmath}_{k}} .
\end{align}
A consequence is that,
for each $k \in \ZZ_{\geq 0}$, the indicator
 function
$$
    \mathbf{1}(k < \mathbf{s}_{i})   : = 
    \begin{cases}
        1 & k < \mathbf{s}_{i} \\
        0 & \mbox{\rm otherwise}
    \end{cases}
$$
is $\mathcal{H}_{k,\overline{\jmath}_{k}}$-measurable. 
For each $k \in \ZZ_{\geq 0}$, define
\begin{subequations}
\begin{align}
    \widehat{\mathbf{f}}_{k+1}^{\mathbf{s}_{i}} & : = 
    \mathbf{1}(k < \mathbf{s}_{i}) 
    \widehat{\mathbf{f}}_{k+1} \\
    \mathbf{f}_{k}^{\mathbf{s}_{i}} & : = \mathbb{E} \left[  \, \widehat{\mathbf{f}}_{k+1}^{\mathbf{s}_{i}} \, | \, \mathcal{H}_{k,\overline{\jmath}_{k}} \,  \right] 
    \\
    \mathbf{u}_{k+1}^{\mathbf{s}_{i}} & : = \widehat{\mathbf{f}}_{k+1}^{\mathbf{s}_{i}} - \mathbf{f}_{k}^{\mathbf{s}_{i}} = 
    \mathbf{1}(k < \mathbf{s}_{i}) \left[\widehat{\mathbf{f}}_{k+1} - \mathbf{f}_{k} \right].
    \end{align}
\end{subequations}
It follows that
\begin{align}
    \mathbb{E} \left[ \, \mathbf{u}_{k+1}^{\mathbf{s}_{i}} \, | \, \mathcal{H}_{k,\overline{\jmath}_{k}} \, \right] 
        & =  \mathbb{E} \left[  \, \widehat{\mathbf{f}}_{k+1}^{\mathbf{s}_{i}} \, | \, \mathcal{H}_{k,\overline{\jmath}_{k}} \,  \right] - \mathbf{f}_{k}^{\mathbf{s}_{i}} = 0 .
\end{align}
Next, note that, for each $v \in \Re$,
\begin{align}
    \exp \left( \vphantom{\frac{}{}}
    \mathbf{1}(k < \mathbf{s}_{i}) v \right)
 = 1 - \mathbf{1}(k < \mathbf{s}_{i})  + 
  \mathbf{1}(k < \mathbf{s}_{i}) \exp(v) .
\end{align}
Thus, using the definition of $\mathbf{s}_{i}$ in 
(\ref{eq:si-def}) and the assumed condition
(\ref{eq:subgaussian}),
    \begin{align}
  &      \mathbb{E} \left[ \, \exp \left( \langle \theta, \mathbf{u}_{k+1}^{\mathbf{s}_{i}} \rangle \right) \, | \, \mathcal{H}_{k,\overline{\jmath}_{k}} \, \right] \nonumber \\
  & \qquad \qquad       = 
         1 - \mathbf{1}(k < \mathbf{s}_{i})  + 
  \mathbf{1}(k < \mathbf{s}_{i}) \mathbb{E} \left[ \exp \left( \langle \theta , 
   \left[\widehat{\mathbf{f}}_{k+1} - \mathbf{f}_{k} \right] \rangle \right) \, | \, \mathcal{H}_{k,\overline{\jmath}_{k}} \, \right] \nonumber \\  
        & \qquad \qquad \leq 1 - \mathbf{1}(k < \mathbf{s}_{i}) + \mathbf{1}(k < \mathbf{s}_{i}) \exp \left( 
    \vphantom{\frac{}{}}\gamma(|\mathbf{x}(k,\overline{\jmath}_{k})|) |\theta|^{2} \right) \nonumber \\
    & \qquad\qquad = \exp \left( 
    \vphantom{\frac{}{}} \mathbf{1}(k < \mathbf{s}_{i}) \gamma(|\mathbf{x}(k,\overline{\jmath}_{k})|) |\theta|^{2} \right) \nonumber \\
        & \qquad \qquad \leq  \exp \left( 
    \vphantom{\frac{}{}} \gamma(i) |\theta|^{2} \right) .
        \label{eq:74b}
        \end{align}
Using
\cite[Proposition 1.4(ii)]{Benaimetal2005},
it follows that \eqref{eq:Benaim-bound-appendix-flows} holds 
almost surely with
$\widehat{\mathbf{f}}_{k+1}^{\mathbf{s}_{i}}$
and $\mathbf{f}_{k}^{\mathbf{s}_{i}}$ in
place of 
$\widehat{\mathbf{f}}_{k+1}$ and
$\mathbf{f}_{k}$. Now suppose $\omega \in \Omega$ is such
that $\mathbf{x}(\omega)$ is bounded.
Then there exists $i^{*} \in \ZZ_{\geq 0}$ such that
$\mathbf{s}_{i}(\omega) = \infty$ for all
$i \geq i^{*}$. Therefore,
\eqref{eq:Benaim-bound-appendix-flows} holds for almost all
$\omega \in \Omega$ such that $\mathbf{x}(\omega)$ is bounded.
\end{proof}

\vspace{.1in}

\begin{assumption}
\label{assume:relate-data-stochastic-appendix-jumps}
At least one of the following conditions holds:
\begin{enumerate}
    \item There exists $p \in [1,\infty)$,
a continuous, nondecreasing function
$\gamma:\Re_{\geq 0} \rightarrow \Re_{\geq 0}$, and a summable sequence
of positive real numbers $\left\{ \rho_{j} \right\}_{j=1}^{\infty}$
such that, almost surely, for all $j \in \ZZ_{\geq 0}$,
\begin{align}
\label{eq:variance-appendix-jumps}
    \mathbb{E} \left[ |\mathbf{w}_{j+1}|^{2p} \, | \, \mathcal{H}_{\overline{k}_{j},j} \right]  \leq \rho_{j+1} \gamma \left(
    |\mathbf{x}(\overline{k}_{j},j)| \right) .
\end{align}
\item
There exists 
a continuous, nondecreasing function
$\gamma:\Re_{\geq 0} \rightarrow \Re_{\geq 0}$ and a sequence of positive real numbers
$\left\{ \rho_{j} \right\}_{j=1}^{\infty}$ converging to zero 
such that, almost surely, for all $j \in \ZZ_{\geq 0}$, 
$|\widehat{\mathbf{g}}_{j+1} - \mathbf{g}_{j}| \leq \rho_{j+1} \gamma(|\mathbf{x}(\overline{k}_{j},j)|)$.
\end{enumerate}
\end{assumption}
\begin{proposition}
\label{prop:appendix-jumps}
If Assumption \ref{assume:relate-data-stochastic-appendix-jumps}
holds
then,
for  almost
all $\omega \in \Omega$ such that $\mathbf{x}(\omega)$ is bounded:
   \begin{align}
    \label{eq:Benaim-bound-appendix-jumps}
        \lim_{j \rightarrow \infty}
        \left|  \widehat{\mathbf{g}}_{j+1}(\omega)
        -\mathbf{g}_{j}(\omega ) \right| = 0 .
    \end{align}
\end{proposition}
\begin{proof} We borrow from the proof of Proposition \ref{prop:appendix-flows}. Suppose the first condition
in Assumption \ref{assume:relate-data-stochastic-appendix-jumps}
holds. For each $i \in \ZZ_{\geq 0}$, let $\mathbf{s}_{i}:\Omega \rightarrow
\ZZ_{\geq 0}  \cup \left\{ \infty \right\}$ 
be defined as
\begin{align}
\label{eq:si-def-jumps}
\mathbf{s}_{i}: = \inf \left\{ j \in \ZZ_{\geq 0}
 : |\mathbf{x}(\overline{k}_{j},j)| > i \right\}.
\end{align}
Then,
for each $j \in \ZZ_{\geq 0}$,
\begin{align}
    \left\{ \mathbf{s}_{i} \leq j
\right\}
    = \bigcup_{\ell=0}^{j} \left\{ 
    | \mathbf{x}(\overline{k}_{\ell}, \ell) | > i \right\}
    \in \mathcal{H}_{\overline{k}_{j},j} .
\end{align}
A consequence is that,
for each $j \in \ZZ_{\geq 0}$, the indicator
 function
    $\mathbf{1}(j < \mathbf{s}_{i})$
is $\mathcal{H}_{\overline{k}_{j},j}$-measurable. 
For each $j \in \ZZ_{\geq 0}$, define
\begin{subequations}
\begin{align}
    \widehat{\mathbf{g}}_{j+1}^{\mathbf{s}_{i}} & : = 
    \mathbf{1}(j < \mathbf{s}_{i}) 
    \widehat{\mathbf{g}}_{j+1} \\
    \mathbf{g}_{j}^{\mathbf{s}_{i}} & : = \mathbb{E} \left[  \, \widehat{\mathbf{g}}_{j+1}^{\mathbf{s}_{i}} \, | \, \mathcal{H}_{\overline{k}_{j},j} \,  \right] 
    \\
    \mathbf{u}_{j+1}^{\mathbf{s}_{i}} & : = \widehat{\mathbf{g}}_{j+1}^{\mathbf{s}_{i}} - \mathbf{g}_{j}^{\mathbf{s}_{i}} = 
    \mathbf{1}(j < \mathbf{s}_{i}) \left[\widehat{\mathbf{g}}_{j+1} - \mathbf{g}_{j} \right].
    \end{align}
\end{subequations}
It follows that, for each $(i,j) \in \ZZ_{\geq 0} \times \ZZ_{\geq 0}$,
\begin{subequations}
\begin{align}
    \mathbb{E} \left[ \, \mathbf{u}_{j+1}^{\mathbf{s}_{i}} \, | \, \mathcal{H}_{\overline{k}_{j},j} \, \right] 
        & =  \mathbb{E} \left[  \, \widehat{\mathbf{g}}_{j+1}^{\mathbf{s}_{i}} \, | \, \mathcal{H}_{\overline{k}_{j},j} \,  \right] - \mathbf{g}_{j}^{\mathbf{s}_{i}} = 0  \\
               \mathbb{E} \left[ \, |\mathbf{u}_{j+1}^{\mathbf{s}_{i}}|^{2p} \, | \, \mathcal{H}_{\overline{k}_{j},j} \, \right] 
        & = 
        \mathbf{1}(j < \mathbf{s}_{i})  \mathbf{w}^{p}(\omega) \\
        & \leq  \mathbf{1}(j < \mathbf{s}_{i}) \rho_{j+1}
        \gamma \left( |\mathbf{x}(\overline{k}_{j},j)| \right)
        \leq \rho_{j+1} \gamma(i) .
\end{align}
\end{subequations}
Then, using properties of conditional expectations, for each $(i,j) \in \ZZ_{\geq 0} \times \ZZ_{\geq 0}$,
\begin{align}
\mathbb{E} \left[ \, |\mathbf{u}_{j+1}^{\mathbf{s}_{i}} |^{2p} \right] \leq \rho_{j+1} \gamma(i) .
\end{align}
Since the sequence $\left\{ \rho_{j} \right\}_{j=1}^{\infty}$ is summable, it follows that, for each $i \in \ZZ_{\geq 0}$,
$\lim_{j \rightarrow \infty} \mathbf{u}_{j+1}^{\mathbf{s}_{i}} = 0$ almost surely (see, for example, the proof of 
\cite[Lemma 7 (Section 12.2)]{FristedtGray97}).
Now suppose $\omega \in \Omega$ is such
that $\mathbf{x}(\omega)$ is bounded.
Then there exists $i^{*} \in \ZZ_{\geq 0}$ such that
$\mathbf{s}_{i}(\omega) = \infty$ for all
$i \geq i^{*}$. Therefore,
\eqref{eq:Benaim-bound-appendix-jumps} holds for almost all
$\omega \in \Omega$ such that $\mathbf{x}(\omega)$ is bounded.

The proof in the case that the second condition in Assumption \ref{assume:relate-data-stochastic-appendix-jumps} holds
is similar but more direct; the details are omitted.
\end{proof}

\vspace{.1in}

The following corollary is the result of combining Theorem \ref{theorem: for-stochastic-simulators} 
with Propositions \ref{prop:appendix-flows} and \ref{prop:appendix-jumps}.

\begin{corollary}
\label{corollary:simulator-dataless-corollary}
Pose the Hybrid Basic Conditions,
let the mapping $\mathbf{x}$ 
be
an almost sure asymptotic simulation of (\ref{eq:hybrid}) using conditional means,
and let the step sizes $\left\{ h_{k} \right\}_{k=1}^{\infty}$ be admissible.
If Assumptions \ref{assume:relate-data-stochastic-appendix-flows} and 
\ref{assume:relate-data-stochastic-appendix-jumps} hold
then almost every bounded, complete sample path of $\mathbf{x}$ is an asymptotic simulation
of \eqref{eq:hybrid} and
the $\omega$-limit set of each such sample path is nonempty,
compact, and weakly invariant and
internally  chain transitive for \eqref{eq:hybrid}.
\end{corollary}

\section{An example: simulated annealing}
\label{section:an example}

\subsection{Problem statement and assumptions}

This section's example is related to 
stochastic-approximation-based
simulated annealing algorithms, which combine stochastic
gradient
descent with random search for global minimization. 
See \cite[Chapter 8]{Spall2003} for a general
discussion of simulated annealing algorithms and 
their connections 
to the stochastic approximation literature,
which
includes the detailed analysis of
stochastic-approximation-based simulated annealing given in \cite{Kushner87}. 
The approach taken in this section is to consider an algorithm that can be viewed as a stochastic approximation of a hybrid system.

The scalar-valued function to minimize is denoted
$\Suzie$. It is defined on an open set $\mathcal{O}
\subset \Re^{n}$ and is assumed to be
nonpathological; see \cite{Valadier88} or
\cite[Definition 3]{BacciottiCeragioli06Automatica}.
In particular, $\Suzie$ is locally Lipschitz and
satisfies a convenient chain rule when composed
with an absolutely continuous function, e.g.,
a solution to a differential inclusion. The set
over which $\Suzie$ is to be minimized is
denoted $S \subset \mathcal{O}$. The symbol
$\SuzieS$ is used to denote the restriction
of $\Suzie$ to the set $S$. The set $S$ is assumed to be closed and prox-regular in the sense of
 \cite[Definition 1.1]{PRock2000}; equivalently, according to
 \cite[Theorem 1.3(a)$\Leftrightarrow$(k)]{PRock2000},
 $S$ is closed 
and the closest point projection, i.e., the mapping that projects points in $\Re^n$ onto $S$,
denoted $P_{S}$, is single valued near $S$. An additional assumption
 imposed on $S$ and $\Suzie$ pertains to the so-called
 {\em regular values} of $\SuzieS$.
 A point $x \in S$ is a {\em critical point}
of $\SuzieS$  if $0 \in \partial \Suzie(x)$, 
where $\partial \Suzie:\mathcal{O} \rightrightarrows \Re^{n}$ is the Clarke generalized gradient of
$\Suzie$. Since $\Suzie$ is
assumed to be locally Lipschitz on $\mathcal{O}$, 
$\partial \Suzie$ is well defined with nonempty, convex values \cite[Definition 10.3]{Clarke2013}, locally bounded
\cite[Proposition 10.5]{Clarke2013}, and
outer semicontinuous 
\cite[Proposition 10.10]{Clarke2013}. 
Thus, it satisfies the conditions imposed on a flow map in the Hybrid Basic Conditions.
 The set
of critical points of $\SuzieS$ is denoted
$\mbox{\rm Cr}(\left. \Suzie \right|_{S})$.
A value
$r \in \Suzie(S)$ is a regular value of $\left. \Suzie \right|_{S}$ if 
$r \notin \Theta(\mbox{\rm Cr}(\left. \Suzie \right|_{S}))$. It is assumed that the regular values
of $\SuzieS$ are dense in $\Suzie(S)$. 
By the Morse-Sard theorem, this property holds
when $\Suzie$ is $n$ times continuously differentiable \cite{Morse39AM}, \cite{Sard42}.

One situation covered by the assumptions is when
$S=\Re^{m} \times \mathcal{Q}$, where $\mathcal{Q}
\subset \Re$ is a finite set, and
$\mathcal{O} =\Re^{m} \times (\mathcal{Q}+\varepsilon \Ball^{\circ})$, where $\varepsilon>0$ small enough
so that the projection of a point in $\mathcal{Q}+\varepsilon \Ball^{\circ}$ onto $\mathcal{Q}$ is unique. In this case,
$\Suzie:\mathcal{O} \rightarrow \Re$ is assumed
to have the form
$\Suzie(z,r)=\Suzie_{P_{Q}(r)}$ for all $(z,r) \in
\mathcal{O}$, where
$\Suzie_{q}: \Re^{m} \rightarrow \Re$ is nonpathological
for each $q \in \mathcal{Q}$
with regular values that are dense in $\Theta_{q}(\Re^{m})$.

\subsection{The underlying hybrid system}

The stochastic minimization algorithm considered
is constructed as a stochastic approximation of
a hybrid system. That hybrid system's dynamics is? the combination of a gradient descent differential
inclusion and a hybrid automaton that enforces an
average dwell-time constraint on the number of jumps
 experienced by a solution in a given amount of flow time \cite{Hespanha99a}.
 In the hybrid system considered, the jumps do not change the state that is executing gradient descent. In the stochastic approximation algorithm, the jumps can change the state that is executing (stochastic, approximate)
gradient descent, but to an extent that
fades to zero with time almost surely.
Such jumps provide
an opportunity for the algorithm to
randomly search for a global minimum of $\SuzieS$.

Letting $N \in \ZZ_{\geq 1}$ and $\delta>0$, the
hybrid system's data is
\begin{subequations}
\label{eq:simulated-annealing-data}
\begin{align}
  &  C:=S \times [0,N], \quad F(y,\tau) := - \partial \Suzie(y) \times [0,\delta]
    \label{eq:simulated-annealing-data-CT} \\
  & D:=S \times [1,N], \quad
  G(y,\tau):=(y,\tau-1) .
  \label{eq:simulated-annealing-data-DT}
\end{align}
\end{subequations}
The hybrid system (\ref{eq:hybrid}) with the
data (\ref{eq:simulated-annealing-data}) is referred to as the underlying
hybrid system throughout this example.
It includes the constrained differential inclusion
$y \in S$, $\dot{y} \in -\partial \Suzie(y)$ together with jumps $y^{+}=y$, which
do not change the state $y$. The jumps satisfy an average dwell-time
condition due to the presence of the hybrid automaton 
\begin{subequations}
    \begin{align}
        \tau \in [0,N] & \qquad \dot{\tau} \in [0,\delta] \\
        \tau \in [1,N] & \qquad \tau^{+} = \tau -1 .
    \end{align}
\end{subequations}
The result \cite[Appendix, Proposition 1.1]{CaiTeelGoebel08TAC},
discusses this hybrid automaton and shows that a
hybrid time domain $E$ can be generated by this automaton
if and only if, for every pair of times $(s,i),(t,j) \in E$ with $t+j \geq s+i$,
$j-i \leq \delta (t-s) +N$.
Because the jumps
satisfy this average dwell-time condition, solutions
of the underlying hybrid system 
that are complete 
are necessarily complete in the $t$ direction.
Since $\Suzie$ is nonpathological,
every bounded, complete solution
of the underlying hybrid system
converges to the set $\mbox{\rm Cr}(\SuzieS) \times [0,N]$.
Since $\Suzie$ is nonpathological and the
regular values of $\SuzieS$ are dense
in $\Suzie(S)$, it can be shown using
the ideas used in 
\cite[Example 7.4]{TeelSanfeliceGoebel25ARC}
that if the set $K \subset \Re^{n+1}$ is compact and internally chain
transitive for the underlying hybrid system  then
$K \subset \mbox{\rm Cr}(\SuzieS) \times [0,N]$. The latter is relevant
for characterizing the omega-limit sets
of asymptotic solutions of the underlying
hybrid system, such as those that arise
from stochastic approximation, which
is discussed next.

\subsection{Stochastic approximation}

A prescription for generating
an almost sure asymptotic simulation of 
the underlying hybrid system
 is now given. Attention is paid
to making sure that
the conditions of Corollary \ref{corollary:simulator-dataless-corollary} are satisfied. In that
case, and using the assertion at the end of
the previous subsection,
it follows that the omega-limit set of
almost every complete,
bounded sample path of the asymptotic simulation
is contained in
$\mbox{\rm Cr}(\SuzieS) \times [0,N]$. The discussion
is broadened in the next subsection to discuss how to make it
more likely that the omega-limit set of
almost every complete, bounded sample path
is contained in $\mbox{\rm M}(\SuzieS) \times [0,N]$
where $\mbox{\rm M}(\SuzieS) \subset \mbox{\rm Cr}(\SuzieS)$ denotes the set of minimizers of $\SuzieS$.

Suppose
the stochastic collection of
hybrid sequences $\mathbf{x}=(\mathbf{y},\tau)$
is such that, for almost every
$\omega \in \Omega$, if
    $(k,j),(k,j+1) \in \dom  \mathbf{x}(\omega)$
then
\begin{subequations}
\begin{align}
\label{eq:simulated-annealing-jumps}
    \mathbf{y}(k,j+1) & \in \mbox{\rm argmin}_{p \in \left\{ \mathbf{y}(k,j),
    P_{S}(\mathbf{y}(k,j) + \ell_{j+1}
\mathbf{z}_{j+1} )   \right\}} \Suzie(p)  \\
\tau(k,j+1) & = \tau(k,j) - 1 
\end{align}
\end{subequations}
where the sequence of random variables $\left\{ \mathbf{z}_{j} \right\}_{j=1}^{\infty}$ take values in $\Re^{n}$
and where each element of the sequence of real
numbers
$\left\{ \mathbf{\ell}_{j} \right\}_{j=1}^{\infty}$
is positive.
In order to guarantee that every relatively
open subset of $S$ is explored with positive probability
during jumps, it is desired that,
for every open set $U$ such that
$S \cap U$ is nonempty, 
the values $P_{S}(\mathbf{y}(k,j) + \ell_{j+1}
\mathbf{z}_{j+1})$ belong to the set $S \cap U$ with positive
probability for every $j \in \ZZ_{\geq 0}$ snd every
value of $\mathbf{y}(k,j)$. To assure
this property holds it is enough to assume that,
for each open set $\mathcal{O}
\subset \Re^{n}$ and each $j \in \ZZ_{\geq 1}$,
\begin{align}
\label{eq:101}
\mathbb{P}(\mathbf{z}_{j} \in \mathcal{O})>0 .
\end{align}
To see how this condition is effective, let $\mathcal{N}$
be a neighborhood of $S$ such that the restriction
of $P_{S}$ to
$\mathcal{N}$ is a continuous
function. The existence of $\mathcal{N}$ follows from
$P_{S}$ being single valued near $S$, which is a consequence of $S$ being prox-regular, together
with the fact that $P_{S}$ is everywhere outer semicontinuous and locally bounded
in general \cite[Example 5.23(a)]{RockafellarWets98}
so that it is continuous where it is single valued
\cite[Corollary 5.20]{RockafellarWets98}. Let $U$ be an open set such that $S \cap U$ is nonempty. Define $\mathcal{O}_{U}:=P_{S}^{-1}(S \cap U) \cap \mathcal{N}$. Since
the restriction of $P_{S}$ to $\mathcal{N}$ is continuous, $\mathcal{O}_{U}$ is open; moreover $P_{S}(\mathcal{O}_{U}) = S \cap U$. 
Taking $\mathcal{O}:= \ell_{j+1}^{-1} \left(-\mathbf{y}(k,j) +
\mathcal{O}_{U}\right)$, it follows from (\ref{eq:101}) that, as desired,
\begin{align}
   0< \mathbb{P}(\mathbf{z}_{j+1} \in \mathcal{O})
   & =  \mathbb{P}\left(\mathbf{y}(k,j) + \ell_{j+1} \mathbf{z}_{j+1} \in \mathcal{O}_{U}\right) \\ \nonumber
   & = \mathbb{P}\left( \vphantom{\frac{}{}} P_{S}(\mathbf{y}\left( k,j) + \ell_{j+1} \mathbf{z}_{j+1}\right) \in S \cap U \right) .
\end{align}

To satisfy the conditions of Corollary \ref{corollary:simulator-dataless-corollary},
 the sequence of positive real numbers
$\left\{ \mathbf{\ell}_{j} \right\}_{j=1}^{\infty}$ is chosen
such that,
for almost all $\omega \in \Omega$,
\begin{align}
\label{eq:limit-with-z}
    \lim_{j \rightarrow \infty} \ell_{j} \mathbf{z}_{j} = 0 .
\end{align}
This condition is not ruled out by (\ref{eq:101}). Indeed, 
such a sequence can be found due to the Borel-Cantelli lemma. To see this fact, note that,
for each $j \in \ZZ_{\geq 1}$, there exists $c_{j} \geq 1$ be such that
\begin{align}
\label{eq:for-BC}
    \mathbb{P}(|\mathbf{z}_{j}| > c_{j})
    \leq 2^{-j}  .
\end{align}
For each $j \in \ZZ_{\geq 1}$,
define $\ell_{j}:=1/(j c_{j})$.
It follows from \eqref{eq:for-BC} that
\begin{align}
    \sum_{j=1}^{\infty} \mathbb{P}(|\mathbf{z}_{j}| > c_{j}) < \infty
\end{align}
and thus, from the Borel-Cantelli lemma, that
for almost every $\omega \in \Omega$, there
exists $j^{*} \in \ZZ_{\geq 1}$ such that, for
all $j \in \ZZ_{\geq j^{*}}$,
    $|\mathbf{z}_{j}(\omega)| \leq c_{j}$.
In turn,
$\ell_{j} |\mathbf{z}_{j}(\omega)|
    \leq 1/j$,  
which establishes (\ref{eq:limit-with-z}).

The simplest way to guarantee that the remaining
conditions of Corollary \ref{corollary:simulator-dataless-corollary} hold is to assume that
$(k,j),(k+1,j) \in \dom \mathbf{x}(\omega)$
implies
\begin{subequations}
\label{eq:approximate-flows-example}
\begin{align}
\frac{
\mathbf{y}(k+1,j)-\mathbf{y}(k,j)}{h_{k+1}}
& \in -\partial \Suzie(\mathbf{y}(k,j)) \\
\frac{\tau(k+1,j) - \tau(k,j)}{h_{k+1}} & \in [0,\delta] .
\end{align}
\end{subequations}
However, stochastic algorithms where the condition
expectation of the quantities on the left-hand
side of (\ref{eq:approximate-flows-example}) belong to the sets on
the right-hand side of (\ref{eq:approximate-flows-example}) as $k$
tends to infinity are also
possible. 

It then follows from Corollary
\ref{corollary:simulator-dataless-corollary} that, for almost
every $\omega$ for which $\mathbf{x}(\omega)$
is bounded and complete,
$\mbox{\rm omega}(\mathbf{x}(\omega)) \subset
\mbox{\rm Cr}(\SuzieS) \times [0,N]$.

\subsection{Toward global minimization}

The interest in focusing on bounded sample
paths that are complete in the $j$
direction is that those sample paths attempt
to repeatedly execute a global, random search to locate
a global minimum, if one exists, as described above.
However, this global search effect
dies off with $j$ almost surely
under the condition (\ref{eq:limit-with-z}).
Still, there is interest in sequences
$\left\{ \ell_{j} \right\}_{j=1}^{\infty}$
that results in this global search
dying off slowly, to promote the likelihood
of finding a global minimum. For example,
suppose 
$\left\{ \widetilde{\ell}_{j} \right\}_{j=1}^{\infty}$ is a sequence, perhaps converging to zero, such that
for each $M>0$,
$|\widetilde{\ell}_{j} \mathbf{z}_{j}| > M$ happens infinitely many times
almost surely. 
Clearly, such a sequence  does not satisfy (\ref{eq:limit-with-z})
almost surely. On the other hand, if the sequence $\left\{ \ell_{j} \right\}_{j=1}^{\infty}$ does satisfy (\ref{eq:limit-with-z})
almost surely then so does the sequence
\begin{align}
\label{eq:where-the-parameter-L-shows-up}
\widehat{\ell}_{j,\beta}: = \min \left\{ \widetilde{\ell}_{j}, \beta \ell_{j} \right\}
\end{align}
for each $\beta>0$. In turn, the larger that $\beta>0$ is, the longer 
$\widehat{\ell}_{j,\beta} =\widetilde{\ell}_{j}$.
In this case, beyond $\mbox{\rm omega}(\mathbf{x}(\omega)) \subset
\mbox{\rm Cr}(\Suzie) \times [0,N]$ as indicated
in the previous subsection, it becomes more likely
that $\mbox{\rm omega}(\mathbf{x}(\omega))
\subset \mbox{\rm M}(\SuzieS) \times [0,N]$.

To guarantee convergence to the set of
global minima almost surely,
there is interest in
asymptotic simulations that do not
satisfy (\ref{eq:limit-with-z}). Such
asymptotic simulations would retain
a stochastic nature in the limit as time grows unbounded, similar to the behavior
of a stochastic
hybrid inclusion with randomness in the jumps,
as considered in \cite{TeelARC13} for example.
Such generalizations go beyond what has been
provided in this paper.

\section{Comments on boundedness}

All of the results above pertain to simulations
that are bounded. A simulation of a well-behaved
dynamical system may not be bounded. For example,
consider the forward-Euler approximation of the 
differential equation
$\dot{z} = - z^{3}$, which has the origin globally
asymptotically stable,
using the admissible step sizes $h_{k} = 1/k^{0.75}$
for $k \in \ZZ_{\geq 0}$. That is, consider the iteration
    $z_{k+1} = z_{k} - h_{k+1} z_{k}^{3}$.
It is straightforward to establish that if
$z_{0}^{2} \geq 3$ then
$h_{k+1} z_{k}^{2} \geq 3$ for all $k \in \ZZ_{\geq 0}$ and thus $z_{k}$ grows unbounded.

This phenomenon can be mitigated using resets
within the framework of hybrid systems. For example,
the differential equation $\dot{z}=-z^{3}$ can be
embedded in a hybrid system employing an automaton
with state $\tau$ that enables resets of the state $z$. For example, consider the ideal hybrid system with state
$x:=(z,\tau) \in \Re^{2}$, and data
\begin{subequations}
\label{eq:simple-system-data}
    \begin{align}
 & C:= \Re \times [0,1] \qquad f(x) = \left[ \begin{array}{c} - z^{3} \\ 1 \end{array} \right] \\
 &  D:= \Re \times [1,\infty) \qquad 
 g(x) = \left[ \begin{array}{c} z \\ 0 \end{array} \right]
    \end{align}
\end{subequations}
and a simulation model that uses a forward-Euler approximation for flows and jumps that, in place of
$g$, use the jump map
\begin{align}
    \widehat{g}_{c}(x) = \left[ \begin{array}{c} \mbox{\rm sgn}(z) \min \left\{ |z| , c \right\} \\ 0 \end{array} \right] 
\end{align}
where $c>0$. It is not difficult to establish that complete simulations of the type described are bounded and are asymptotic simulations of the ideal hybrid system with data in (\ref{eq:simple-system-data}), which also has the set $\left\{ 0 \right\} \times [0,1]$ globally asymptotically stable.

More elaborate situations where
resets can be invoked to guarantee boundedness while
also guaranteeing that simulations are asymptotic simulations of the (possibly hybrid) system without
resets are discussed in
\cite[Section V.A]{GoebelTeel2026CDC}.

\vspace{.25in}

\noindent
{\bf Acknowledgments}
\vspace{.1in}

Andrew R. Teel's research supported in part by
ARO grant W911NF-26-1-0001 and AFOSR grant
FA9550-25-1-0186.

\bibliographystyle{elsarticle-num}
\bibliography{bibfile}

\end{document}